\pdfoutput=1
\documentclass[12pt]{amsart}
\usepackage[english]{babel}
\usepackage[usenames,dvipsnames]{xcolor}
\usepackage{amsmath,amssymb,amsthm,tikz,url,braket}
\usepackage{mathtools}
\usepackage{fullpage}
\usepackage{enumitem}
\usepackage{bm}

\usepackage[scr]{rsfso}
\newcommand{\HS}{{\mathscr{H}}}

\usepackage[protrusion=true,expansion=true]{microtype}
\usepackage{hyperref}
\hypersetup{
    pdftitle={Vertex Operators, Solvable Lattice Models and Metaplectic Whittaker Functions},
    pdfauthor={Ben Brubaker, Valentin Buciumas, Daniel Bump, Henrik P. A. Gustafsson},
    linktocpage=true,           % color toc pages, not text
    colorlinks=true,            % false: boxed links; true: colored links
    linkcolor=blue!75!black,    % color of internal links
    citecolor=green!50!black,   % color of links to bibliography
    filecolor=magenta,          % color of file links
    urlcolor=blue!75!black      % color of external links
}

\newtheorem{maintheorem}{Theorem}

\newtheorem{theorem}{Theorem}[section]
\newtheorem{lemma}[theorem]{Lemma}
\newtheorem{proposition}[theorem]{Proposition}
\newtheorem{assumption}[theorem]{Assumption}
\newtheorem{corollary}[theorem]{Corollary}

\theoremstyle{definition}
\newtheorem{remark}[theorem]{Remark}
\newtheorem{example}{Example}            % for separate 1,2,3,... numbering

\numberwithin{equation}{section}

\renewcommand{\leq}{\leqslant}
\renewcommand{\geq}{\geqslant}
\newcommand{\GL}{\operatorname{GL}}
\newcommand{\Sp}{\operatorname{Sp}}

\newcommand{\Z}{\mathbb{Z}}
\newcommand{\C}{\mathbb{C}}

\newcommand{\cmplx}{\mathbb{C}}

\newcommand{\Uq}{U_q(\widehat{\mathfrak{sl}}_n)}
\newcommand{\aH}{\widehat{H}_N}
\newcommand{\z}{\mathbf{z}}
\newcommand{\UqD}{U^{\bm{\alpha}}_q(\widehat{\mathfrak{sl}}_n)}
\newcommand{\End}{\operatorname{End}}
\newcommand{\wt}{\operatorname{wt}}

\renewcommand{\Im}{\operatorname{Im}}
\newcommand{\Ker}{\operatorname{Ker}}

\newcommand{\F}{\mathfrak{F}}

\newcounter{todocounter}

\newcommand{\gammaice}[6]{
\begin{tikzpicture}
\coordinate (a) at (-.75, 0);
\coordinate (b) at (0, .75);
\coordinate (c) at (.75, 0);
\coordinate (d) at (0, -.75);
\coordinate (aa) at (-.75,.5);
\coordinate (cc) at (.75,.5);
\coordinate (ee) at (0,0);
\draw[thick] (a)--(c);
\draw[thick] (b)--(d);
\draw[fill=white, thick] (a) circle (.25);
\draw[fill=white, thick] (b) circle (.25);
\draw[fill=white, thick] (c) circle (.25);
\draw[fill=white, thick] (d) circle (.25);
\node at (0,1) { };
\node at (a) {$#1$};
\node at (b) {$#2$};
\node at (c) {$#3$};
\node at (d) {$#4$};
\node at (aa) {$#5$};
\node at (cc) {$#6$};
\draw[fill=black] (ee) circle (.075);
\end{tikzpicture}}

\makeatletter
\@namedef{subjclassname@2020}{%
  \textup{2020} Mathematics Subject Classification}
\makeatother

\begin{document}
\title{Vertex Operators, Solvable Lattice Models and Metaplectic Whittaker Functions}
\author{Ben Brubaker}
\address{School of Mathematics, University of Minnesota, Minneapolis, MN 55455}
\email{brubaker@math.umn.edu}
\author{Valentin Buciumas}
\address{School of Mathematics and Physics, 
The University of Queensland, 
St. Lucia, QLD, 4072, 
Australia}
\email{valentin.buciumas@gmail.com}
\author{Daniel Bump}
\address{Department of Mathematics, Stanford University, Stanford, CA 94305-2125}
\email{bump@math.stanford.edu}
\author{Henrik P. A. Gustafsson}
\address{\hspace{-\parindent}\textnormal{Until September 15, 2019:} \newline
  \indent Department of Mathematics, Stanford University, Stanford, CA 94305-2125. \newline 
  \textnormal{Since September 16, 2019 (i.e.\ after journal submission):} \newline
  \indent School of Mathematics, Institute for Advanced Study, Princeton, NJ~08540. \newline 
  \indent Department of Mathematics, Rutgers University, Piscataway, NJ~08854. \newline 
  \indent Department of Mathematical Sciences, University of Gothenburg and Chalmers University of Technology, SE-412~96 Gothenburg, Sweden.}
\email{gustafsson@ias.edu}

\maketitle

\begin{abstract}
  We show that spherical Whittaker functions on an $n$-fold cover of the
  general linear group arise naturally from the quantum Fock space
  representation of $U_q(\widehat{\mathfrak{sl}}(n))$ introduced by Kashiwara,
  Miwa and Stern (KMS). We arrive at this connection by reconsidering the
  solvable lattice models known as ``metaplectic ice'' whose partition
  functions are metaplectic Whittaker functions. First, we show that a 
  certain Hecke action on metaplectic Whittaker coinvariants agrees (up to
  twisting) with a Hecke action of Ginzburg, Reshetikhin, and Vasserot arising
  in quantum affine Schur-Weyl duality. This allows us to expand the
  framework of KMS by Drinfeld twisting to introduce Gauss sums into the
  quantum wedge, which are necessary for connections to metaplectic forms.
  Our main theorem interprets the row transfer matrices of this ice model as
  ``half'' vertex operators on quantum Fock space that intertwine with the
  action of $U_q(\widehat{\mathfrak{sl}}(n))$.
  	
  In the process, we introduce new symmetric functions
  termed \textit{metaplectic symmetric functions} and explain how they are
  related to Whittaker functions on an $n$-fold metaplectic cover of
  $\GL_r$. These resemble \textit{LLT polynomials} or \textit{ribbon symmetric
  functions} introduced by Lascoux, Leclerc and Thibon, and in fact the
  metaplectic symmetric functions are (up to twisting) specializations
  of \textit{supersymmetric LLT polynomials} defined by Lam. Indeed Lam 
  constructed families of symmetric functions from Heisenberg algebra
  actions on the Fock space commuting with the $U_q(\widehat{\mathfrak{sl}}(n))$-action.
  The Heisenberg algebra is independent of Drinfeld twisting of the quantum group.
  We explain that half vertex operators agree with Lam's construction and this interpretation 
  allows for many new identities for metaplectic symmetric and Whittaker functions, including Cauchy identities.
  While both metaplectic symmetric
  functions and LLT polynomials can be related to vertex operators on the
  quantum Fock space, only metaplectic symmetric functions are connected to
  solvable lattice models.
\end{abstract}

\section{Introduction}

This paper concerns two mechanisms by which the quantum groups $U_q
(\hat{\mathfrak{g}})$, for $\mathfrak{g}$ a simple Lie algebra or
superalgebra, produce families of special functions with a suite of
interesting properties including functional equations, branching rules and
unexpected algebraic relations. The first mechanism uses solvable lattice models associated to
finite-dimensional modules of $U_q (\hat{\mathfrak{g}})$. The second mechanism uses
actions of Heisenberg and Clifford algebras on a fermionic Fock space, as in the boson-fermion 
correspondence~{\cite{DateJimboMiwaSolitons,JimboMiwa,KacRaina}} with connections to soliton theory. 
We will use these two points of view to provide new insight into the theory of metaplectic Whittaker functions
for the general linear group and relate them to LLT polynomials. To begin, we explain these two approaches to special functions from quantum affine
groups in more detail.

If $V$ is a
finite-dimensional module of $\mathfrak{g}$, then since
$U_q(\hat{\mathfrak{g}})$ is the quantization of a central extension of
$\mathfrak{g} \otimes \mathbb{C} [t, t^{- 1}]$, we obtain a family of
{\textit{evaluation modules}} $V_z$ ($z \in \mathbb{C}^{\times}$) in which $t$
is specialized to the value $z$. Using quasitriangularity, we have
$U_q(\hat{\mathfrak{g}})$-homomorphisms (almost always isomorphisms) $V_{z_1} \otimes V_{z_2} \longrightarrow V_{z_2} \otimes
V_{z_1}$ dictated by an $R$-matrix $R(z_1, z_2)$ satisfying
\begin{equation}
  \label{paramybe}
   R_{12}(z_1,z_2) R_{13}(z_1,z_3) R_{23}(z_2, z_3)  = R_{23}(z_2, z_3)
   R_{13}(z_1,z_3) R_{12}(z_1,z_2),
\end{equation}
in $\text{End}(V_{z_1} \otimes V_{z_2} \otimes V_{z_3})$. This
identity is called the
{\textit{parametrized (quantum) Yang-Baxter equation}} with parameter group
$\mathbb{C}^{\times}$. These are endomorphisms of
$V_{z_1}\otimes V_{z_2}\otimes V_{z_3}$ and the subscripts $R_{ij}$ mean
that the matrix $R$ is applied to the $i$-th and $j$-th component of the
threefold tensor product.

Given any such matrix $R$, we may ask for a matrix $T(z)$ satisfying the ``RTT'' relation:
\begin{equation}
  \label{rttybe}
  R(z_1,z_2) T(z_1) T(z_2) = T(z_2) T(z_1) R(z_1,z_2).
\end{equation}
Typically, the matrix $T(z)$ arises as an endomorphism of $V_{z}\otimes W$ where
$W$ is a fixed object in the category of
$U_q(\hat{\mathfrak{g}})$-modules. If $W=V_{z_3}$ and $T(z)=R(z,z_3)$
then (\ref{rttybe}) is
equivalent to (\ref{paramybe}). For arbitrary $W$, the
existence of a $T(z)\in\End(V_z,W)$ making (\ref{rttybe}) true
follows from quasitriangularity. 

Solutions $T(z)$ to (\ref{rttybe}) may also arise as a ``row transfer matrix'' in a solvable lattice model.
For example in the case of the field-free six-vertex model, 
Baxter~\cite{Baxter} demonstrates that the resulting partition function is a symmetric function in the $z_i$
when its Boltzmann weights satisfy the Yang-Baxter equation (\ref{paramybe}). The underlying algebra was explained
by Kulish and Reshetikhin~\cite{KulishReshetikhin}, Sklyanin~\cite{Sklyanin}, 
Drinfeld~\cite{Drinfeld} and Jimbo~\cite{JimboHecke} and the
relevant quantum group associated to the $R$-matrix is $U_q(\widehat{\mathfrak{sl}}_2)$.
To connect to the presentation of the Yang-Baxter equation in the previous paragraph,
each edge in the planar lattice model is associated to a two-dimensional evaluation
module~$V_z$ and the local Boltzmann weights encode endomorphisms among them.

In \cite{BrubakerBuciumasBump}, the first three authors considered examples of 
solvable square lattice models connected to $R$-matrices of evaluation modules for $U_q(\widehat{\mathfrak{gl}}(n|1))$.
In these examples (Theorem~1 in \cite{BrubakerBuciumasBump}),
the matrices $T(z)$ in (\ref{rttybe}) do not quite fit the standard paradigm.
Each vertex in the square lattice receives a Boltzmann weight reflecting the action of $T(z)$ on basis elements
determined by adjacent edges; while the horizontal edges may be identified with evaluation modules for 
$U_q(\widehat{\mathfrak{gl}}(n|1))$, the vertical edges represent a two-dimensional vector space with no known
algebraic connection to this quantum group. The problem is that
we are not aware of any candidate for a two dimensional module $M$ of
$U_q(\widehat{\mathfrak{gl}}(n|1))$ that would explain the matrix $T(z)$.
In other words, we would like there to exist an $M$ such that
the $R$-matrix for $V_z \otimes M$ is the matrix for a set of Boltzmann weights used in this paper. See
Table~\ref{tab:mweights} in Section~\ref{deltaicefockspace}. If no such two-dimensional module exists,
then we have an example of a parametrized Yang-Baxter equation that
is not explained by quasitriangularity. This is an important unresolved
question.

Nevertheless in \cite{BrubakerBuciumasBump} the partition function of the model is shown to be solvable and equal the spherical
Whittaker function on an $n$-fold metaplectic cover of the general linear
group; this will be our primary example of the sort of special functions mentioned
at the outset. 

As we will explain in the present paper, an alternate algebraic interpretation is possible if we take
$T(z)$ to be the row transfer matrix of an infinite grid; then a
module explaining $T(z)$ does appear, and it is the quantum fermionic Fock space defined by Kashiwara, Miwa and
Stern~{\cite{KMS}}.
Thus instead of trying to interpret the vertically oriented edges (which
can have only two states $\pm$) as 2-dimensional modules in the category,
there is an alternative approach -- one that takes us from the solvable
lattice model point of view to the Heisenberg algebra point of view, our
second mechanism for producing special functions.
In this approach, an infinite sequence of vertical edges in a fixed row of our square
lattice model parametrizes a vector in the
\textit{fermionic Fock space} $\mathfrak{F}$. The row transfer matrix for the
model then becomes an operator
$T(z) : \mathfrak{F} \longrightarrow \mathfrak{F}$ with
$z \in \mathbb{C}^{\times}$ a fixed parameter. The Yang-Baxter equation
implies that the operators $T(z_i)$ and $T(z_j)$ commute for any $i$ and $j$.

In these examples, the space $\mathfrak{F}$ is not the usual fermionic Fock
space described (for example) in~{\cite{KacRaina}}. Instead it is the quantum Fock
space $\mathfrak{F} = \mathfrak{F}_q^{(n)}$ of~{\cite{KMS}}, which is a module for $U_q
(\widehat{\mathfrak{sl}}_n)$. It will be a consequence of our main
theorem that the operators $T(z)$ are $U_q
(\widehat{\mathfrak{sl}}_n)$-module homomorphisms. It also gives a
proof, independent of the Yang-Baxter equation, that the operators
$T(z)$ commute. Thus our method here succeeds in providing a quantum group
interpretations to these problematic vertical edges in the metaplectic ice model.

We may picture the Fock space $\mathfrak{F}$ as follows. Similar to the way Dirac described the electron
sea, consider a quantum particle with an infinite number of states, one
for each energy level, and a system of such particles obeying the Pauli exclusion principle where the lowest energy levels are all occupied and the
highest levels are unoccupied. Thus if $u_i$ represents the particle in a
state with energy $i$, then a basis of $\mathfrak{F}$ consists of vectors
\begin{equation}
  \label{semimon} u_{\mathbf{i}}:=u_{i_m} \wedge u_{i_{m - 1}} \wedge \cdots
\end{equation}
where $\mathbf{i}=(i_m,i_{m-1},\cdots)$ is a strictly
decreasing sequence such that $i_k=k$ for $k\ll 0$.
Here $i_m, i_{m - 1}, \cdots$ are the energy levels of occupied states; we
may arrange that $i_m > i_{m - 1} > \cdots$. The condition that $i_k = k$ for $k \ll 0$ 
ensures that all sufficiently low energy levels are occupied. The totality of
such states for fixed $m$ is the level $m$ space $\mathfrak{F}_m$ and
$\mathfrak{F} = \bigoplus_m \mathfrak{F}_m$.

If $m$ is given, we may parametrize the semi-infinite monomials
(\ref{semimon}) by partitions: if $\lambda = (\lambda_1, \lambda_2, \cdots)$
is a partition, then we may take $i_m = m+\lambda_1$, $i_{m - 1} =
m-1+\lambda_2$, etc. This gives a bijection between partitions and basis
vectors of $\mathfrak{F}_m$. Thus we write
\begin{equation}
   \label{partitionnot}
   |\lambda\rangle = |\lambda; m\rangle := 
   u_{m+\lambda_1} \wedge u_{m-1+\lambda_2} \wedge \cdots\;.
\end{equation}

In Section~\ref{sec:fermionicfock} we review and generalize
the construction of the quantum Fock space of~\cite{KMS}. 
In
Theorem~\ref{thm:heckeactionsagree} we relate the
Hecke action that underlies this construction (due to Ginzburg, Reshetikhin
and Vasserot~\cite{GRV}) to another Hecke action, which was motivated by
the action in~\cite{BBBF} on Whittaker coinvariants. Because of this we are able to easily build an action of the Hecke algebra modified by a Drinfeld twist.
This generalization allows us to introduce
Gauss sums to the anticommutation rule for vectors in the Fock space.
This twisting is needed for the application to metaplectic Whittaker
functions, but is more general than what is needed for this application
and so may be of importance for other purposes.

To connect the quantum Fock space to solvable lattice models, we
introduce a grid, infinite in width, whose boundary edges
encode vectors in the Fock space.
The column edges of the solvable lattice model
in \cite{BrubakerBuciumasBump} are likewise indexed by partitions, so may be
viewed as semi-infinite wedge products according to the above correspondence.
This point of view will be detailed further in Section~\ref{deltaicefockspace}.

Let us now explain our main theorem which considers two solvable lattice models connected to $U_q(\widehat{\mathfrak{gl}}(n|1))$ $R$-matrices called \emph{Gamma ice} and \emph{Delta ice} detailed in Section~\ref{sec:proofmain}, and their row transfer matrices $T_\Gamma(z)$ and $T_\Delta(z)$. 

In addition to being a $U_q
(\widehat{\mathfrak{sl}}_n)$-module, $\mathfrak{F}$ is a module for a
Heisenberg Lie algebra, spanned by ``current'' operators $J_k$, and by a
central vector $1$. The operator $J_k$ (denoted $B_k$ in {\cite{KMS}}, and defined in \eqref{eq:J-def}) shifts
one fermion to a different level by changing its energy from $i$ to $i - kn$.
The operators $J_k$ with $k > 0$ are thus right-moving operators, and those
with $k < 0$ are left-moving. They satisfy $[J_k, J_l] = 0$ unless $k = - l$.

Introduce the operators $H_+ (z)$ and $H_- (z)$ defined by
\begin{equation}
    \label{eq:intro-Hpm}
  H_{\pm} (z) := \sum_{k = 1}^{\infty} \frac{1}{k}  (1 - v^k) z^{\pm nk}
  J_{\pm k} \hspace{0.17em} .
\end{equation}
Our main theorem, which will be proved in Section~\ref{sec:proofmain}, is:
\begin{maintheorem}
  \label{thm:eH-equals-T}The operators $e^{H_+(z)}$ and $e^{H_-(z)}$
  equal the row transfer matrices of Gamma and Delta ice:
  \begin{equation}
    \label{ehtform}
    e^{H_+(z)} = T_\Delta(z),\qquad e^{H_-(z)} = T_\Gamma(z)\;.
  \end{equation}
\end{maintheorem}
Operators such as these occur in conformal field theory, and also other areas
of mathematics such as soliton theory, ``monstrous moonshine'' and the
abstract boson-fermion correspondence. Generally, we will call an operator of
the form
\begin{equation}
  \label{hplusadef} \exp (H_+ [a] (z)), \hspace{2em} H_+ [a] (z) = \sum_{k =
  1}^{\infty} a_k J_k z^k
\end{equation}
or
\begin{equation}
  \label{hminadef} \exp (H_- [b] (z)), \hspace{2em} H_- [b] (z) = \sum_{k =
  1}^{\infty} b_{- k} J_{- k} z^{- k}
\end{equation}
a \textit{half-vertex operator}. We must be careful with $H_- [b]$, since
$H_- [b] (z) | \lambda \rangle$ is an infinite sum and not in
$\mathfrak{F}_m$. Nevertheless the sum $\langle \mu |H_- [b] (z) |
\lambda \rangle$ is finite and therefore such expressions make sense; in fact
just $\langle \mu|H_- [b] (z)$ is a finite sum. (Here we use the usual Dirac
notation for operators on $\mathfrak{F}_m$. If $H : \mathfrak{F}_m
\longrightarrow \mathfrak{F}_m$ is an operator, we will denote by $\langle
\mu|H| \lambda \rangle$ the inner product of $H| \lambda \rangle$ with
$|\mu \rangle$.)

Operators of the form $\exp(H_- [b] (z)) \cdot \exp(H_+ [a] (z))$ appear in mathematical
physics. See for example~{\cite{FrenkelKac}}, \cite{KMPS} Part~II in Volume~I
or {\cite{JimboMiwa}} (1.15). Subject to a locality assumption
({\cite{KacBeginners,FrenkelBenZvi}}), they are called \textit{vertex
operators}. In this paper we will deal mainly with half-vertex operators.
Yet there are representation theory contexts in which Gamma ice and Delta ice occur together (\cite{wmd5book,BBBG,Ivanov,BBCG,GrayThesis}) 
leading to vertex operators as above. 
In Section~\ref{sec:vertexops} we show that the locality properties of such operators
fit into the algebraic framework of Frenkel and
Reshetikhin~\cite{FrenkelReshetikhinChiral}.

As mentioned above, the method of Baxter~{\cite{Baxter}} based on
the Yang-Baxter equation produces families of
commuting row-transfer matrices. That is,
$T_{\Delta} (z_1)T_{\Delta} (z_2) = T_{\Delta} (z_2) T_{\Delta} (z_1)$, and
similarly for $T_\Gamma(z)$. On the other hand, the commutativity also follows
from the identity (\ref{ehtform}), because $J_k$ and $J_l$ commute if $k$ and
$l$ have the same sign. Note however that $T_\Delta$ does not commute with
$T_\Gamma$. In Theorem~\ref{thm:gdcommute} we compute precisely a scalar $C(z,w)$
such that $T_{\Delta}(z)T_\Gamma(w)=C(z,w)T_\Gamma(w)T_\Delta(z)$, and
this calculation is essential to our discussion of locality.

In the paragraphs above we have described relationships between quantum
groups, solvable lattice models, and Heisenberg algebras acting on a
Fock space $\mathfrak{F}$. Using these we will make two connections to
existing literature. First, it is shown in {\cite{BrubakerBuciumasBump,BBBG}}
that the Boltzmann weights that we use in this paper can be used in finite
systems whose partition functions are Whittaker functions on the $n$-fold
metaplectic covers of $\GL_r$ over a local field. It is striking that for
these, the relevant quantum group is $U_q (\widehat{\mathfrak{gl}}_n)$ or its
relatives $U_q (\widehat{\mathfrak{gl}} (n| 1))$ or $U_q
(\widehat{\mathfrak{sl}}_n)$. The relationship between the degree $n$ of the
cover and the rank of the quantum group was very unexpected. For the
application to metaplectic Whittaker functions, the quantum group must be
modified by Drinfeld twisting in order to introduce Gauss sums into the
comultiplication of $U_q (\widehat{\mathfrak{sl}}_n)$, and consequently
into the $R$-matrix and quantum wedge relations in $\mathfrak{F}$.

Although the metaplectic Whittaker functions are not symmetric in the
Langlands parameters $\mathbf{z} = (z_1, \cdots, z_r)$, when we switch to the
infinite grids and the Fock space $\mathfrak{F}$, we find expressions such as
\begin{equation}
  \label{metaplecticsf}
   \mathcal{M}_{\lambda / \mu}^n (\mathbf{z}) = \langle \mu |T_{\Delta} (z_1)
   \cdots T_{\Delta} (z_r) | \lambda \rangle = \left\langle \mu \middle| \exp
   \left( \sum_{k = 1}^{\infty} \frac{1}{k} (1 - v^k) p_{nk} (\mathbf{z}) J_k
   \right) \middle| \lambda \right\rangle,
\end{equation}
where $p_{nk} (\mathbf{z}) = \sum_i z_i^{nk}$ is the power-sum symmetric
function. (We use the notation $\mathcal{M}^n_\lambda$ if $\mu$ is the
empty partition.) By Theorem~A, $\mathcal{M}^n_{\lambda/\mu}$
can be interpreted as a partition function very similar to the metaplectic
Whittaker functions. But unlike Whittaker functions, these polynomials are
symmetric. We will call them \textit{metaplectic symmetric functions}. 
In Theorem~\ref{whitfrommw} we will show
how metaplectic Whittaker functions (which are not symmetric) can be expressed
in terms of the new metaplectic symmetric functions.

Thus we will show that the solvable models of \cite{BrubakerBuciumasBump,BBBG}
admit an interpretation in terms of a Heisenberg algebra commuting with a
$U_q(\widehat{\mathfrak{sl}}_n)$ action on Fock space. The case when $n=1$, which
reduces to the Shintani-Casselman-Shalika formula for the general linear group
(or Tokuyama's formula), was treated in Brubaker and Schultz
\cite{BrubakerSchultz}. In that case, values of Whittaker functions are Schur
polynomials, and so recovers a result expressing Schur polynomials
as partition functions of free-fermionic six-vertex models
\cite{HamelKing,ZinnJustinTiling,hkice,PZJ}.

This brings us to the second connection to existing literature. The quantum Fock space
has in prior results {\cite{LLTRibbon,LamRibbon,LamBoson}} been applied in
the theory of LLT polynomials, also known as \textit{ribbon symmetric
functions}. These are $q$-deformations of products of $n$ Schur functions. If
$n$ is large, they become Hall-Littlewood polynomials. They are a reflection
of the plethysm with power-sum symmetric functions (Adams operations) and are
connected with algorithms in the (modular) representation theory of symmetric
groups. They have reappeared in other contexts such as Schur positivity and
affine Schubert calculus.

Lam~{\cite{LamBoson}} formalized a generalized boson-fermion correspondence
that includes these examples and others such as the LLT polynomials.  The
bosonic Fock space $\mathfrak{B}$ may be identified with the ring $\Lambda$ of
symmetric polynomials and (over $\mathbb{Q}$) the power-sum symmetric
functions $p_k$ generate. They give rise to a representation of the Heisenberg
Lie algebra on $\mathfrak{B}$ in which multiplication by, or differentiation
with respect to, the $p_k$ correspond to the operators $J_k$ on the fermionic
Fock space. (See also {\cite{KacRaina,JimboMiwa,DateJimboMiwaSolitons}}.) Lam
explained how to construct symmetric functions from any such Heisenberg
algebra action and reinterpreted results of \cite{LLTRibbon} to put LLT
polynomials into this framework.

As we demonstrate in Section~\ref{LLTsection}, Lam's symmetric function construction is equivalent
to action by half-vertex operators. Thus LLT polynomials may be expressed in the form
\begin{equation}
 \label{eq:defLLT}
 \mathcal{G}_{\lambda / \mu}^n (\mathbf{z}^n) = \left\langle \mu \middle| \exp
 \left( \sum_{k = 1}^{\infty} \frac{1}{k} p_{nk} (\mathbf{z}) J_k \right) \middle|
 \lambda \right\rangle .
\end{equation}
This is very similar to the metaplectic symmetric functions, and indeed we
will show that the metaplectic symmetric functions are specializations of super LLT polynomials, presented
in Definition 29 of~\cite{LamRibbon}. One might suspect from this that LLT polynomials might
likewise be expressible as partition functions of the solvable lattice models from~\cite{BrubakerBuciumasBump} with
boundary conditions determined by the pair of partitions $\lambda$ and $\mu$; in fact this is not possible. It is
only this very particular specialization of the super LLT polynomial that results in an appropriate cancellation of
terms and permits the resulting function to be expressed using our solvable models.

Moreover in \cite{LamBoson}, Lam shows that these families of symmetric functions constructed from Heisenberg algebras satisfy 
a large collection of interesting identities, including Cauchy and Pieri identities. 
Thus, as a consequence of the main theorem, we are now able 
to use these same tools to prove analogous identities for metaplectic symmetric
functions. As proof of concept, we prove a Cauchy identity for the new
metaplectic symmetric functions.
(See Theorem~\ref{thm:metaplecticcauchy}.) In the non-metaplectic setting, such Cauchy identities for Schur functions
found application in the Rankin-Selberg method.

It seems an important question to find other theories that connect the
two mechanisms of solvable lattice models and vertex operators. The
well-known relationship between the Heisenberg spin-chain Hamiltonians and the field-free six
and eight vertex models may be one example. (See Baxter~\cite{BaxterAnisotropic}.)
Another place to look for an analog of our Main Theorem is in the theory
of Hall-Littlewood polynomials. Thus in Jing~\cite{JingHallLittlewood,JingBoson}
a quantum boson-fermion correspondence is described, where the commuting
actions of a Heisenberg Lie algebra with a quantum group is used to study
Hall-Littlewood polynomials in the context of vertex operators. But on the other
hand Korff~\cite{KorffVerlinde}, taking a point of view surprisingly close to
ours, develops a theory of Hall-Littlewood polynomials using lattice models
based on Boltzmann weights that connect with a $q$-deformed \textit{bosonic}
Fock space. Borodin and Wheeler~\cite{BorodinWheeler} and Wheeler and
Zinn-Justin~\cite{WheelerZJIII}
contain further developments of this viewpoint.

As noted above, the results described above concern mainly half-vertex operators,
which have expansions in terms of the positive or negative Heisenberg
generators $J_k$. However, it is also interesting to consider
operators that involve both the positive and negative generators.
Because Gamma ice and Delta ice occur together in several different
contexts, it is
natural to consider ``fields'' such as $V(z)=T_\Gamma(z)T_\Delta(z)$.
We will look at these in Section~\ref{sec:vertexops}, in particular investigating locality properties of
the field $V(z)$. 
It is outside the scope of this paper to fully realize our operators in the language
of vertex algebras, but it seems likely that this can be
done using the framework of quantum vertex
algebras~\cite{FrenkelReshetikhinChiral,EtingofKazhdanQVOA,BorcherdsQuantum} and
we intend to revisit this in a subsequent paper.
Additional future directions may include generalizations of our construction of solvable
lattice models to other Cartan types, perhaps using the abstract Fock space built in
the work of Lanini, Ram and Sobaje \cite{LaniniRamSobaje,LaniniRam}. 

\medbreak\noindent
\textbf{Acknowledgements:} This work was supported by NSF grants
DMS-1406238 and 1801527 (Brubaker) and DMS-1601026 (Bump) and ERC grant AdG 669655 (Buciumas). During his time at Stanford University (when this paper was written), Gustafsson was supported by the Knut and Alice Wallenberg Foundation. 
We thank Gurbir Dhillon, David Kazhdan, Arun Ram and Anne Schilling for helpful
conversations and communications.

\newpage

\section{The Fermionic Fock Space\label{sec:fermionicfock}}
This section reviews the definition of the fermionic Fock space
following Kashiwara, Miwa and Stern~\cite{KMS}. As they showed,
this is a module for the affine quantum group $\Uq$.
However we will require greater generality by giving the Fock space
the structure of a module over a Drinfeld twist of this quantum group.
Thus while we follow~\cite{KMS} very closely, sometimes we add some details
to make clear the differences between working with $\Uq$ or its
Drinfeld twist. Theorem~\ref{thm:heckeactionsagree} appears
to be new and it is a key ingredient that allows us to deduce the action
of the affine Hecke algebra on the Drinfeld twist of the Fock space.

\subsection{The quantum group}

Let $n$ be a positive integer, and let $q$ be either a formal parameter or a \textit{generic} complex number (i.e., not a root of unity). All the indices in the relations in this paper involving elements of the quantum group should be read modulo $n$.

We introduce the quantum group $\Uq$ which acts on the fermionic Fock space, focusing on the quasitriangular bialgebra structure (it is also a Hopf algebra, but we will not be using the antipode anywhere). Let $[m]_q$ be the quantum integer associated to the integer $m$ defined by 
\[[m]_q := \frac{q^{m} - q^{-m}}{q-q^{-1}}.\]
Let $A = (a_{ij})_{0 \leq i,j \leq n-1}$ be the Cartan matrix of affine type $\widehat{A}_{n-1}$. Its non-zero entries are $a_{ii}=2$ and $a_{ij}=-1$ when $i = j \pm 1$ for $n \geq 3$ (where we recall that the indices should be read modulo $n$). For $n=2$ the second equality in the definition of the Cartan matrix is replaced by $a_{ij} = -2$. 

The quantum group $\Uq$ is the unital algebra generated by elements $E_i, F_i, K_i^{\pm}$ for  $0 \leq i \leq n-1$, subject to the following relations (when $n\geq 3$ or $n=1$):

\begin{equation}\label{qgrelations}
\begin{split}
&\begin{aligned}
    K_i K_j &= K_j K_i, \quad K_i E_j = q^{a_{ij}} E_j K_i, \quad K_i F_j = q^{-a_{ij}} F_j K_i, \\[0.5em]
    E_i E_j &= E_j E_i, \quad F_i F_j = F_j F_i  \quad \text{ if } i \neq j \pm 1, \\
    E_i F_j &- F_j E_i = \delta_{i,j}\frac{K_i -K^{-1}_i}{q-q^{-1}}, \\
\end{aligned}\\[0.5em]
&\begin{aligned}
    E_i^{2} E_{i \pm 1} &- (q+q^{-1})E_i E_{i\pm 1} E_i && \hspace{-0.85em} + E_{i\pm 1}E^2_i && \hspace{-0.85em} = 0, \\
    F_i^{2} F_{i \pm 1} &- (q+q^{-1})F_i F_{i\pm 1} F_i && \hspace{-0.85em} + F_{i\pm 1}F^2_i && \hspace{-0.85em} = 0.
\end{aligned}
\end{split}
\end{equation}

In the case $n=2$, the last two relations are replaced by the following relations:
\begin{equation}\label{qgrelationsn=2}
\begin{aligned}
E_i^{3} E_{i \pm 1} &- [3]_q E^2_i E_{i\pm 1} E_i && \hspace{-0.85em} + [3]_q E_i E_{i\pm 1} E^2_i && \hspace{-0.85em} - E_{i\pm 1}E^3_i && \hspace{-0.85em} = 0, \\
F_i^{3} F_{i \pm 1} &- [3]_q F^2_i F_{i\pm 1} F_i && \hspace{-0.85em} + [3]_q F_i F_{i\pm 1} F^2_i && \hspace{-0.85em} - F_{i\pm 1}F^3_i && \hspace{-0.85em} = 0.
\end{aligned}
\end{equation}

The subalgebra of $\Uq$ generated by $E_i, F_i, K_i^{\pm}$ for  $1 \leq i \leq n-1$ is the finite quantum group $U_q(\mathfrak{sl}_n)$.

\begin{remark}
In the case $n=1$, $U_q(\widehat{\mathfrak{sl}}_1)$ is the algebra generated by $K_{0}^{\pm1}$. We will show that our method produces interesting six-vertex models even starting from this ``trivial'' quantum group.  
\end{remark}

\begin{remark}
The quantum group we denote by $\Uq$ is denoted by $U_q'(\widehat{\mathfrak{sl}}_n)$ in \cite{KMS}. Our quantum group does not contain a derivation $d$.  
\end{remark}

The comultiplication $\Delta$ on $\Uq$ is defined as follows:
\begin{equation}\label{qgcomultiplication}
    \begin{split}
        \Delta(K_i) &= K_i \otimes K_i, \\
        \Delta(E_i) &= 1 \otimes E_i + E_i \otimes K_i, \\
        \Delta(F_i) &= F_i \otimes 1 + K_i^{-1} \otimes F_i.
    \end{split}
\end{equation}

Let $V_n$ be an $n$-dimensional vector space with basis $\{v_1, \ldots, v_n\}$. The natural module of $\Uq$, which we denote by $V_n(z)$, is the vector space $V_n \otimes \cmplx[z,z^{-1}] $ with basis $\{ z^k v_i \}$ for $1 \leq i \leq n , k \in \Z$. Another useful basis is $\{ u_j \}$ with $j \in \Z$ satisfying the relations 
\begin{equation}\label{vtoutransition}
u_{j-kn} = z^k v_j,
\end{equation}
for $1 \leq j \leq n, k \in \Z$. The action of $\Uq$ on $V_n(z)$ is as follows:
\begin{equation}\label{qgaction}
\begin{split}
K_i z^k v_j &= q^{\delta_{i,j} - \delta_{i+1,j}} z^k v_j, \\
E_i z^k v_j &= \delta_{i, j-1}z^{k-\delta_{i,0}}v_{j-1} , \\
F_i z^k v_j &= \delta_{i,j} z^{k-\delta_{i,0}} v_{j+1}.
\end{split}
\end{equation}
There is a natural ordering on the basis $\{ z^k v_j\}$:
\begin{equation}\label{eq:sequence} 
\cdots > z^{k-1}v_2 >z^{k-1} v_1 > z^{k} v_n >z^{k} v_{n-1}> \cdots.
\end{equation}
Note that in the $\{ u_j \}$-basis, the ordering is just $u_{j+1}>u_j$. 

There is an action of $\Uq$ on tensor powers of the natural module
$V_n(z)^{\otimes N}$ given by iterations of the comultiplication. An affine
version of Schur-Weyl duality was studied in \cite{GRV,VaragnoloVasserot},
where it is shown that the centralizer of the action of $\Uq$ on
$V_n(z)^{\otimes N}$ is the Hecke algebra $\widehat{H}_N(v)$, for $v=q^2$. In
\cite{Green}, the affine quantum Schur algebra is introduced and a double
centralizer property is proved (though note that the definition of the affine
quantum group in \cite{Green} is slightly different from our definition).

\subsection{\label{sec:aha}The affine Hecke algebra}
The (type A) affine Hecke algebra $\widehat{H}_N(q^2) =:\aH$ is the associative algebra with generators $T_i$ for $1 \leq i \leq N-1$ and $y^{\pm}_j$ for $1 \leq j \leq N$ subject to the following relations:
\begin{equation}\label{heckerelations}
    \begin{split}
        & T_i^2 = (q^2-1)T_i + q^2, \\
        & T_i T_{i+1}T_i = T_{i+1} T_i T_{i+1}, \\
        & \hspace{-0.9em} 
        \begin{aligned}
            T_i T_j &= T_j T_i  &&\text{ if } |i-j| > 1, \\
            y_i y_j &= y_j y_i, \\
            y_j T_i &= T_i y_j  &&\text{ if } i \neq j, j+1, \\
            T_i y_i T_i &= q^2 y_{i+1}. 
        \end{aligned}\\
    \end{split}
\end{equation}
The first relation in the definition of the Hecke algebra can be rewritten as $(T_i +1)(T_i - q^2) = 0$, which allows one to decompose any space on which $T_i$ acts into eigenspaces corresponding to its two eigenvalues: $q^2$ and $-1$. 

We denote by $S_N$ the symmetric group on $N$ strands. For $\sigma \in S_N$, let $\sigma = s_{i_1} \cdots s_{i_l}$ be a minimal length expression, where $s_i \in S_N$ are the simple permutations. It is then well known that the definition 
\begin{equation}
    \label{eq:T-sigma}
    T_\sigma := T_{s_{i_1}} \cdots T_{s_{i_l}}
\end{equation}
is independent of which minimal length expression of $\sigma$ we choose and
that the set $\{T_{\sigma}, \sigma \in S_N\}$ is a basis for the \textit{the finite Hecke}
algebra $H_N \subset \aH$, which by definition is the algebra generated by $T_1,\cdots,T_{N-1}$.

We denote $(V_n(z))^{\otimes N} = V_n(z_1) \otimes V_n(z_2) \otimes \cdots \otimes V_n(z_N)$ to distinguish between the indeterminates corresponding to different copies of $V_n(z)$. The space $V_n(z)^{\otimes N}$ has a basis 
\begin{equation}
  \label{mathbfzdef}
  v_{\mathbf{j}} \otimes \mathbf{z}:= v_{j_1} \otimes \cdots \otimes
  v_{j_N} \otimes z_1^{k_1}\cdots z_N^{k_N}
\end{equation}
where $\mathbf{j} = (j_1,\cdots,j_N)$ and $\mathbf{z}$ is shorthand for $z_1^{k_1}\cdots z_N^{k_N}$, $k_i \in \Z$. The symmetric group $S_N$ acts on all elements of the form $v_{\mathbf{j}}$ by permutation; it also acts on all elements of the form $\mathbf{z}$ as follows: $s_i : \cdots z_i^{k_i} z_{i+1}^{k_{i+1}}\cdots \mapsto \cdots z_{i+1}^{k_i} z_{i}^{k_{i+1}}\cdots $. 

\begin{remark}
  The notation $\mathbf{z}$ has a different meaning in this section than in
  the introduction. In this section (following~\cite{KMS}) $\mathbf{z}$ is
  defined by (\ref{mathbfzdef}).
\end{remark}

There is a right action of the Hecke algebra $\aH$ on the tensor product $V_n(z)^{\otimes N}$ which was first written down in \cite{GRV}: 

\begin{equation}\label{heckeaction1}
\begin{split}
(v_{\mathbf{j}}\otimes \mathbf{z} ) \cdot T_i &=
    \begin{cases}
		(1-q^2)v_{\mathbf{j}}\otimes \frac{z_{i+1}\mathbf{z}^{s_i} -z_i \mathbf{z}}{z_{i}-z_{i+1}} - q v_{s_i(\mathbf{j})} \otimes \mathbf{z}^{s_i}  & \mbox{if } j_i < j_{i+1}, \\
		(1-q^2)v_{\mathbf{j}}\otimes \frac{z_{i}(\mathbf{z}^{s_i} -\mathbf{z})}{z_{i}-z_{i+1}} -  v_{s_i(\mathbf{j})} \otimes \mathbf{z}^{s_i}  & \text{if } j_i = j_{i+1}, \\
		(1-q^2)v_{\mathbf{j}}\otimes \frac{z_{i}(\mathbf{z}^{s_i} -  \mathbf{z})}{z_{i}-z_{i+1}} - q v_{s_i(\mathbf{j})} \otimes \mathbf{z}^{s_i}  & \text{if } j_i > j_{i+1}, \\
    \end{cases}
 \\
(v_{\mathbf{j}}\otimes \z ) \cdot y_i &= (v_{\mathbf{j}}\otimes \z \cdot z^{-1}_i ).
\end{split}
\end{equation}
A crucial fact in defining the quantum Fock space is the property that the right action of $\aH$ and the left action of $\Uq$ on $V_n(z)^{\otimes N}$ commute. 

Let $V_x$, $x \in \cmplx^{\times}$ be the \textit{evaluation} module of $\Uq$. It is the quotient of the natural module by the submodule spanned by elements $v_{i}z^{k+1} - xv_iz^{k}$. It is called the evaluation module because we ``evaluate'' the indeterminate $z$ at $x \in \cmplx^{\times}$. In \cite[Section~3]{BBBF}, the first three authors and Friedberg give examples of representations of the affine Hecke algebra on evaluation modules of quantum groups with applications to the study of metaplectic Whittaker functions. There is a ``natural''  lifting of the action in \cite{BBBF} to an action of $\aH$ on $V_n(z)^{\otimes N}$ which involves the affine $R$-matrix. 

The quantum group $\Uq$ is quasitriangular; this means there is an element living in (a completion of) $\Uq \otimes \Uq$ called the universal $R$-matrix, which we denote by $\mathcal{R}$, satisfying certain well-known properties. See Proposition 4.1 in \cite{FR} for a formula of the universal $R$-matrix of $\Uq$. The action of $\mathcal{R}$ on $V_n(z)^{\otimes 2} = V_n(z_i) \otimes V_n(z_{i+1})$ is given by the affine $R$-matrix $R(z_{i},z_{i+1}) \in \text{End}(V_n \otimes V_n) \otimes \cmplx[z_i, z_{i+1}]$ defined as:
\begin{equation}
\begin{split}
\tau R(z_i, z_{i+1}) = \sum_i (q z_i -q^{-1}z_{i+1}) e_{ii}\otimes e_{ii} + \sum _{i \neq j} (z_i - z_{i+1})e_{ji} \otimes e_{ij} \\+ \sum_{i>j} (q-q^{-1})z_i e_{jj} \otimes e_{ii} + (q-q^{-1})z_{i+1} e_{ii} \otimes e_{jj},
\end{split}
\end{equation}
where $e_{ij} \in \text{End}(V_n \otimes V_n)$ are the maps $e_{ij}:v_j \mapsto \delta_{ij} v_i$ and $\tau \in \text{End}(V_n \otimes V_n)$ is the flip map $\tau : v_{i} \otimes v_j \mapsto v_j \otimes v_i $. See the unnumbered equation between equations 30 and 31 in \cite{FR} and the preceding discussion for an explanation of the fact that the action of $\mathcal{R}$ on $V_n(z_1) \otimes V_n(z_2)$ is $R(z_1,z_2)$.
Denote by $R(z):=R(1,z)$.

\begin{remark}
If we replace the indeterminates $z_i$ and $z_{i+1}$ in $R(z_i, z_{i+1})$ by complex numbers $x_i$ and $x_{i+1}$, we obtain the affine (type A) $R$-matrix for evaluation modules discovered by Jimbo \cite{JimboToda} before the work of Frenkel and Reshetikhin \cite{FR}.
\end{remark}

The natural version of the evaluation action of the affine Hecke algebra given in \cite{BBBF}, Theorem 3.3 reads as follows:
\begin{equation}\label{heckeaction2}
\begin{split}
    (v_{\mathbf{j}}\otimes \mathbf{z} ) \cdot T_i  &= (q^2-1) \frac{z_i}{z_i-z_{i+1}} v_{\mathbf{j}} \otimes \mathbf{z} -  \frac{q}{z_i-z_{i+1}} (\tau R_q)_{i,i+1} (z_{i},z_{i+1}) v_{\mathbf{j}} \otimes \mathbf{z}^{s_i}, \\
    (v_{\mathbf{j}}\otimes \z ) \cdot y_i &= (v_{\mathbf{j}}\otimes \z \cdot z^{-1}_i ).
\end{split}
\end{equation}

\begin{theorem}\label{thm:heckeactionsagree}
The actions of the affine Hecke algebra in equations \eqref{heckeaction1} and \eqref{heckeaction2} agree.
\end{theorem}
\begin{proof}
This follows by the following computation:
\begin{equation*}
\begin{split}
& ((v_{\mathbf{j}}\otimes \mathbf{z} ) \cdot T_i)_{\text{eq. } \eqref{heckeaction1}} = 
    \begin{cases}
		(1-q^2)v_{\mathbf{j}}\otimes \frac{z_{i+1}\mathbf{z}^{s_i} -z_i \mathbf{z}}{z_{i}-z_{i+1}} - q v_{s_i(\mathbf{j})} \otimes \mathbf{z}^{s_i}  & \mbox{if } j_i < j_{i+1} \\
		(1-q^2)v_{\mathbf{j}}\otimes \frac{z_{i}(\mathbf{z}^{s_i} -\mathbf{z})}{z_{i}-z_{i+1}} -  v_{s_i(\mathbf{j})} \otimes \mathbf{z}^{s_i}  & \text{if } j_i = j_{i+1} \\
		(1-q^2)v_{\mathbf{j}}\otimes \frac{z_{i}(\mathbf{z}^{s_i} -  \mathbf{z})}{z_{i}-z_{i+1}} - q v_{s_i(\mathbf{j})} \otimes \mathbf{z}^{s_i}  & \text{if } j_i > j_{i+1} \\
    \end{cases}
\\
&= (q^2-1) \frac{z_i}{z_i-z_{i+1}} v_{\mathbf{j}} \otimes \mathbf{z} - \frac{1}{z_i-z_{i+1}}
    \begin{cases}
		((q^2-1)v_{\mathbf{j}} z_{i+1} + q (z_i-z_{i+1})v_{s_i(\mathbf{j})} )\otimes \mathbf{z}^{s_i}  & \text{if } j_i < j_{i+1} \\
		((q^2-1)v_{\mathbf{j}} z_i  + (z_i-z_{i+1}) v_{s_i(\mathbf{j})} ) \otimes \mathbf{z}^{s_i}  & \text{if } j_i = j_{i+1} \\
		((q^2-1)v_{\mathbf{j}} z_{i} + q(z_i-z_{i+1}) v_{s_i(\mathbf{j})}) \otimes \mathbf{z}^{s_i}  & \text{if } j_i > j_{i+1} \\
    \end{cases}
\\
& =(q^2-1) \frac{z_i}{z_i-z_{i+1}} v_{\mathbf{j}} \otimes \mathbf{z} - \frac{1}{z_i-z_{i+1}} 
    \begin{cases}
		((q^2-1)v_{\mathbf{j}} z_{i+1} + q (z_i-z_{i+1})v_{s_i(\mathbf{j})} )\otimes \mathbf{z}^{s_i}  & \text{if } j_i < j_{i+1} \\
		((q^2 z_i  -z_{i+1}) v_{\mathbf{j}} ) \otimes \mathbf{z}^{s_i}  & \mbox{if } j_i = j_{i+1} \\
		((q^2-1)v_{\mathbf{j}} z_{i} + q(z_i-z_{i+1}) v_{s_i(\mathbf{j})}) \otimes \mathbf{z}^{s_i}  & \text{if } j_i > j_{i+1} \\
	\end{cases}
\\
& = (q^2-1) \frac{z_i}{z_i-z_{i+1}} v_{\mathbf{j}} \otimes \mathbf{z} -  \frac{q}{z_i-z_{i+1}} (\tau R_q)_{i,i+1} (z_i, z_{i+1}) v_{\mathbf{j}} \otimes \z^{s_i} = ((v_{\mathbf{j}}\otimes \mathbf{z} ) \cdot T_i)_{\text{eq. } \eqref{heckeaction2}} .
\end{split}
\end{equation*}
\end{proof}

The importance of Theorem \ref{thm:heckeactionsagree} is twofold. First it
clarifies the relation between the two actions of the affine Hecke algebra
which were discovered in different contexts. Secondly, it gives us a way to
rewrite the action in \cite{KMS}, which is instrumental in the construction of
the Fock space representation, in terms of the affine $R$-matrix. In the next
sections we use a Drinfeld twist of the $R$-matrix to write down a different
action of the affine Hecke algebra on $V_n(z)^{\otimes N}$ (which commutes
with the action of a Drinfeld twist of the quantum group $\Uq$). This allows
us to define the Drinfeld twist of the quantum Fock space. 
   
\subsection{Drinfeld twisting}

The Drinfeld twist \cite{DrinfeldQuasiHopf} is a deformation of the Hopf
algebra structure of a quantum group that changes the comultiplication,
antipode and the universal $R$-matrix, but leaves the multiplication, unit and
counit intact. Drinfeld twisting produces new solutions of the Yang-Baxter
equation.

Reshetikhin \cite{ReshetikhinMultiparameter} proved that given a quantum group $H$ and an element $F \in H \otimes H$ of the form $F = \sum_i f^i \otimes f_i$ satisfying certain properties (\cite[Section~1]{ReshetikhinMultiparameter}), one can define a Drinfeld twist of $H$, denoted $H^F$, with a new comultiplication and universal $R$-matrix given by  
\begin{equation}\label{eq:DtwistDeltaRmatrix}
    \begin{split}
        \Delta^{F}(a) &= F \Delta(a) F^{-1} \\
        \mathcal{R}^F &= F_{21} \mathcal{R} F^{-1}
    \end{split}
\end{equation}
where $F_{21} = \sum_i f_i \otimes f^i$. He then shows, in \cite{ReshetikhinMultiparameter} Section 2, that  
\begin{equation}\label{eq:definitionofF}
F =  \exp \left(\sum_{1 \leq i < j \leq n } a_{ij} (H_i \otimes H_j - H_j \otimes H_i) \right) \end{equation}
satisfies the relations needed to produce a Drinfeld twist of $U_q(\mathfrak{sl}_n)$, where $K_i = q^{H_i}$ and $a_{ij} \in \cmplx$. 

Let $U_q^F(\widehat{\mathfrak{sl}}_n)$ be the quantum group obtained by applying a Drinfeld twist on $\Uq$ using the element $F \in U_q(\mathfrak{sl}_n) \otimes U_q(\mathfrak{sl}_n) \subset \Uq \otimes \Uq$ defined in equation \eqref{eq:definitionofF}. Its comultiplication and universal $R$-matrix will be given by equation \eqref{eq:DtwistDeltaRmatrix}. The twisted quantum group $U_q^F(\widehat{\mathfrak{sl}}_n)$ is the same as $U_q(\widehat{\mathfrak{sl}}_n)$ as algebras, however the coproduct, universal $R$-matrix and antipode are different. The twisted quantum group $U_q^F(\widehat{\mathfrak{sl}}_n)$ also has a natural module $V_n(z).$ (We will abuse notation, but it should be clear throughout the paper when $V_n(z)$ is the natural module of the twisted or untwisted quantum group.) Since the twisted and untwisted quantum groups are the same as algebras, the action $U_q^F(\widehat{\mathfrak{sl}}_n)$ on $V_n(z)$ is the same as the action given in \eqref{qgaction}. However, since the comultiplication is different, the action of the twisted and untwisted quantum groups on $V_n(z)^{\otimes N}$ for $N>1$ will be different.

\begin{remark}
The quantum group in \cite{ReshetikhinMultiparameter} is defined over $\C[[h]]$ as opposed to being defined over $\C(q)$ as in our case. It follows that $F$ defined in \eqref{eq:definitionofF} does not live in $\Uq \otimes \Uq$, but in a certain completion of the tensor product. Similarly, the universal $R$-matrix $\mathcal{R}$ also lives in a completion of $\Uq \otimes \Uq$. These facts will not be problematic for our purposes.
\end{remark}

Recall the definition of $u_i$ from equation \eqref{vtoutransition} and denote the tensor product $u_i \otimes u_j$ by $u_{ij}$. 

\begin{lemma}\label{lemma:F-action}
The elements $F$ (defined in equation \eqref{eq:definitionofF}) and $F_{21}$ act on $V_n(z)^{\otimes 2}$ as follows:
\begin{equation}\label{eq:F-action}
    \begin{split}
        F &: u_i \otimes u_j \mapsto \sqrt{\alpha_{ij}} u_i \otimes u_j, \\
        F_{21} &: u_i\otimes u_j \mapsto \sqrt{\alpha_{ji}} u_i \otimes u_j,
    \end{split}
\end{equation}
where $\alpha_{ii} = 1$ and $\alpha_{ij}$ when $i \neq j$ is given by
\begin{equation}\label{eq:alphaij}
\alpha_{ij} = \exp (2a_{i,j} -2a_{i-1,j}-2a_{i,j-1} +2a_{i-1,j-1})
\end{equation} 
\end{lemma}
\begin{proof}
This follows from noting that $H_i : u_{j} \mapsto (\delta_{i,j}-\delta_{i,j-1} )u_{j}$.
\end{proof}

\begin{proposition}\label{proposition:DaffineRmatrix}
The affine $R$-matrix $R^F(z_{i},z_{i+1})$ corresponding to $U_q^F(\widehat{\mathfrak{sl}}_n)$ is
\begin{equation}
\begin{split}
\tau R^F(z_i, z_{i+1}) = \sum_i (q z_i -q^{-1}z_{i+1}) e_{ii}\otimes e_{ii} + \sum _{i \neq j} \alpha_{ij} (z_i - z_{i+1})e_{ji} \otimes e_{ij} \\+ \sum_{i>j} (q-q^{-1})z_i e_{jj} \otimes e_{ii} + (q-q^{-1})z_{i+1} e_{ii} \otimes e_{jj}.
\end{split}
\end{equation}
\end{proposition}
\begin{proof}
Note that $R^F(z_i, z_{i+1})$ is the action of the universal $R$-matrix $\mathcal{R}^F$ on the representation $V_n(z)^{\otimes 2}$. The result follows immediately after using the action of $F^{-1}$ and $F_{21}$ on $V_n(z)^{\otimes 2}$ from Lemma \ref{lemma:F-action}.  
\end{proof}

Given $U_q^F(\widehat{\mathfrak{sl}}_n)$ with $R$-matrix $R^F(z_{i},z_{i+1})$ as in Proposition \ref{proposition:DaffineRmatrix} which depends on complex numbers  $a_{ij}$, denote by $\bm{\alpha}$ the set of numbers $\alpha_{ij}, 1 \leq i, j \leq n$ obtained from $a_{ij}$ using equation \eqref{eq:alphaij}. From now on we will write the dependence of the Drinfeld twisting in terms of $\bm{\alpha}$ instead of $F$ (so we write $\UqD$ instead of $U_q^F(\widehat{\mathfrak{sl}}_n)$ and $R^{\bm{\alpha}}(z_{i},z_{i+1})$ instead of $R^F(z_{i},z_{i+1})$). Even though there are different choices of $F$ that produce the same set $\bm{\alpha}$, we will not distinguish between such quantum groups. For our purposes, Drinfeld twists by different $F$'s with the same $\bm{\alpha}$'s will correspond to the same six-vertex models in future sections. 

One should keep in mind that for $\alpha_{ij}=1$, $\UqD$ is the non-twisted quantum groups $\Uq$ and that $\alpha_{ij} \alpha_{ji}=1 = \alpha_{ii}$ for any $\bm{\alpha}$. 
A standard, though tedious, computation shows:

\begin{proposition}
There is an action of the Hecke algebra $\aH$ on $V_n(z)^{\otimes N}$ where $y_i$ acts by multiplication with $z_i^{-1}$ and 
\begin{equation}\label{heckeaction2D}
(v_{\mathbf{j}}\otimes \mathbf{z} ) \cdot T_i  = (q^2-1) \frac{z_i}{z_i-z_{i+1}} v_{\mathbf{j}} \otimes \mathbf{z} -  \frac{q}{z_i-z_{i+1}} (\tau R^{\bm{\alpha}}_q)_{i,i+1} (z_{i},z_{i+1}) v_{\mathbf{j}} \otimes \mathbf{z}^{s_i}. 
\end{equation}
\end{proposition}

Equation \eqref{heckeaction2D} can be rewritten, via the same process as in the proof of Theorem \ref{thm:heckeactionsagree}, as
\begin{equation}\label{heckeaction1D}
(v_{\mathbf{j}}\otimes \mathbf{z} ) \cdot T_i= \left\{
	\begin{array}{ll}
		(1-q^2)v_{\mathbf{j}}\otimes \frac{z_{i+1}\mathbf{z}^{s_i} -z_i \mathbf{z}}{z_{i}-z_{i+1}} - q \alpha_{j_i j_{i+1}} v_{s_i(\mathbf{j})} \otimes \mathbf{z}^{s_i}  & \mbox{if } j_i < j_{i+1} \\
		(1-q^2)v_{\mathbf{j}}\otimes \frac{z_{i}(\mathbf{z}^{s_i} -\mathbf{z})}{z_{i}-z_{i+1}} -  v_{s_i(\mathbf{j})} \otimes \mathbf{z}^{s_i}  & \mbox{if } j_i = j_{i+1} \\
		(1-q^2)v_{\mathbf{j}}\otimes \frac{z_{i}(\mathbf{z}^{s_i} -  \mathbf{z})}{z_{i}-z_{i+1}} - q \alpha_{j_i j_{i+1}} v_{s_i(\mathbf{j})} \otimes \mathbf{z}^{s_i}  & \mbox{if } j_i > j_{i+1} \\
	\end{array}
\right. . 
\end{equation}

\begin{proposition}
The action of $T_i$ in equation \eqref{heckeaction1D} is an $\UqD$-module homomorphism. 
\end{proposition}
\begin{proof}
Note that $a \in \UqD$ acts on $V_n(z)^{\otimes 2}$ via $\Delta^F(a) = F \Delta(z) F^{-1}$ and using the action of $F$ and $F^{-1}$ on $V_n(z)^{\otimes 2}$ from equation \eqref{eq:F-action}, the proof becomes a routine calculation. 

A non-computational proof goes as follows: by equation \eqref{heckeaction2D}, the action of $T_i$ on $V_n(z)^{\otimes N}$ is a linear combination of the identity map and $(\tau R)_{i,i+1}(z_i, z_{i+1})$, both of which are $\UqD$-module homomorphisms.
\end{proof}

It follows that the right action of $\aH$ from equation \eqref{heckeaction1D} (which depends on $\bm{\alpha}$) and the left action of $\UqD$ on $V_n(z)^{\otimes N}$ commute.

\subsection{The quantum wedge}

We now define the exterior product of $V_n(z)$ following \cite{KMS}. 

Define the $q$-antisymmetrizing operator $A^{(N)}$ acting on $V_n(z)^{\otimes N}$ to be
\[ A^{(N)} = \sum_{\sigma \in S_N} T_{\sigma}. \]
where $T_\sigma$ was defined in~\eqref{eq:T-sigma}.

\begin{proposition}
The $\UqD$-module $V_n(z)^{\otimes N}$ decomposes as 
\[ V_n(z)^{\otimes N} = \Im A^{(N)} \oplus \Ker A^{(N)}  \]
and the space $\Ker A^{(N)}$ is the sum of the kernels of the operators $T_i+1$ for $1 \leq i \leq N-1$. 
\end{proposition}
\begin{proof}
See Propositions 1.1 and 1.2 in \cite{KMS}. Their proof goes through unchanged even for the new action of $\aH$ on $V_n(z)^{\otimes N}$ from equation \eqref{heckeaction1D}.
\end{proof}

In order to understand the spaces $\Ker (T_i +1)$ which determine $\Ker(A^{(N)})$, take $N=2$ and $T:= T_1$. 

Given integers $m$ and $l$, let $k_1,k_2 \in \Z $ and $1\leq j_1,j_2 \leq n$ be
  such that $l=j_1-k_1n$ and $m=j_2-k_2n$ so that $u_l=v_{j_1} z^{k_1}$ and $u_m=v_{j_2} z^{k_2}$. For such integers $m$ and $l$, define $\alpha_{lm} := \alpha_{ij}$. Then the following elements in $V_n(z) \otimes V_n(z)$ are in $\Ker(T+1)$:  
  
\begin{equation}\label{eq:relationsinkernel}
\begin{aligned}
& u_{l} \otimes u_m + u_m \otimes u_l & & \text{ if } l \equiv m \text{ mod } n \\[1em]
& \begin{multlined}[b]
    u_l \otimes u_m +q \alpha_{l,m} u_m \otimes u_l +{} \\ + u_{m-i} \otimes u_{l+i} + q \alpha_{m-i, l+i} u_{l+i} \otimes u_{m-i}
\end{multlined}
& & \text{ if } m-l \equiv i \text{ mod } n \text{ and } 0 < i < n. 
\end{aligned}
\end{equation}

Note that $\alpha_{l,m} = \alpha_{m-i,l+i}$ when $m-l \equiv i \text{ mod } n$ and $0<i<n$. 

Define the quantum wedge $\Lambda^2 V_n(z)$ to be the quotient $V_n(z)^{\otimes 2}/ \Ker A^{(2)}$ and denote by $u_l \wedge u_m$ the image of $u_l \otimes u_m$ in $\Lambda^2 V_n(z)$. 
It is easy to see from equation \eqref{eq:relationsinkernel} that the following relation holds in $\Lambda^2 V_n(z)$ when $m = l \text{ mod }n$:
\begin{equation}\label{eq:qwedge1}
u_l \wedge u_m = - u_m \wedge u_l.
\end{equation}

If $m, l$ are integers such that $m>l$ and $m-l \equiv i \text{ mod } n$, then consider the following sequence of ordered elements taken out of equation \eqref{eq:sequence}:
\begin{equation}\label{eq:sequence45}
\cdots > u_m > u_{m-i} > u_{m-n} > u_{m-n-i} > \cdots > u_{l+n+i} > u_{l+n} > u_{l+i} > u_l \cdots
\end{equation}

We say a wedge $u_l \wedge u_m$ is \textit{normal-ordered} if $l>m$, so that
$u_l > u_m$ in the order given by equation \eqref{eq:sequence}. For $u_m > u_l$, the following relation holds in $\Lambda^2 V_n(z)$:
\begin{equation}\label{eq:KMS45}
\begin{split}
u_l \wedge u_m  = -q \alpha_{l,m} u_m \wedge u_l + (q^2-1)(
u_{m-i} \wedge u_{l+i} - q \alpha_{l, m} u_{m-n} \wedge u_{l+n} + \\
q^2 u_{m-n-i} \wedge u_{l+n+i} - q^3 \alpha_{m,l} u_{m-2n} \wedge u_{l+2n}  \cdots )
\end{split}
\end{equation}
where the sum on the right uses entries in the sequence \eqref{eq:sequence45}
and continues as long as we get normal-ordered wedges.
Here $i$ is the unique value with $0<i<n$ and $m-i\equiv l$ modulo~$n$. 
This fact follows by applying the second line in
equation \eqref{eq:relationsinkernel} repeatedly until we obtain a formula for
$u_l \wedge u_m$ in terms of normal-ordered wedges only.

In the special case when $\alpha_{ij}=1$ we get back equation (45) in
\cite{KMS}. The specialization we need may be described as follows. Let $g$ be
a function of integers modulo $n$ that satisfies the following Assumption.

\begin{assumption}\label{assumptiong(a)}
  Let $v$ denote $q^2$. The function $g$ satisfies $g(0)=-v$, and if $a$ is
  not congruent to $0$ modulo $n$, then $g(a)g(-a)=v$.
\end{assumption}

Now let us take $\alpha_{ij} = -q^{-1} g(i-j)$ when $i\neq j$ (we always want
$\alpha_{ii}=1$.) We obtain the formula, valid if $m>l$:

\begin{equation}
    \label{eq:quantum-wedge}
    u_l \wedge u_m = \left\{ \begin{aligned}
     &- u_m \wedge u_l & \text{if $l \equiv m$ mod $n$},\\
     &\begin{multlined}[b]
    g (l - m) u_m \wedge u_l + (q^2 - 1) (u_{m - i} \wedge u_{l + i} + g (l -
   m) u_{m - n} \wedge u_{l + n} \\ 
   +q^2 u_{m - n - i} \wedge u_{l + n + i} + q^2\, g (l - m) u_{m - 2 n} \wedge u_{l + 2 n} + \cdots)
\end{multlined}
 & \text{otherwise} .
   \end{aligned} \right.
\end{equation}
As with (\ref{eq:KMS45}), $i$ is the unique value with $0<i<n$ and $m-i\equiv l$
modulo~$n$. And as with (\ref{eq:KMS45}) the summation continues as long as the
terms are of the form $u_a\wedge u_b$ with $a>b$; this is a finite sum.

Let $\Lambda^N V_n(z)$ be the quotient $V_n(z)^{\otimes N} / \Ker(A^{(N)})$. The definition of a normal-ordered wedge extends to $\Lambda^N V_n(z)$.
By identical arguments to the one in Proposition 1.3 of \cite{KMS}, one can show:
\begin{proposition}\label{proposition:generators}
$\Lambda^N V_n(z)$ is the quotient of $V_n(z)^{\otimes N}$ by the relations \eqref{eq:qwedge1} and \eqref{eq:KMS45} in each pair of adjacent factors; the elements 
\[ u_{m_1} \wedge \cdots \wedge u_{m_N} \]
where $m_1 > m_2 > \cdots > m_N$, form a basis for $\Lambda^N V_n(z)$. 
\end{proposition}

\begin{remark}
Note that for $n=1$, $m-l$ is always congruent to $0$ mod $n$. Therefore the
quantum wedge is defined only using relation \eqref{eq:qwedge1}. In this case
the definition of the quantum wedge is the same as the definition of the
classical ($q=1$) wedge for $\widehat{\mathfrak{sl}}_m$ for all~$m$. 
\end{remark}

\subsection{The fermionic Fock space\label{section:ffs}}

Let $S_{\infty}$ be the infinite symmetric group generated by simple reflections $s_i, i \in \mathbb{N}$. Let $\widehat{H}_\infty$ be the infinite affine Hecke algebra, with generators $T_i, y_i^{\pm}, i = 1, 2, 3, \ldots$ subject to the relations \eqref{heckerelations}. It acts on $V_n(z)\otimes V_n(z) \otimes V_n(z) \cdots $ via \eqref{heckeaction1D}; the action is well-defined because each $T_i$ acts only on a pair of adjacent factors. 

Let $\mathfrak{U}_m$ be the linear span of vectors of the form 
\[ u_{i_m} \otimes u_{i_{m-1}} \otimes u_{i_{m-2}} \otimes \cdots \]
such that $i_k = k$ for $k \ll 0$. The Fock space of level $m$ is denoted by $\F_{m}$; it is the quotient of $\mathfrak{U}_m$ by the space $\sum_i \Ker(T_i+1)$, or equivalently, by the relations \eqref{eq:qwedge1} and \eqref{eq:KMS45} in each pair of adjacent factors. 

There is a ``formal'' action of the quantum group $\UqD$ on the space $\mathfrak{U}_m$ via the coproduct \eqref{qgcomultiplication} which descends to genuine action on $\F_m$.  A basis of $\F_{m}$ is given by elements of the form 
\[ u_{i_m}\wedge u_{i_{m-1}}\wedge u_{i_{m-2}} \wedge \cdots \]
where $i_m > i_{m-1} > \cdots$ and $i_k = k$ for $k \ll 0$. Define 
\[\ket{m} := u_m \wedge u_{m-1} \wedge u_{m-3}\wedge \cdots \in \F_{m}, \] 
which we call the vacuum in $\mathfrak{F}_m$.
The Fock space $\F$ is defined as \[\F = \bigoplus_{m \in \mathbb{Z}} \F_{m}.\] 

\bigbreak
\begin{remark}
We caution the reader that due to the ``correction terms'' in
(\ref{eq:quantum-wedge}) there may be unexpected terms in many
calculations. For example, if $n = 2$, (\ref{eq:quantum-wedge}) shows that
\[ u_1 \wedge u_4 \wedge u_1 \wedge u_0 \wedge \cdots = g (- 3) u_3 \wedge u_2
   \wedge u_1 \wedge u_0 \wedge \cdots \; . \]
In the usual wedge, the left-hand side would be zero due to the repeated
factor $u_1$; however we see that this is not true in the quantum Fock space.
\end{remark}

Now let us introduce operators $J_k$ on $\F$. These operators are
$\Uq$-module endomorphisms that are denoted $B_k$ in~\cite{KMS}.
Let $u_{i_m} \wedge u_{i_{m-1}} \wedge u_{i_{m-2}} \wedge \cdots \in \mathfrak{F}_m$
and for a non-zero $k \in \mathbb{Z}$ define the displacement operator $J_k :
\mathfrak{F}_m \to \mathfrak{F}_m$ by
\begin{equation}
  \label{eq:J-def}
  J_k  (u_{i_m} \wedge u_{i_{m-1}} \wedge u_{i_{m-2}} \wedge \cdots) = (u_{i_m - nk}
  \wedge u_{i_{m-1}} \wedge u_{i_{m-2}} \wedge \cdots) + (u_{i_m} \wedge u_{i_{m-2} - nk}
  \wedge u_{i_{m-2}} \wedge \cdots) + \cdots \hspace{0.17em} .
\end{equation}
That this is indeed an action on the quantum Fock space consistent with the
quantum wedge resulting in a finite sum of wedges is shown in {\cite[Lemma 2.1]{KMS}}.
For $u_i \wedge \eta \in \mathfrak{F}_m$ we note that
\begin{equation}
  \label{eq:J-derivative} J_k  (u_i \wedge \eta) = u_{i - nk} \wedge \eta +
  u_i \wedge J_k (\eta) \hspace{0.17em} .
\end{equation}
The following commutation relation holds:
\begin{equation}
  \label{jkjmkcomm} [J_k, J_l] = k \frac{1 - v^{n |k|}}{1 - v^{|k|}} \delta_{k, - l}
  .
\end{equation}
This is Proposition~2.6 in {\cite{KMS}}. This commutator is not affected by the
Drinfeld twisting.

\section{The Main Theorem\label{deltaicefockspace}}

We recall two types of solvable lattice models called Gamma and Delta ice.
These first appeared in \cite{BrubakerBuciumasBump} in the context of
metaplectic Whittaker functions, but as we will exhibit later, they have surprising
connections to symmetric functions beyond this particular application.

Let us begin with a planar grid having a finite number $r$ of rows. The
grid may either have finitely many or infinitely many columns. We will number the
rows $1,\cdots,r$; for Delta ice, the row numbers increase
from the bottom up, and for Gamma ice, they increase from
the top down. We will also number the columns by integers, in decreasing
order. The column numbers may be all integers in the case of infinitely many columns 
or a finite interval, say $0,1,2,\cdots,N$, in the case of finitely many columns.  We will fix nonzero complex numbers $z_1,\cdots,z_r$ and
associate $z_i$ to the row numbered $i$. There are \textit{vertices} at every intersection of a row
and column, and four \textit{edges} adjacent to each vertex as in Table~\ref{tab:mweights}. A \textit{boundary edge} is
an edge that is adjacent to a single vertex.

\begin{table}[h]
\[
\begin{array}{|c|c|c|c|c|c|c|}
\hline
&\tt{a}_1&\tt{a}_2&\tt{b}_1&\tt{b}_2&\tt{c}_1&\tt{c}_2\\
\hline
\Gamma \text{-ice} &
\begin{array}{c}\hspace{-6pt}\scalebox{.75}{\gammaice{+}{+}{+}{+}{a+1}{a}}\\z^{-1}\end{array} &
\begin{array}{c}\scalebox{.75}{\gammaice{-}{-}{-}{-}{0}{0}}\\1\end{array} &
\begin{array}{c}\hspace{-6pt}\scalebox{.75}{\gammaice{+}{-}{+}{-}{a+1}{a}}\\z^{-1}\,g(a)\end{array} &
\begin{array}{c}\scalebox{.75}{\gammaice{-}{+}{-}{+}{0}{0}}\\1\end{array} &
\begin{array}{c}\scalebox{.75}{\gammaice{-}{+}{+}{-}{0}{0}}\\1-v\end{array} &
\begin{array}{c}\scalebox{.75}{\gammaice{+}{-}{-}{+}{1}{0}}\\z^{-1}\end{array} \\
\hline
\Delta \text{-ice} &
\begin{array}{c}\scalebox{.75}{\gammaice{+}{+}{+}{+}{0}{0}}\\1\end{array} &
\begin{array}{c}\scalebox{.75}{\gammaice{-}{-}{-}{-}{a}{a+1}\hspace{-6pt}}\\g(a) z\end{array} &
\begin{array}{c}\scalebox{.75}{\gammaice{+}{-}{+}{-}{0}{0}}\\1\end{array} &
\begin{array}{c}\scalebox{.75}{\gammaice{-}{+}{-}{+}{a}{a+1}\hspace{-5pt}}\\z\end{array} &
\begin{array}{c}\scalebox{.75}{\gammaice{-}{+}{+}{-}{0}{0}}\\ (1-v) z\end{array} &
\begin{array}{c}\scalebox{.75}{\gammaice{+}{-}{-}{+}{0}{1}}\\1\end{array} \\
\hline
\end{array}
\]
\caption{The Boltzmann weights for $\Gamma$ and $\Delta$ vertices associated to
a row parameter $z \in \mathbb{C}^\times$. The charge $a$ above an edge indicates any choice of
charge mod $n$ and gives the indicated weight. The weights depend on a parameter $v$ and any function $g$ with $g(0)=-v$
and $g(n-a)g(a) = v$ if $a \not\equiv 0$ mod $n$. If a configuration does not
appear in this table, its weight is zero. We take $z=z_i$ in the $i$-th row
(from the top for Gamma ice, or from the bottom for Delta ice). For Gamma ice,
the Boltzmann weights used in~\cite{BrubakerBuciumasBump} and~\cite{BBBG} are
multiplied by $z$. This change from those papers only multiplies the
partition function by a constant power of $z_1\cdots z_r$.}
\label{tab:mweights}
\end{table}

A \textit{state} of Gamma or Delta ice is given by the assignment of a 
\textit{spin} $\pm$ to each edge of the grid with certain restrictions. To each horizontally oriented edge,
we will also associate a \textit{charge} which will be an integer
$a$ modulo $n$. The combination of the spin and charge will be
called a \textit{decorated spin} and will be denoted $\pm^a$.
For Delta ice, we only allow the spin $+^a$ when $a$ is $0$
modulo $n$; for Gamma ice, we only allow $-^a$ when $a$ is
$0$ modulo~$n$. Thus in either cases, there are $n+1$
allowed decorated spins.

For the boundary edges, the spins and (for horizontal edges,
the charges) will be fixed. Their specification, together with a
set of Boltzmann weights associated to each vertex according to~\ref{tab:mweights}, 
will define what we call the \textit{system}. In this section, we will consider
systems of infinite width, whose columns are labeled
by all integers. In Section~\ref{sec:metwhit} we will
consider finite systems.

Thus let us describe the boundary conditions when the grid is infinite. 
The boundary edges are all therefore vertically oriented. Let us fix an
integer $m$ and consider two strictly decreasing
sequences of integers,
\begin{equation}
  \label{boldij} \mathbf{i}=(i_m,i_{m-1},\cdots)\,,\qquad \mathbf{j}=(j_m,j_{m-1},\cdots)
\end{equation}
such that $i_k=j_k=k$ if $0\gg k$. The associated boundary
spins along the top edge are $-$
for the edges in columns $i_m,i_{m-1},\cdots$ and
$+$ for the edges in columns $i_m,i_{m-1},\cdots$.
We similarly fix the spins along the bottom boundary to
be $-$ in columns $j_{m},j_{m-1},\cdots$ and $+$
in the others. With these data we may associate the following vectors in $\F_m$:
\[\xi=u_{\mathbf{i}}=u_{i_m}\wedge u_{i_{m-1}} \wedge \cdots \qquad \qquad
\eta=u_{\mathbf{j}}=u_{j_m}\wedge u_{j_{m-1}} \wedge \cdots . \]

A state of this infinite system thus requires assigning
spins to the internal vertical edges and decorated spins for the horizontal
internal ones. For Delta ice (resp.~Gamma ice), we require that all but
finitely many horizontal edges have spins $+^0$ (resp.~$-^0$).
Regardless of whether the grid is finite or infinite, a state $\mathfrak{s}$ of the system 
will be called \textit{admissible} if  the configuration of spins at the adjacent edges 
of every vertex is one of the configurations in a fixed row of
Table~\ref{tab:mweights}. Let $\mathfrak{S}$ denote the set of all admissible states $s$
of the system, determined by the boundary conditions and Boltzmann weights. When
no confusion may arise, we sometimes use the same notation $\mathfrak{S}$ to denote
either the system or its set of admissible states. The two systems we consider will thus be denoted
$\mathfrak{S}_{\mathbf{z},\xi,\eta,r}^\Delta$ or $\mathfrak{S}_{\mathbf{z},\xi,\eta,r}^\Gamma$ according
to the weights in row one and row two of Table~\ref{tab:mweights}, respectively.

\begin{lemma}
  \label{interleavinglemma}
  Let $r=1$ and let $\xi=u_{\mathbf{i}}, \eta=u_\mathbf{j} \in \mathfrak{F}_m$. For either $\mathfrak{S}^\Gamma$ or $\mathfrak{S}^\Delta$, there
  exists at most one admissible state for the system
  $\mathfrak{S}_{\mathbf{z},\xi,\eta,1}$, so $\langle\eta|T(z)|\xi\rangle$
  is the Boltzmann weight of this state (or zero if no admissible
  state exists). If such a state exists, then, for $\mathfrak{S}^\Delta_{\mathbf{z},\xi,\eta,1}$
  \begin{equation}
    \label{interleaving} i_m \geqslant j_m \geqslant i_{m - 1} \geqslant j_{m
    - 1} \geqslant \cdots .
  \end{equation}
  For $\mathfrak{S}^\Gamma_{z,\xi,\eta,1}$, we have instead
  $j_m\geqslant i_m\geqslant \cdots$.
\end{lemma}

\begin{proof}
  Let us consider the case of Delta ice.  
  To see that the state (if it exists) is unique, observe that
  every vertex must have an even number of $-$ signs on its
  adjoining edges. 
  We have required all but finitely many horizontal edges to have
  configuration $+^0$. Suppose that $j_m>i_m$. In that case,
  this observation shows that the spin to the left of the
  $j_m$ column is~$+^0$; so at the vertex in the $j_m$
  column the configuration would be
  \[\scalebox{0.75}{\gammaice{+}{+}{-}{-}{}{}}\]
  which is an illegal pattern. Thus $i_m\geqslant j_m$,
  and continuing this way gives (\ref{interleaving}).
  The case of Gamma ice is similar.
  Compare~\cite{wmd5book} Proposition~19.1 or~\cite{Baxter} Section~8.2.
\end{proof}

The {\textit{Boltzmann
weight of the state}} is the product of the Boltzmann weights at the
vertices. The {\textit{partition function}} $Z(\mathfrak{S})$ is the sum of
the Boltzmann weights over all states.  These definitions make sense by the
following result.

\begin{proposition}
  In the case where the grid is infinite, there are only a
  finite number of states for $\mathfrak{S}^\Delta_{\mathbf{z},\xi,\eta,r}$ or
  $\mathfrak{S}^\Gamma_{\mathbf{z},\xi,\eta,r}$. For each state, all but
  finitely many vertices have Boltzmann weight~$1$,
  so the Boltzmann weight of the state is a finite
  product.
\end{proposition}

\begin{proof}
  The fact that there are only finitely many states is a consequence of
  Lemma~\ref{interleavinglemma}.  With our assumption that all but finitely
  many horizontal edges have decorated spin $+^0$ for Delta ice or $-^0$ for
  Gamma ice, it is not hard to see that for any state $\mathfrak{s}$ all but
  finitely many vertices are in configuration $\tt{a}_1$ or $\tt{b}_1$ for
  Delta ice, or $\tt{a}_2$ or $\tt{b}_2$ for Gamma ice. Since those vertices
  have Boltzmann weight~$1$, the Boltzmann weight of a state is a finite product.
\end{proof}

We will sometimes use the Dirac notation $\xi=|\xi\rangle$
for elements of $\mathfrak{F}$. 
Let us define an inner product on $\mathfrak{F}$ in which the normal-ordered
monomials
\[ \xi = u_{i_m} \wedge u_{i_{m - 1}} \wedge \cdots, \hspace{2em} i_m > i_{m -
   1} > \cdots \]
is an orthonormal basis. There is a unique involution on $\mathfrak{F}$ which is
conjugate-antilinear and which is the identity on the real vector space
spanned by the normal-ordered monomials. If $\xi = | \xi \rangle$ is an
element of $\mathfrak{F}$ we will denote by $\langle \xi |$ its image under
the involution. Then $\langle\eta|\xi\rangle$ will denote the inner
product of $\xi$ and $\eta$. This inner product is linear in $\xi$
and conjugate-linear in $\eta$.

Now let us specialize to Delta ice.
We may define an operator $T_\Delta(\mathbf{z})$ on $\mathfrak{F}_m$ by
\begin{equation}
  \label{transfermat}
 T_\Delta(\mathbf{z})\,\xi=
 T_\Delta(\mathbf{z})\,|\xi\rangle=\sum_\eta Z(\mathfrak{S}^\Delta_{\mathbf{z},\xi,\eta,r})|\eta\rangle.
\end{equation}
It is a consequence of Lemma~\ref{interleavinglemma} that
there are only finitely many terms in the right-hand side. (This
would fail for $\mathfrak{S}^\Gamma_{\mathbf{z},\xi,\eta,r}$.)

In the same notation we may write
\begin{equation}
  \label{ztransip}
  Z(\mathfrak{S}^\Delta_{\mathbf{z},\xi,\eta,r})=\langle\eta|T_\Delta(\mathbf{z})|\xi\rangle.
\end{equation}
In the special case where $r=1$, we will use the notation $T_\Delta(z)$
with $\mathbf{z}=(z)$. We call the operator the \textit{row transfer matrix}.
We have $T_\Delta(\mathbf{z})=T_\Delta(z_1)\cdots T_\Delta(z_r)$.

\begin{remark}
  \label{rmk:Gamma-finite}
  In (\ref{ztransip}) we have specialized to the case of Delta ice.
  For $\mathfrak{S}^\Gamma_{\mathbf{z},\xi,\eta,r}$ the sum (\ref{transfermat})
  would fail to be finite. Nevertheless we could similarly define $T_\Gamma(\mathbf{z})$
  for Gamma ice as an operator on ``bras'' $\langle\eta|$ instead of ``kets''
  $|\xi\rangle$ by the formula
  \[\langle\eta|T_\Gamma(\mathbf{z})=\sum_{\xi}Z(\mathfrak{S}^\Gamma_{\mathbf{z},\xi,\eta,r})\langle\xi|,\]
  which is a finite sum. Then (\ref{ztransip}) would still be correct.
\end{remark}

We specialize now to the case $r=1$ and denote $z=z_1$. As in \eqref{eq:intro-Hpm}, we define
operators $H_+(z)$ and $H_-(z)$ on $\F_m$ by 
\begin{equation}
  \label{hpmdef}
  H_{\pm} (z) := \sum_{k = 1}^{\infty} \frac{1}{k}  (1 - v^k) z^{\pm nk}
  J_{\pm k}
\end{equation}
If $\xi\in\mathfrak{F}$ then $H_+(z)\xi=H_+(z)|\xi\rangle$ is a finite
sum. For $H_-(z)$, this fails, but as with $T_\Gamma(z)$, 
we may interpret $H_-(z)$ as an operator by the formula
\[\langle\eta| H_-(z)=\sum_\xi \langle\eta|H_-(z)|\xi\rangle\langle\xi|,\]
and this is a finite sum.

Our main theorem (Theorem~\ref{thm:eH-equals-T}), states that
\begin{equation}
    e^{H_+(z)} = T_\Delta(z),\qquad e^{H_-(z)} = T_\Gamma(z)\;.
\end{equation}

We will prove this in the next section. As an immediate consequence,
the row transfer matrices $T_\Delta(z)$ and $T_\Gamma(z)$ are $\Uq$-module
homomorphisms, because the operators $J_k$ are.

\section{\label{sec:proofmain}Proof of the Main Theorem}

The proof is structured as follows. We will first prove the statement for Delta ice using induction to reduce the proof to an identity for two finite subsystems where we get a finite number of cases that are checked in Tables~\ref{tab:caseisubcases} and~\ref{tab:caseiisubcases}. One reason for starting with Delta ice is because of Remark~\ref{rmk:Gamma-finite} together with normal-ordering issues. The transfer matrix for Gamma ice is then related to the adjoint of the Delta ice transfer matrix in Subsection~\ref{sec:gammaproof}.
Therefore, until Subsection~\ref{sec:gammaproof} we will consider
Delta ice. We will fix $z$, and let $T=T(z)$ be the transfer matrix
(\ref{transfermat}) of the one-rowed system, and $H=H_+(z)$.

We pause to refine the criterion in Lemma~\ref{interleavinglemma}
for an admissible state to exist in the one-row system
$\mathfrak{S}_{\mathbf{z},\xi,\eta,1}$. For even if (\ref{interleaving}) is satisfied,
there may not be an admissible state $\mathfrak{s}$. Let us describe a further
condition that must be satisfied.

We may write $\xi = u_{i_m} \wedge u_{i_{m - 1}} \wedge \cdots$ and $\eta =
u_{j_m} \wedge u_{j_{m - 1}} \wedge \cdots$. By (\ref{interleaving})
$i_m \geqslant j_m \geqslant i_{m - 1} \geqslant \cdots$, and if $r$ is sufficiently negative, then $i_r =
r$ and $j_r = r$. The substance of the lemma that we will now state is that
there is a bijection between the two sequences $\mathbf{i} = (i_m, i_{m - 1},
\cdots)$ and $\mathbf{j} = (j_m, j_{m - 1}, \cdots)$, and that corresponding
elements are congruent modulo $n$.

Since the elements of $\mathbf{j}$ are distinct, each $i_a$ can be equal to a
unique $j_b$, which must be either $j_a$ or $j_{a + 1}$. In this case we say
that $i_a$ and $j_b$ are \textit{paired}. It remains for the bijection to be
defined on those elements of $\mathbf{i}$ (resp.~$\mathbf{j}$) that are not
equal to any element of the other sequence. Thus we say that the index $i_a$
is \textit{isolated} for the pair $\xi, \eta$ if $j_{a + 1} > i_a > j_a$,
and similarly we say that the index $j_b$ is \textit{isolated} if $i_b > j_b
> i_{b - 1}$. The isolated indices $i_a$ and $j_b$ are \textit{paired} if
\begin{equation}
  \label{isolatedij}
   j_{a + 1} > i_a > j_a = i_{a - 1} > j_{a - 1} = i_{a - 2} > \cdots j_{b +
     1} = i_b > j_b > i_{b - 1} .
\end{equation}
(We omit the condition $j_{a + 1} > i_a$ if $a = m$.)
The condition (\ref{isolatedij}) means there are no isolated indices between
$i_a$ and $j_b$, though there may be many indices that are not isolated. If
$i_a$ is not isolated, then either $i_a = j_a$ or $i_a = j_{a + 1}$. In this
case, we consider $i_a$ to be paired with $j_a$ or~$j_{a + 1}$.

\begin{lemma}
  \label{isolatedinpairs}For any admissible state $\mathfrak{s}$, every
  isolated $i_a$ is paired with a unique isolated $j_b$. The pairing
  relationship is a bijection between the $i_a$ and the $j_b$, and if $i_a$
  and $j_b$ are paired, then $i_a \equiv j_b$ modulo $n$.
\end{lemma}

\begin{proof}
  It is obvious that if $i_a$ (resp. $j_b$) is not isolated, then it is paired
  with a unique $j_b$ (resp.~$i_a$). Since these are equal, they are $\equiv
  0$ mod $n$. Therefore we have to consider the isolated vertices. Here we
  make use of the hypothesis $\langle \eta |T| \xi \rangle \neq 0$. Consider
  the state of the model, with the columns labeled:
  \[ \begin{array}{ccccccccccccccc}
       & i_a &  &  &  &  &  & j_a &  &  &  &  &  & j_b & \\
       & - &  & + &  & \cdots &  & - &  & + &  & \cdots &  & + & \\
       + &  & - &  & - &  & - &  & - &  & - &  & - &  & +\\
       & + &  & + &  & \cdots &  & - &  & + &  & \cdots &  & - & 
     \end{array} \]
  The charges at the two horizontal edges labeled $+$ must both be $\equiv 0$
  modulo $n$. This implies that $i_a \equiv j_b$ modulo $n$.
\end{proof}

Let $\psi^{\ast}_k : \mathfrak{F}_m \to \mathfrak{F}_{m + 1}$ denote the
creation operator defined by
\begin{equation}
  \psi^{\ast}_j  (u_{i_m} \wedge u_{i_{m-1}} \wedge u_{i_{m-2}} \wedge \cdots) = u_j
  \wedge u_{i_m} \wedge u_{i_{m-1}} \wedge u_{i_{m-2}} \wedge \cdots \hspace{0.17em}
\end{equation}
and introduce the generating function
\begin{equation}
  \psi^{\ast} (x) = \sum_{j \in \mathbb{Z}} \psi^{\ast}_j x^j
  \hspace{0.17em},
\end{equation}
as well as the operator $\rho^{\ast}_k (z) : \mathfrak{F}_m \to
\mathfrak{F}_{m + 1}$
\begin{equation}
  \rho^{\ast}_k (z) = \psi^{\ast}_k - z \psi^{\ast}_{k - n} \hspace{0.17em} .
\end{equation}

We will use the following consequences of the Baker-Campbell-Hausdorff formula.
If $A$ and $B$ are elements of a Lie algebra such that $[A, B]$ commutes with both
$A$ and $B$, then
\[ e^A e^B = e^{[A, B]} e^B e^A . \]
If $[A, B] = c B$ where $c$ is a constant, then
\begin{equation}
    e^A B e^{-A} = e^c B \, .
\end{equation}

\begin{proposition}
  \label{prop:rho-H}
  With $H=H_+(z) = \sum_{k = 1}^{\infty} \frac{1}{k}  (1 - v^k) z^{nk}
  J_{k}$ and $\psi^*(x)$ as defined above, we have that
  \begin{equation}
    \label{eq:BCH-Hamiltonian} e^H \psi^{\ast} (x) e^{- H} = \frac{1 - x^n
    vz^n}{1 - x^n z^n} \psi^{\ast} (x)
  \end{equation}
  or equivalently that
  \begin{equation}
    \label{eq:rho-H} e^H \rho^{\ast}_k (z^n) = \rho^{\ast}_k  (vz^n) e^H
    \hspace{0.17em} .
  \end{equation}
\end{proposition}

\begin{proof}
  For any $\xi \in \mathfrak{F}_m$ we have from \eqref{eq:J-derivative} that $[J_k, \psi^{\ast}_j]
  (\xi) = J_k (\psi^{\ast}_j (\xi)) - \psi^{\ast}_j (J_k (\xi)) = J_k 
  (u_j \wedge \xi) - u_j \wedge J_k (\xi) = J_k (u_j) \wedge \xi = u_{j - nk}
  \wedge \xi = \psi^{\ast}_{j - nk} (\xi)$ which implies
  \begin{equation}
    [J_k, \psi^{\ast}_j] = \psi^{\ast}_{j - nk} \hspace{0.17em} .
  \end{equation}
  Then,
  \[ [H, \psi^{\ast} (x)] = \sum_{k \geqslant 1} \sum_{j \in \mathbb{Z}}
     \frac{z^{nk} - v^k z^{n k}}{k} x^j \psi^{\ast}_{j - nk} (x) \]
  \[ = \sum_{k \geqslant 1} \frac{z^{nk} - v^k z^{nk}}{k} x^{nk} \psi^{\ast}
     (x) = \log \left( \frac{1 - x^n vz^n}{1 - x^n z^n} \right) \psi^{\ast}
     (x) \]
  from which we obtain (\ref{eq:BCH-Hamiltonian}) using the
  Baker-Campbell-Hausdorff formula. The equivalence of
  (\ref{eq:BCH-Hamiltonian}) and (\ref{eq:rho-H}) follows by comparing
  coefficients for different powers of~$x$.
\end{proof}

We will work now with finite-dimensional wedge spaces $\mathfrak{F}  (k, n -
k, r)$ spanned by vectors
\begin{equation}
  \label{xiaswedge} \xi = u_{i_1} \wedge \cdots \wedge u_{i_r}
\end{equation}
where $k \geqslant i_1 > \cdots > i_r \geqslant i_{k - n}$. Let $\mathfrak{F} 
(k, n - k) = \bigoplus_r \mathfrak{F}  (k, n - k, r)$. We will define
operators $\psi_k^{\ast}$ and $\psi_{k - n}^{\ast} : \mathfrak{F}  (k, n - k,
r) \longrightarrow \mathfrak{F}  (k, n - k, r + 1)$ by
\[ \psi_k^{\ast} (\xi) = u_k \wedge \xi, \qquad \psi_{k - n}^{\ast} (\xi) =
   u_{k - n} \wedge \xi . \]
These operators are analogous to the operators $\psi_k^{\ast}, \psi_{k -
n}^{\ast} : \mathfrak{F}_m \longrightarrow \mathfrak{F}_{m + 1}$ already
defined, and indeed if $\zeta = u_{j_{m - r}} \wedge u_{j_{m - r - 1}} \wedge
\cdots \in \mathfrak{F}_{m - r}$ is such that $k - n > j_{m - r} > j_{m - r -
1} > \cdots$ then $\xi \wedge \zeta$ is naturally in $\mathfrak{F}_m$ and
$\psi_k^{\ast}  (\xi \wedge \zeta) = \psi_k^{\ast} (\xi) \wedge \zeta$ and
similarly for $\psi_{k - n}^{\ast}$. We also define $\rho_k^{\ast} (z) = \psi_k^{\ast} - z \psi_{k - n}^{\ast}$ as
before.

Finally, we define an operator $\hat{T}$ on $\mathfrak{F}  (k, n - k)$. It is
enough to define constants $\langle \eta | \hat{T} | \xi \rangle$ where $\xi
\in \mathfrak{F}  (k, n - k, r)$ and $\eta \in \mathfrak{F}  (k, n - k, r')$.
Let us write $\varepsilon = \varepsilon (\xi) = (\varepsilon_k, \cdots,
\varepsilon_{k - n})$ where the spins $\varepsilon_i = \pm$ and $i = i_1,
\cdots, i_r$ in (\ref{xiaswedge}) are precisely the values where
$\varepsilon_i = -$. Similarly let $\delta = \delta (\eta) = (\delta_k,
\cdots, \delta_{k - n})$ be spins corresponding to~$\eta$. Let
\[ \eta = u_{j_1} \wedge u_{j_2} \wedge \cdots \wedge u_{j_{r'}} . \]
We require $i_1 \geqslant j_1 \geqslant i_2 \geqslant \cdots$ and for this reason
either $r' = r$ or $r' = r - 1$.

Now we define a finite system as follows. We make a grid with $n + 1$ columns
labeled $k, k - 1, \cdots, k - n$ in decreasing order.
\begin{equation}
  \label{shortice} \begin{array}{ccccccccc}
    & \varepsilon_k &  & \varepsilon_{k - 1} &  & \cdots &  & \varepsilon_{k
    - n} & \\
    +^0 &  &  &  &  &  &  &  & \pm^a\\
    & \delta_k &  & \delta_{k - 1} &  & \cdots &  & \delta_{k - n} & 
  \end{array} \quad .
\end{equation}
The boundary conditions at the left and right edge are as follows. At the left
boundary, we always put $+^0$. At the right boundary, there will, for each
row, be a unique decorated spin $\pm^a$ such that the partition function of
this system can have nonzero value. The sign $+$ or $-$ is determined by the
condition that the total number of $-$ spins around the whole boundary is even. Thus
it is $+$ if $r' = r$ and $-$ if $r' = r - 1$. The charge is also determined
by the requirement that there be a (uniquely determined) state $\mathfrak{s}$ with
the given boundary conditions. Then we define $\langle \eta | \hat{T} | \xi
\rangle$ to be the Boltzmann weight of this state, using the weights in
Table~\ref{tab:mweights}.

Now the operator $\hat{T} : \mathfrak{F}  (k, n - k, r) \longrightarrow
\mathfrak{F}  (k, n - k, r) \oplus \mathfrak{F}  (k, n - k, r - 1)$ is defined
by
\[ \hat{T} (\xi) = \sum_{\eta} \langle \eta | \hat{T} | \xi \rangle \eta . \]
\begin{proposition}
  \label{etxirhocompat}Let $\xi$ and $\eta$ be basis vectors of $\mathfrak{F} 
  (k, n - k)$ as above. Then
  \begin{equation}
    \label{etxiformula} \langle \eta | \hat{T} \rho_k^{\ast} (z^n) | \xi
    \rangle = \langle \eta | \rho_k^{\ast} (v z^{n - 1}) \hat{T} | \xi \rangle
    .
  \end{equation}
  Moreover, the spins $\pm a$ that appear on the left- and right-hand sides of this
  calculation are the same (with $a$ determined modulo $n$).
\end{proposition}

We will prove this in Section~\ref{sec:proof}. The meaning of the second assertion is
as follows. Suppose we compute
\[ \langle \eta | \hat{T} \rho_k^{\ast} (z^n) | \xi \rangle . \]
This equals $\langle \eta | \hat{T} | \psi_k^{\ast} \xi \rangle - z^n  \langle
\eta | \hat{T} | \psi_{k - n}^{\ast} \xi \rangle$ and in this computation two
right edge spins $\pm a$ and $\pm b$ will appear. (See (\ref{shortice}).)
Similarly on the other side of the computation, two right edge spins $\pm c$
and $\pm d$ will appear. The assertion is that these four spins are equal in
sign, and $a \equiv b \equiv c \equiv d$ modulo~$n$.

\begin{proposition}
  \label{prop:rho-T}Let $\xi = u_{i_m} \wedge u_{i_{m - 1}} \wedge \cdots \in
  \mathfrak{F}_m$ with $i_m > i_{m - 1} > \ldots$ and let $k > i_m$. Then,
  \begin{equation}
    T \rho^{\ast}_k (z^n) | \xi \rangle = \rho^{\ast}_k  (v z^n) T| \xi
    \rangle \hspace{0.17em} .
  \end{equation}
\end{proposition}

\begin{proof}
  Let $\eta \in \mathfrak{F}_m$. We write
  $\eta = u_{j_m} \wedge u_{j_{m - 1}} \wedge \cdots$ with $j_m > j_{m - 1} > \ldots$.
  Unless $k \geqslant j_m$ it is easy
  to deduce that $\langle \eta |T \psi_k^{\ast} | \xi \rangle$,
  $\langle \eta |T \psi_{k - n}^{\ast} | \xi \rangle$,
  $\langle \eta | \psi_k^{\ast} T| \xi \rangle$ and
  $\langle \eta | \psi_{k - n}^{\ast} T| \xi \rangle$ are all
  zero from Lemma~\ref{interleavinglemma}, and from the fact that if $\xi'$
  does not involve any $u_m$ with $m > k$ then neither does $\psi_k^{\ast}
  \xi'$ or $\psi_{k - n}^{\ast} \xi'$. Therefore it is enough to prove that
  \[ \langle \eta |T \rho_k^{\ast} (z^n) | \xi \rangle = \langle \eta |
     \rho_k^{\ast} (v z^n) T| \xi \rangle \]
  under the assumption that $k \geqslant j_m$.
  
  Let us find $r$ such that $i_r \geqslant k-n > i_{r - 1}$ and write
  $\xi = \xi_1 \wedge \xi_2$ with
  \[ \xi_1 = u_{i_m} \wedge \cdots \wedge u_{i_r}, \qquad \xi_2 = u_{i_{r -
     1}} \wedge u_{i_{r - 2}} \wedge \cdots\;. \]
  Similarly we write $\eta = \eta_1 \wedge \eta_2$ where
  \[ \eta_1 = u_{j_m} \wedge \cdots \wedge u_{j_r'}, \qquad \eta_2 = u_{j_{r' -
     1}} \wedge u_{j_{r' - 2}} \wedge \cdots \]
  and $r'$ is such that $j_{r'}\geqslant k-n >j_{r'-1}$.
  
  Now let $\mathfrak{s}$ be the unique state associated with $\langle \eta |T|
  \psi_k^{\ast} \xi \rangle$. We will cut the partition function to the right
  of the $k - n$ column. Thus we partition the Boltzmann
  weights into those from columns numbered $\geqslant k - n$, and
  those from columns $< k - n$. Since $k\geqslant i_m,j_m$ the spin
  in the horizontal edge to the left of the $k$-th column must be $+^0$.
  Depending on $\xi$ and $\eta$, let $\pm^a$ be the
  decorated spin attached to the horizontal edge to the right of the
  $(k-n)$-th column.  We obtain
  \[ \langle \eta |T \psi_k^{\ast} | \xi \rangle = \langle \eta_1 | \hat{T}
     \psi_k^{\ast} | \xi_1 \rangle \cdot C \]
  where $C$ is the Boltzmann weight of the following state of an (infinite)
  truncated system:
  \[ \begin{array}{cccccc}
       & \varepsilon_{k - n - 1} &  & \varepsilon_{k - n - 2} &  & \cdots\\
       \pm^a &  &  &  &  & \\
       & \delta_{k - n - 1} &  & \delta_{k - n - 2} &  & \cdots
     \end{array} \]
  where $\varepsilon_i = -$ if $i$ is among the indices $i_{r - 1}, i_{r - 2},
  \cdots$ in $\xi_2$ and $\delta_i$ is similarly derived from $\eta_2$. 
  
  Now we similarly have
  \[ \langle \eta |T \psi_{k - n}^{\ast} | \xi \rangle = \langle \eta_1 |
     \hat{T} \psi_{k - n}^{\ast} | \xi_1 \rangle \cdot C, \qquad \langle \eta
     | \psi_k^{\ast} T| \xi \rangle = \langle \eta_1 | \psi_k^{\ast}  \hat{T}
     | \xi_1 \rangle \cdot C, \]
  and
  \[ \langle \eta | \psi_{k - n}^{\ast} T| \xi \rangle = \langle \eta_1 |
     \psi_{k - n}^{\ast}  \hat{T} | \xi_1 \rangle \cdot C, \]
  with the \textit{same} constant $C$ in every case. The fact that the
  constant $C$ is the same in every case follows from the last assertion in
  Proposition~\ref{etxirhocompat}. Hence we can pull out the constant and the
  identity needed follows from (\ref{etxiformula}).
\end{proof}

For an element $\xi = u_{i_m} \wedge u_{i_{m - 1}} \wedge \cdots \in
\mathfrak{F}_m$ with $i_m > i_{m - 1} > \cdots$ we define the degree $\deg
(\xi)$ of $\xi$ as follows
\begin{equation}
  \label{eq:degree} \deg (\xi) = \sum_{r \leqslant m} (i_r - r)
\end{equation}
which we note is positive since $i_r \geqslant r$ for all $r$, and finite
since $i_r = r$ for $r \ll 0$. If $\deg (\xi) = 0$, then $\xi$ is the vacuum
$| m \rangle$ in $\mathfrak{F}_m$.

Using the following lemma we can similarly define the degree of any $\xi =
u_{i_m} \wedge u_{i_{m - 1}} \wedge \cdots \in \mathfrak{F}_m$ even if it is
not normal-ordered.

% \pagebreak

\begin{lemma}
  \label{lem:degree}The degree defined above has the following properties:
  \begin{enumerate}
    \item \label{itm:degree-not-ordered} Suppose $\xi = u_{i_m} \wedge u_{i_{m
    - 1}} \wedge \cdots \in \mathfrak{F}_m$ is not normal-ordered, that is
    $i_r < i_{r - 1}$ for some $r \leq m$. Then writing $\xi$ in terms of the
    basis of $\mathfrak{F}_m$ of normal-ordered wedges, each term has the
    same degree, which equals $\sum_{r \leq m} (i_r - r)$.
    
    \item \label{itm:degree-extra} Let $\xi' = u_{i_{m - 1}} \wedge u_{i_{m -
    2}} \wedge \cdots \in \mathfrak{F}_{m - 1}$ with $i_{m - 1} > i_{m - 2} >
    \cdots$. For any $k$, let $\xi = u_k \wedge \xi' \in \mathfrak{F}_m$ which
    is not necessarily normal-ordered. If $\xi \neq 0$, then,
    \begin{equation}
      \deg (\xi) = (k - m) + \deg (\xi') \hspace{0.17em} .
    \end{equation}
  \end{enumerate}
\end{lemma}

Note that, even for the quantum wedge, if $i_r = i_{r - 1}$ for some $r$, then
$\xi = 0$. However, because of the extra terms in (\ref{eq:quantum-wedge})
compared to the classical ($q = 1$) wedge, if $i_r = i_{r - 2}$ for example, then $\xi$ is
not necessarily zero.

\begin{proof}
  For the first statement we notice that in the
  right-hand side of the quantum wedge (\ref{eq:KMS45}) for $u_j
  \wedge u_i$ with $j < i$, each term is of the form $u_a \wedge u_b$ with $a +
  b = i + j$. Since $\xi$ can be normal-ordered by repeated use of
  (\ref{eq:KMS45}) this proves the first assertion.
  
  The second statement follows from the first by
  letting $i_m = k$:
  \[ \deg (\xi) = \sum_{r \leq m} (i_r - r) = (i_m - m) + \sum_{r \leq m - 1}
     (i_r - r) = (k - m) + \deg (\xi') \hspace{0.17em} . \qedhere \]
\end{proof}

\begin{proof}[Proof of Theorem \ref{thm:eH-equals-T} (Delta ice)]
  We will show, for an arbitrary $\xi = u_{i_m} \wedge u_{i_{m - 1}} \wedge
  \cdots \in \mathfrak{F}_m$ with $i_m > i_{m - 1} > \cdots$ that $e^H \xi = T
  \xi$ using induction over the degree of $\xi$.
  
  The base case, $\deg (\xi) = 0$, is when $| \xi \rangle$ is the vacuum $|m
  \rangle$, for which we have that $J_k |m \rangle = 0$. Thus $e^H |m \rangle
  = |m \rangle$. It is easy to check that $T|m \rangle = |m \rangle$ also, as
  required.
  
  From now on, assume that $\xi$ is not a vacuum, which means that $i_m > m$.
  Let $\xi' = u_{i_{m - 1}} \wedge u_{i_{m - 2}} \wedge \cdots \in
  \mathfrak{F}_{m - 1}$ and $\xi'' = u_{i_m - n} \wedge \xi' \in
  \mathfrak{F}_m$. Then $\xi = u_{i_m} \wedge \xi' = \rho^{\ast}_{i_m} (z^n)
  \xi' + z^n \xi''$. Note that $u_{i_m - n} \wedge \xi'$ is not necessarily
  normal-ordered or nonzero. Using Lemma~\ref{lem:degree} we have that $\deg
  (\xi') = \deg (\xi) - (i_m - m) < \deg (\xi)$ and, if $\xi'' \neq 0$, $\deg
  (\xi'') = \deg (\xi) - n < \deg (\xi)$.
  
  We assume, for $\eta \in \mathfrak{F}$ with $\deg (\eta) < \deg (\xi)$, that
  $e^H \eta = T \eta$ (which also holds for $\eta = \xi'' = 0$). Then, for the
  induction step we have that
  \[ T \xi = T \rho^{\ast}_{i_m} (z^n) \xi' + z^n T \xi'' = T
     \rho^{\ast}_{i_m} (z^n) \xi' + z^n e^H \xi'' \hspace{0.17em} . \]
  Using Proposition~\ref{prop:rho-T} together with the induction hypothesis,
  we have that
  \[ T \rho^{\ast}_{i_m} (z^n) \xi' = \rho^{\ast}_{i_m}  (v z^n) T \xi' =
     \rho^{\ast}_{i_m} (v z^n) e^H \xi' = e^H \rho^{\ast}_{i_m} (z^n) \xi'
     \hspace{0.17em}, \]
  where, in the last step we have also used (\ref{eq:rho-H}) of
  Proposition~\ref{prop:rho-H}. Thus,
  \[ T \xi = e^H  (\rho^{\ast}_{i_m} (z^n) \xi' + z^n \xi'') = e^H \xi
     \hspace{0.17em} . \]

  The statement for Gamma ice is proved in Subsection~\ref{sec:gammaproof}.
\end{proof}

\subsection{Proof of Proposition~\ref{etxirhocompat}}
\label{sec:proof}
Let $\xi = u_{i_1} \wedge \cdots \wedge u_{i_r}$ and $\eta = u_{j_1} \wedge
\cdots \wedge u_{j_{r'}}$ be elements of $\mathfrak{F}  (k, n - k)$ with $i_1
> i_2 > \cdots > i_r$ and $j_1 > \cdots > j_{r'}$. We must show
\begin{equation}
  \label{fourterms} \langle \eta | \hat{T} \psi_k^{\ast} | \xi \rangle - z^n 
  \langle \eta | \hat{T} \psi_{k - n}^{\ast} | \xi \rangle = \langle \eta |
  \psi_k^{\ast}  \hat{T} | \xi \rangle - v z^n  \langle \eta | \psi_{k -
  n}^{\ast}  \hat{T} | \xi \rangle .
\end{equation}
Let $\varepsilon_i$ and $\delta_i$ with $k \geqslant i \geqslant k - n$ be the
spins associated with $\xi$ and $\eta$, so that $\varepsilon_i = -$ if $i =
i_j$ for some $j$, and $\varepsilon_i = +$ otherwise, and similarly for
$\delta_i$.

% \pagebreak

\begin{proposition}
  \label{casesonetwo}Suppose that any one of the four terms in
  (\ref{fourterms}) is nonzero. Then either:
  \begin{enumerate}[label=(\roman*)]
    \item We have $\varepsilon_i = \delta_i$ for $k > i > k - n$; or
    
    \item There is a unique value $s$ with $k > s > k - n$ such that
    $\varepsilon_s = -$ and $\delta_s = +$, and $\varepsilon_i = \delta_i$ for
    $k > i > k - n$, $i \neq s$.
  \end{enumerate}
\end{proposition}

\begin{proof}
  Note that applying $\psi_k^{\ast}$ or $\psi_{k - n}^{\ast}$ to $\xi$ cannot
  affect $\varepsilon_i$ with $k > i > k - n$. In particular, $\psi^{\ast}_{k
  - n} (\xi) = u_{k - n} \wedge u_{i_1} \wedge \cdots \wedge u_{i_r}$ is not
  normal-ordered. However when we use (\ref{eq:quantum-wedge}) to put it in
  normal order, we get
  \begin{equation}
    \label{intergauss} u_{k - n} \wedge u_{i_1} \wedge \cdots \wedge u_{i_r} =
    \pm \left( \prod_{\substack{\varepsilon_i = - \\ k > i > k - n}}
    g (k - i) \right) u_{i_1} \wedge \cdots \wedge u_{i_r}
    \wedge u_{k - n}
  \end{equation}
  where the sign is $+$ if $\varepsilon_k = +$ and $-$ if $\varepsilon_k = -$.
  For this, there are no correction terms because the interchanged vectors are
  of the form $u_a \wedge u_b$ with $|a - b| \leqslant n$.
  
  Therefore each of the four terms in (\ref{fourterms}) is (possibly up to a
  constant such as the one in (\ref{intergauss})) of the form $\langle
  \eta' | \hat{T} | \xi' \rangle$ where $\xi'$ and $\eta'$ correspond to
  sequences $\varepsilon_i'$ and $\delta_i'$ of spins and (for the two terms
  on the left-hand side) $\delta_i' = \delta_i$ for all $k \geqslant i
  \geqslant k - n$ and also $\varepsilon_i' = \varepsilon_i$ except for one of
  the two cases $i = k$ or $i = k - n$. Similarly for the two terms on the
  right-hand side, $\varepsilon_i' = \varepsilon_i$ for all $k \geqslant i
  \geqslant k - n$ and $\delta_i' = \delta_i$ except when $i = k$ or $k - n$.
  Since $\varepsilon_i = \varepsilon_i'$ and $\delta_i = \delta_i'$ for $k > i
  > k - n$, we may replace $\varepsilon_i$ and $\delta_i$ by $\varepsilon_i'$
  and $\delta_i'$ in the statement of the proposition.
  
  Fixing one of these four cases, let $\xi' = u_{i_1'} \wedge u_{i_2'} \wedge \cdots$ and
  $\eta' = u_{j_1'} \wedge u_{j_2'} \wedge \cdots$. Under the assumption that
  $\langle \eta' | \hat{T} | \xi' \rangle \neq 0$, analogs of
  Lemmas~\ref{interleavinglemma} and~\ref{isolatedinpairs} are true. The
  analog of Lemma~\ref{interleavinglemma} means that $i_1' \geqslant j_1'
  \geqslant i_2' \geqslant j_2' \geqslant \cdots$.
  
  Moreover, the proof of Lemma~\ref{isolatedinpairs} will show that there is
  at most one isolated index in the interval $k > i > k - n$. We recall that
  an index $s$ is \textit{isolated} if $\varepsilon_s \neq \delta_s$. If $k
  > s > k - n$, this is clearly equivalent to $\varepsilon_s' \neq \delta_s'$.
  As in Lemma~\ref{isolatedinpairs} isolated indices come in pairs separated
  by a multiple of $n$. Thus if there are isolated indices, we must have $i_1'
  = j_1'$, $i_2' = j_2'$, up to the first isolated index, $i_m' > j_m'$. Then
  the next isolated index would have to be $\leqslant i_s' - n$, but this is
  outside of the considered interval. Let $s = i_m'$. Then $\varepsilon_s = \varepsilon_s' =
  -$, while $\delta_s = \delta_s' = +$, and there are no other isolated
  indices.
\end{proof}

So there are two types of cases we have to consider, depending on whether we are
in Case~(i) or Case~(ii) of Proposition~\ref{casesonetwo}. With each of these cases we have 16 subcases
depending on the values of $\varepsilon_k, \delta_k, \varepsilon_{k - n}$ and
$\delta_{k - n}$.

\begin{remark}
  It is possible to argue more efficiently and only check half these 32 cases,
  namely those in which $\varepsilon_k=+$. This is because in
  Proposition~\ref{prop:rho-T} we have $k > i_m$, and $i_m$ denotes the first minus sign
  of $\xi$. For completeness we included all 32 cases in
  Tables~\ref{tab:caseisubcases} and~\ref{tab:caseiisubcases}.
\end{remark}

For Case (i), let us denote
\[ G = \prod_{\substack{\varepsilon_i = - \\ k > i > k - n}} g (k - i) . \]
Case (i), subcase: $(\varepsilon_k, \delta_k, \varepsilon_{k - n}, \delta_{k -
n}) = (+, +, +, +$). We observe that $\langle \eta | \rho_k^{\ast}(v z^n) 
\hat{T} | \xi \rangle = 0$ since there is no way a component $\eta$ of $
\rho_k^{\ast}(v z^n)  \hat{T} | \xi \rangle$ can have both $\delta_k =
\delta_{k - n} = +$. So we must show that $\langle \eta | \hat{T}
\rho_k^{\ast}(z^n) | \xi \rangle = 0$. This has two terms, which will cancel.
First $\langle \eta | \hat{T} \psi_k^{\ast} | \xi \rangle$ is the Boltzmann
weight of the state
\[ \begin{array}{|l|}
     \hline
     \begin{array}{ccccccccccc}
       & - &  & \varepsilon_{k - 1} &  & \cdots &  & \varepsilon_{k - n + 1}
       &  & + & \\
       + &  & - &  & - &  & - &  & -^0 &  & -^1\\
       & + &  & \varepsilon_{k - 1} &  & \cdots &  & \varepsilon_{k - n + 1}
       &  & + & 
     \end{array}\\
     \hline
   \end{array} \]
that is, $Gz^n$, where the product is over $r$ patterns of type
$\tt{a}_2$ and $n - r$ of type $\tt{b}_2$. The second
term is $- z^n  \langle \eta | \hat{T} \psi_{k - n}^{\ast} | \xi \rangle$.
This equals $- z^n G$ times the Boltzmann weight of the state
\[ \begin{array}{|l|}
     \hline
     \begin{array}{ccccccccccc}
       & + &  & \varepsilon_{k - 1} &  & \cdots &  & \varepsilon_{k - n + 1}
       &  & - & \\
       + &  & + &  & + &  & + &  & +^0 &  & -^1\\
       & + &  & \varepsilon_{k - 1} &  & \cdots &  & \varepsilon_{k - n + 1}
       &  & + & 
     \end{array}\\
     \hline
   \end{array} \hspace{0.27em} . \]
Here the factor of $G$ comes from (\ref{intergauss}). The Boltzmann weight of
the last state is $1$, so the two terms cancel and the proposition is true in
this case.

To summarize, there are two ways that a factor of $G$ can appear. One is
through (\ref{intergauss}), and the other is through the Boltzmann weight of a
state. There are 16 subcases for Case~(i) and these are summarized in
Table~\ref{tab:caseisubcases}. It is easy to see that in all these cases the last
assertion of Proposition~\ref{etxirhocompat} (about the identity of the
decorated spins appearing at the right edges of the states contributing to the
nonzero terms in any subcase) is satisfied.

\begin{table} 
\[ \begin{array}{|c|cccc|}
     \hline
     (\varepsilon_k, \varepsilon_{k - n}, \delta_k, \delta_{k - n}) & \langle
     \eta |\hat{T} \psi_k^{\ast} | \xi \rangle & - z^n \langle \eta |\hat{T} \psi_{k -
     n}^{\ast} | \xi \rangle & \langle \eta | \psi_k^{\ast} \hat{T}| \xi \rangle & -
     v z^{n} \langle \eta | \psi_{k - n}^{\ast} \hat{T}| \xi \rangle\\
     \hline
     (+ + + +) & z^n G & - z^n G & 0 & 0\\
     (+ + + -) & G z^{n} (1-v)z & - z^n G & 0 & - v z^{n+1} G\\
     (+ + - +) & 1 & 0 & 1 & 0\\
     (+ - + -) & - G z^{n - 1} & 0 & 0 & - G z^{n - 1}\\
     (+ + - +) & 1 & 0 & 1 & 0\\
     (+ + - -) & 0 & 0 & 0 & 0\\
     (+ - - +) & 1 & 0 & 1 & 0\\
     (+ - - -) & 1 & 0 & 1 & 0\\
     (- + + +) & 0 & 0 & 0 & 0\\
     (- + + -) & 0 & - v z^{2 n} G^2 & 0 & - v z^{2 n} G^2\\
     (- - + +) & 0 & 0 & 0 & 0\\
     (- - + -) & 0 & 0 & 0 & 0\\
     (- + - +) & 0 & z^n G & z^n G & 0\\
     (- + - -) & 0 & z^n G & (1-v) z^{n} G & v z^{n} G\\
     (- - - +) & 0 & 0 & 0 & 0\\
     (- - - -) & 0 & 0 & - v z^{n} G & v z^{n} G\\
     \hline
   \end{array} \] 
   \caption{\label{tab:caseisubcases}Case (i) subcases, confirming (\ref{fourterms}).}
\end{table}

We now turn to Case~(ii). Let us again do one subcase completely, then
summarize all cases in a table. Let us consider the subcase where
$(\varepsilon_k, \varepsilon_{k - n}, \delta_k, \delta_{k - n}) = (+, +, +,
-$). We do not need to consider the contributions of $\psi_k^{\ast}$ to either
the left- or the right-hand side since these would involve an illegal pattern
in the $s$ column. On the other hand
\[ - z^n \langle \eta |\hat{T} \psi^{\ast}_{k - n} | \xi \rangle = (- z^n) \left[
   \prod_{\substack{k - n + 1 \leqslant i \leqslant k - 1\\
     \varepsilon_i = - 1}}
    g (k - i) \right] Z \]
where $Z$ is the Boltzmann weight of the state
\[ \begin{array}{|l|}
     \hline
     \begin{array}{ccccccccccccc}
       & + &  & \varepsilon_{k - 1} &  & \cdots &  & - &  & \cdots &  & - &
       \\
       + &  & + &  &  &  & + & (s) & - &  & - &  & -^{s + 1 - k + n}\\
       & + &  & \varepsilon_{k - 1} &  & \cdots &  & + &  & \cdots &  & - & 
     \end{array}\\
     \hline
   \end{array} \hspace{0.27em} . \]
The product in brackets comes from (\ref{intergauss}). We have
\[ Z = \left[ \prod_{\substack{
     k - n + 1 \leqslant i \leqslant s - 1\\
     \varepsilon_i = - 1
   }} g (s - i) \right] z^{s - k + n} g (s - k) . \]

We may combine two factors using the identity $g (k - s) g (s - k) = v$
and so
\[ - z^n \langle \eta |\hat{T} \psi^{\ast}_{k - n} | \xi \rangle = - \left[
   \prod_{\substack{
     k - n + 1 \leqslant i \leqslant k - 1\\
     \varepsilon_i = - 1\\
     i \neq s
     }} g (k - i) \right] \left[ \prod_{\substack{
     k - n + 1 \leqslant i \leqslant s - 1\\
     \varepsilon_i = - 1
   }} g (s - i) \right] v z^{s - k + 2 n}. \]
On the other side of the equation,
\[ - vz^{n} \langle \eta | \psi^{\ast}_{k - n} \hat{T}| \xi \rangle = (- v z^n
    ) \left[ \prod_{\substack{
     k - n + 1 \leqslant i \leqslant k - 1\\
     \varepsilon_i = - 1\\
     i \neq s
     }} g (k - i) \right] Z' \]
where $Z'$ is the Boltzmann weight of the state
\[ \begin{array}{|l|}
     \hline
     \begin{array}{ccccccccccccc}
       & + &  & \varepsilon_{k - 1} &  & \cdots &  & - &  & \cdots &  & + &
       \\
       + &  & + &  &  &  & + & (s) & - &  & - &  & -^{s + 1 - k + n}\\
       & + &  & \varepsilon_{k - 1} &  & \cdots &  & + &  & \cdots &  & + & 
     \end{array}\\
     \hline
   \end{array} \hspace{0.27em}. \]
That is,
\[ Z' = \left[ \prod_{\substack{
     k - n + 1 \leqslant i \leqslant s - 1\\
     \varepsilon_i = - 1
   }} g (s - i) \right] z^{s - k + n} . \]
We see that in this case:
\[ \langle \eta | \hat{T} \rho^{\ast}_k (z^n) | \xi \rangle = - z^n  \langle
   \eta | \hat{T} \psi^{\ast}_{k - n} | \xi \rangle = - v z^n \langle \eta |
   \psi^{\ast}_{k - n}  \hat{T} | \xi \rangle = \langle \eta | \rho^{\ast}_k
   (v z^n) \hat{T} | \xi \rangle. \]

Now let us define
\[G'=\prod_{\substack{
     k - n + 1 \leqslant i \leqslant k - 1\\
     \varepsilon_i = - 1\\
     i \neq s
     }} g (k - i),\qquad
  G''= \prod_{\substack{
     k - n + 1 \leqslant i \leqslant s - 1\\
     \varepsilon_i = - 1
   }} g (s - i)\;.
\]
We summarize the Case (ii) subcases in Table~\ref{tab:caseiisubcases}. As
in Case (i) it is easy to verify the last assertion of Proposition~\ref{etxirhocompat} regarding the decorated spins at the right edge, and the first assertion is verified in every subcase by Table~\ref{tab:caseiisubcases}. Thus Proposition~\ref{etxirhocompat} is now proved.

\begin{table}
\[ \begin{array}{|c|cccc|}
     \hline
     (\varepsilon_k, \varepsilon_{k - n}, \delta_k, \delta_{k - n}) & \langle
     \eta |\hat{T} \psi_k^{\ast} | \xi \rangle & - z^n \langle \eta |\hat{T} \psi_{k -
     n}^{\ast} | \xi \rangle & \langle \eta | \psi_k^{\ast} \hat{T}| \xi \rangle & -
     vz^n \langle \eta | \psi_{k - n}^{\ast} \hat{T}| \xi \rangle\\
     \hline
     (+ + + +) & 0 & 0 & 0 & 0\\
     (+ + + -) & 0 & \scriptstyle - G' G'' v z^{s - k + 2 n} & 0 &
     \scriptstyle - G' G'' v z^{s - k + 2 n }\\
     (+ + - +) & G'' z^{s - k + n} & 0 & G'' z^{s - k + n} & 0\\
     (+ - + -) & 0 & 0 & 0 & 0\\
     (+ + - +) & G'' z^{s - k + n} & 0 & G'' z^{s - k + n} & 0\\
     (+ + - -) & \scriptstyle G'' (1-v) z^{s - k + n}  & 0 & \scriptstyle G''
     (1-v) z^{s - k + n} & 0\\
     (+ - - +) & 0 & 0 & 0 & 0\\
     (+ - - -) & \scriptstyle G'' z^{s - k + n} g (s - k ) & 0 & \scriptstyle G'' z^{s - k + n} g (s - k) & 0\\
     (- + + +) & 0 & 0 & 0 & 0\\
     (- + + -) & 0 & 0 & 0 & 0\\
     (- - + +) & 0 & 0 & 0 & 0\\
     (- - + -) & 0 & 0 & 0 & 0\\
     (- + - +) & 0 & 0 & 0 & 0\\
     (- + - -) & 0 & \scriptstyle G' G'' v z^{s - k + 2 n}  & 0 & \scriptstyle
     G' G'' v z^{s - k + 2 n}\\
     (- - - +) & 0 & 0 & 0 & 0\\
     (- - - -) & 0 & 0 & 0 & 0\\
     \hline
   \end{array} \]
  \caption{\label{tab:caseiisubcases}Case (ii) subcases, confirming (\ref{fourterms}).}
\end{table}

\subsection{\label{sec:gammaproof}Gamma ice}
We will deduce the second identity in (\ref{ehtform}) for Gamma ice
from the first, which is  already proved. 
If $T$ is an operator on $\mathfrak{F}$ we define its adjoint $T^{\dagger}$ by
the formula
\[ \langle T^{\dagger} \eta | \xi \rangle = \langle \eta |T \xi \rangle . \]

In the following proof, we will assume that the parameter $v$ is real, and
moreover we will assume that the conjugate of $g (a)$ is $g (- a)$. In our applications
to Whittaker functions, $g$ is a Gauss sum, $| g (a) | = \sqrt{v}$, the
reciprocal of the square root of the residue cardinality. (See Remarks~1 and~2
in {\cite{BrubakerBuciumasBump}}.) Then $g (a)$ and $g (- a)$ are complex
conjugates since $g (a) g (- a) = v$.

Since our result is essentially an algebraic identity, if
we prove it under the restriction that $v$ is real and $g(a)$, $g(-a)$
are complex conjugates, it will follow in general.
Alternatively, we could take the $g(a)$ to be indeterminates in an 
algebra over $\mathbb{C}$, with an involution
that maps $g (a)$ to $g (- a)$.

\begin{proposition}
  The adjoint of $T_{\Delta} (z)$ is $T_{\Gamma} \left( 1 / \overline{z}
  \right)$.
\end{proposition}

\begin{proof}
  We must check the identity
  \[ \langle T_{\Gamma} \left( 1 / \overline{z} \right) \eta | \xi \rangle =
     \langle \eta |T_{\Delta} (z) \xi \rangle . \]
  We will write this
  \[ \overline{\langle \xi |T_{\Gamma} \left( 1 / \overline{z} \right) \eta
     \rangle} = \langle \eta |T_{\Delta} (z) \xi \rangle . \]
  We may check this for normal-ordered $\xi, \eta \in \mathfrak{F}_m$. Let
  \[ \xi = u_{i_m} \wedge u_{i_{m - 1}} \wedge \cdots, \hspace{2em} \eta =
     u_{j_m} \wedge u_{j_{m - 1}} \wedge \cdots . \]
  Both sides vanish unless
  \[ i_m \geqslant j_m \geqslant i_{m - 1} \geqslant j_{m - 1} \geqslant
     \cdots, \]
  which we assume. Now $\langle \eta |T_{\Delta} (z) \xi \rangle$ is the
  partition function of a system with a unique state, with $-$ in the top
  (resp. bottom) vertical edges in the columns $i_m$ (resp.~$j_m$) and $+$
  elsewhere. Similarly $\langle \xi |T_{\Gamma} (z^{- 1}) \eta \rangle$ is the
  partition function of a system with the top and bottom vertical edges
  reversed. We may obtain its unique state by taking the state of the $\langle
  \eta |T_{\Delta} (z) \xi \rangle$ system, and replacing each horizontal
  decorated spin $+^0$ by $-^0$, or $-^a$ by $+^{- a}$. Now an examination of
  Table~\ref{tab:mweights} shows that this operation interchanges $\texttt{a}_1$
  patterns with $\texttt{b}_2$ patterns, and similarly $\texttt{a}_2
  \Leftrightarrow \texttt{b}_1$, $\texttt{c}_1 \Leftrightarrow \texttt{c}_2$.
  Remembering that the $\texttt{c}_1$ and $\texttt{c}_2$ patterns occur in pairs,
  we see that $\langle \xi |T_{\Gamma} (z^{- 1}) \eta_{} \rangle$ is obtained
  from $\langle \eta |T_{\Delta} (z) \xi \rangle$ by replacing $g (a)$ by $g
  (- a)$. If we further replace $z$ by its complex conjugate, we see that
  $\langle \xi |T_{\Gamma} \left( 1 / \overline{z} \right) \eta_{} \rangle$
  and $\langle \eta |T_{\Delta} (z) \xi \rangle$ are complex conjugates, as
  required.
\end{proof}

\begin{proposition}
  \label{jkadjoint}
  The adjoint of $J_k$ is $J_{- k}$ if $k \neq 0$.
\end{proposition}

\begin{proof}
  See~\cite{LLTRibbon}, remark after (21) on page 1055. This point
  is explained in more detail in~\cite{LamThesis}, Section~3.3
  (where the inner product is introduced) and Section~4.1.1,
  making use of results of both~\cite{KMS} and~\cite{LLTRibbon}
  in the context of a general Boson-Fermion correspondence.
\end{proof}

\begin{proof}[Proof of Theorem \ref{thm:eH-equals-T} (Gamma ice)]
  We will prove that $T_{\Gamma} (z) = e^{H_- (z)}$.
  Because $J_k$ and $J_{- k}$ are adjoints, by (\ref{hpmdef})
  \[ H_- (z) = H_+ \left( 1 / \overline{z} \right)^{\dagger} . \]
  Exponentiating then gives
  \[ \exp\bigl(H_- (z)\bigl) = \exp \bigl( H_+ \left( 1 / \overline{z} \bigr)
     \right)^{\dagger} = T_{\Delta} \bigl( 1 / \overline{z} \bigr)^{\dagger}
     = T_{\Gamma} (z) . \qedhere \]
\end{proof}

\bigbreak
\section{LLT and Metaplectic Symmetric Functions\label{LLTsection}}
\renewcommand\L{{L}}

The quantum Fock space of Kashiwara, Miwa and Stern, which underlies our
results, is also fundamental in the theory of LLT~\cite{LLTRibbon} or ribbon
symmetric functions. In this section, inspired by ideas from Lam~\cite{LamBoson}, we will show how the
LLT polynomials can be written in the form
\begin{equation}
  \label{voforllt}
  \mathcal{G}_{\lambda/\mu}^n(\mathbf{z})=\langle\mu|e^{\L_+(\mathbf{z})}|\lambda\rangle,
\end{equation}
where $\mathbf{z} = (z_1, \cdots, z_r)$ and
\begin{equation}
    \label{eq:H0}
    \L_+ (\mathbf{z}) = \sum_{k = 1}^{\infty} \frac{1}{k} p_k (\mathbf{z}) J_k.
\end{equation}
(We are using the notation (\ref{partitionnot}) for basis vectors of
$\mathfrak{F}_m$, and we may fix $m=0$ in this section.)

\begin{remark}
   \label{specialgremark}
   As we will prove, the polynomials (\ref{voforllt}) coincide with the
   LLT or ribbon symmetric polynomials provided we take
   \begin{equation}
   \label{specialgforllt}
    g(a)=\left\{\begin{array}{cc} -v&\text{if $a\equiv 0$ mod $n$;}\\
    -\sqrt{v}&\text{otherwise.}
     \end{array}\right.
   \end{equation}
   If $g$ is a more general function satisfying
   Assumption~\ref{assumptiong(a)}, then the results of this section
   will remain valid, but $\mathcal{G}_{\lambda/\mu}$ will
   be a \textit{generalization} of the LLT polynomials that
   are in the literature.
\end{remark}

The operator $\L_+(\mathbf{z})$ is similar to the operator $H_+$ defined in (\ref{hpmdef}),
that appears in our main theorem.
Indeed, in Definition 29 of~\cite{LamRibbon}, Lam defined a super generalization
$\mathcal{G}_{\lambda/\mu}^n(\mathbf{z}|\mathbf{w})$ of
the LLT polynomials, and we will prove that
\[\langle\mu|e^{H_+(\mathbf{z})}|\lambda\rangle=
\mathcal{G}_{\lambda/\mu}^n(\mathbf{z}^n|v\,\mathbf{z}^n),\]
where $H_+(\mathbf{z}) = \sum_{i=1}^r H_+(z_i)$.
For this statement we are omitting the Drinfeld twisting
which introduces the ``Gauss sums'' $g(a)$ into the
definition of the Fock space. If we include the Drinfeld
twisting, then we would obtain a generalization of the
LLT polynomials, and a similar statement would be true.

\begin{remark}
We do not know a statement generalizing
our Theorem~\ref{thm:eH-equals-T} that would express the
half-vertex operator $e^{\L_+}$ that appears in (\ref{voforllt})
as a row transfer matrix. This is available only in the special case
of the supersymmetric LLT polynomials with $\mathbf{w}=v\mathbf{z}$.
\end{remark}

Let $J_1, J_2, \cdots$ be independent commuting variables. (Eventually we will
specialize them to operators on $\mathfrak{F}_0$ as before, but for our first
result this is not needed.) 
Following the definitions in Section 3 of~\cite{LamBoson}, let
\[ u_k = \sum_{\lambda \vdash k} z_{\lambda}^{- 1} J_{\lambda}, \]
where if $\lambda = (1^{m_1} 2^{m_2} \cdots)$ is a partition of $k$ then
$z_{\lambda} = \prod_i (i^{m_i} m_i !)$ and $J_{\lambda} = J_{\lambda_1}
J_{\lambda_2} \cdots$. Also, if $\lambda$ is a partition define
$\varepsilon_{\lambda} = (- 1)^{| \lambda | - \ell (\lambda)}$ and define
\[ \widetilde{u}_k = \sum_{\lambda \vdash k} z_{\lambda}^{- 1}
   \varepsilon_{\lambda} J_{\lambda} . \]     
\begin{proposition}
  \label{operatorequality}We have
  \begin{equation}
    \label{ehoequ} e^{\L_+ (\mathbf{z})} = \sum_{\nu_1 = 0}^{\infty} \cdots
    \sum_{\nu_r = 0}^{\infty} z_1^{\nu_1} \cdots z_r^{\nu_r}  \hspace{0.17em}
    u_{\nu_r} \cdots u_{\nu_1}
  \end{equation}
  and
  \begin{equation}
    \label{ehoalt} e^{-\L_+(-\mathbf{z})} = \sum_{\nu_1 = 0}^{\infty}
    \cdots \sum_{\nu_r = 0}^{\infty} z_1^{\nu_1} \cdots z_r^{\nu_r} 
    \hspace{0.17em} \widetilde{u}_{\nu_r} \cdots \widetilde{u}_{\nu_1} .
  \end{equation}
\end{proposition}

\begin{proof}
  Let $\Lambda = \Lambda (\mathbf{z})$ be the ring of symmetric functions in
  variables $z_1, z_2, \cdots$ over $\mathbb{Q}$. Let $\Lambda (\mathbf{w})$
  be another copy of $\Lambda$, in variables $w_1, w_2,
  \cdots$. We will use the notation of {\cite{MacdonaldHall}} for symmetric
  functions: $p_k (\mathbf{z})$, $h_k (\mathbf{z})$, $e_k (\mathbf{z})$ will
  denote the power sum, complete and elementary symmetric functions, with
  $p_{\lambda} (\mathbf{z}) = \prod p_{\lambda_i} (\mathbf{z})$, $h_{\lambda}
  (\mathbf{z}) = \prod h_{\lambda_i} (\mathbf{z})$, and $m_{\lambda}$ will be
  the monomial symmetric functions.
    
  Remembering that the $u_k$ commute, we may rearrange the factors $u_{\nu_r},
  \cdots, u_{\nu_1}$ so that $\nu_r \geqslant \nu_{r - 1} \geqslant \cdots
  \geqslant \nu_1$ and rewrite the right-hand side of (\ref{ehoequ}) as
  \begin{equation}
    \label{ehorearr} \sum_{\nu} m_{\nu} (\mathbf{z}) u_{\nu_r} \cdots
    u_{\nu_1}
  \end{equation}
  where now the sum is over partitions (of length $\leqslant r$).
  
  In the ring $\Lambda (\mathbf{z}) \otimes
  \Lambda (\mathbf{w})$, we have the identity
  \[ \sum_{\lambda} z_{\lambda}^{- 1} p_{\lambda} (\mathbf{z}) p_{\lambda}
     (\mathbf{w}) = \prod_{i, j} (1 - z_i w_j)^{- 1} = \sum_{\nu} m_{\nu}
     (\mathbf{z}) h_{\nu} (\mathbf{w}) . \]
  This is proved in Macdonald~{\cite{MacdonaldHall}} Section I.4. Now
  \begin{equation}
    \label{ehxyform} \exp \left( \sum_{k = 1}^{\infty} \frac{1}{k} p_k
    (\mathbf{z}) p_k (\mathbf{w}) \right) = \prod_{i, j} (1 - z_i w_j)^{- 1} =
    \sum_{\lambda} m_{\lambda} (\mathbf{z}) h_{\lambda} (\mathbf{w}) .
  \end{equation}
  Indeed,
  \[ - \log (1 - z_i w_j) = \sum_{k = 1}^{\infty} \frac{(z_i w_j)^k}{k} . \]
  Summing over $i, j$ and exponentiating gives (\ref{ehxyform}). Now we
  specialize $p_k (\mathbf{w}) \mapsto J_k$. Then $h_k \mapsto u_k$ since by
  Macdonald~{\cite{MacdonaldHall}} (I.2.14)
  \begin{equation}
    \label{hintermsofp} h_k = \sum_{\lambda \vdash k} z_{\lambda}^{- 1}
    p_{\lambda} .
  \end{equation}
  Thus specializing (\ref{ehxyform}) gives (\ref{ehoequ}). The identity
  (\ref{ehoalt}) follows similarly from the identity
  \[ \exp \left( - \sum_{k = 1}^{\infty} \frac{1}{k} p_k (\mathbf{z}) p_k (-
     \mathbf{w}) \right) = \prod_{i, j} (1 + z_i w_j) = \sum_{\lambda}
     m_{\lambda} (\mathbf{z}) e_{\lambda} (\mathbf{w}) \]
  which follows from (\ref{ehxyform}) on applying the involution in $\Lambda
  (\mathbf{w})$. See {\cite{MacdonaldHall}}~Section~I.2. Under the
  specialization $p_k (\mathbf{w}) \mapsto J_k$ we get
  $e_k \mapsto \widetilde{u}_k$ because
  \[ e_k = \sum_{\lambda \vdash k} z_{\lambda}^{- 1} \varepsilon_{\lambda}
     p_{\lambda} . \]
  This follows from (\ref{hintermsofp}) by applying the involution using
  {\cite{MacdonaldHall}} equation~(I.2.13).
\end{proof}

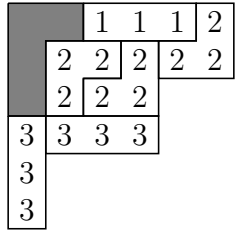
\begin{figure}
  \[\begin{tikzpicture}[scale=.5]
  \draw[fill=gray,semithick] (0,6)--(2,6)--(2,5)--(1,5)--(1,3)--(0,3)--(0,6);
  \draw[semithick] (0,3)--(1,3)--(1,0)--(0,0)--(0,3);
  \node at (0.5,2.5) {$3$};
  \node at (0.5,1.5) {$3$};
  \node at (0.5,0.5) {$3$};
  \draw[semithick] (1,5)--(3,5)--(3,4)--(2,4)--(2,3)--(1,3)--(1,5);
  \node at (1.5,4.5) {$2$};
  \node at (1.5,3.5) {$2$};
  \node at (2.5,4.5) {$2$};
  \draw[semithick] (2,6)--(5,6)--(5,5)--(2,5)--(2,6);
  \node at (2.5,5.5) {$1$};
  \node at (3.5,5.5) {$1$};
  \node at (4.5,5.5) {$1$};
  \draw[semithick] (2,4)--(3,4)--(3,5)--(4,5)--(4,3)--(2,3)--(2,4);
  \node at (2.5,3.5) {$2$};
  \node at (3.5,3.5) {$2$};
  \node at (3.5,4.5) {$2$};
  \draw[semithick] (4,5)--(5,5)--(5,6)--(6,6)--(6,4)--(4,4)--(4,5);
  \node at (4.5,4.5) {$2$};
  \node at (5.5,4.5) {$2$};
  \node at (5.5,5.5) {$2$};
  \draw[semithick] (1,3)--(4,3)--(4,2)--(1,2)--(1,3);
  \node at (1.5,2.5) {$3$};
  \node at (2.5,2.5) {$3$};
  \node at (3.5,2.5) {$3$};
  \end{tikzpicture}\]
  \caption{A $3$-ribbon tableau with spin $5$ and weight $(1,3,2)$.}
  \label{fig:ribbon}
\end{figure}

We recall from~{\cite{LLTRibbon,LamRibbon,LamBoson}} that an
\textit{$n$-ribbon} is a skew partition $\lambda / \mu$ of size $n$ that
is connected and does not contain any $2 \times 2$ block. (Here we
are identifying the skew partition with its Young diagram.)
The \textit{spin} of an $n$-ribbon is its height in columns, minus 1. A
\textit{horizontal $n$-ribbon strip} is a skew shape $\lambda / \mu$ that can
be decomposed into disjoint $n$-ribbons, each of which has its top-right most
box adjacent to $\mu$, or else its top-right most box lies in the first line.
The \textit{spin} $s(\lambda / \mu)$ of $\lambda / \mu$ is then the sum of the spins of its
constituent $n$-ribbons. Thus we are following~{\cite{LamRibbon}} in our
definition of spin, not~{\cite{LLTRibbon}} who define the spin to be half
$s(\lambda / \mu)$. See Figure~\ref{fig:ribbon} for an example
illustrating the concepts of $n$-ribbon and horizontal $n$-ribbon strip.

An \textit{$n$-ribbon skew tableau $T$} of shape $\lambda / \mu$ is a
sequence of partitions
\begin{equation}
  \label{nribbontabdef} \mu = \alpha^0 \subset \alpha^1 \subset \cdots \subset
  \alpha^r = \lambda,
\end{equation}
where $\alpha^{i + 1} / \alpha^i$ is a horizontal $n$-ribbon strip. We may
associate with such data a tableau in which the strip
$\alpha^{i + 1} / \alpha^i$ is filled with $i$'s. The weight $\nu = \wt(T)$ will then be
$(\nu_1, \cdots, \nu_r)$ where $\nu_i$ is $\alpha^{i + 1} / \alpha^i$ divided by~$n$.

Now we define the \textit{LLT} or \textit{ribbon symmetric function}
\[ \mathcal{G}_{\lambda / \mu}^n (\mathbf{z}) = \mathcal{G}_{\lambda / \mu}^n (\mathbf{z}; q) = \sum_T q^{s (T)}
   \mathbf{z}^{\wt (T)}, \]
where the sum is over $n$-ribbon skew tableaux of shape $\lambda / \mu$. 
(Here $v=q^2$.) This
is consistent with the notation in~{\cite{LamRibbon}} but differs from the
notation in~{\cite{LLTRibbon}}. 

Let us regard $J_k$ as in prior sections to be an operator on the quantum Fock
space $\mathfrak{F}_0$. If $\lambda$ is a partition, let $| \lambda \rangle$ denote the element
$u_{\lambda_1} \wedge u_{\lambda_2 - 1} \wedge \cdots$ of $\mathfrak{F}_0$.
If $\lambda$ is the empty partition, we will instead use $|0 \rangle$ to denote
the vacuum. Consistent with our earlier notation, we will denote by $\langle
\mu |e^{\L_+ (\mathbf{z})} | \lambda \rangle$ the coefficient of $| \mu
\rangle$ in $e^{\L_+ (\mathbf{z})} | \lambda \rangle$, where we now regard
$e^{\L_+ (\mathbf{z})}$ as an operator on~$\mathfrak{F}_0$.

Following {\cite{LLTRibbon,LamRibbon}} we define an operator $\mathcal{U}_k$
on $\mathfrak{F}_0$ by
\[ \mathcal{U}_k | \lambda \rangle = \sum_{\substack{\text{$\lambda / \mu$ a horizontal
   $n$-ribbon strip}\\|\lambda/\mu|=n k}} q^{s (\lambda / \mu)}\,| \mu \rangle, \]
where the sum is over $\mu \subset \lambda$ such that $\lambda / \mu$ is a
horizontal $n$-ribbon strip of size $nk$. Similarly let
\[ \widetilde{\mathcal{U}}_k | \lambda \rangle = \sum_{\substack{\text{$\lambda / \mu$ a
   vertical $n$-ribbon strip}\\|\lambda/\mu|=nk}} q^{s (\lambda / \mu)} | \mu \rangle. \]
(Vertical $n$-ribbon strips are defined similarly to horizontal ones.)

We note that the notation in {\cite{LLTRibbon}} differs from that in
{\cite{LamRibbon}} (and also {\cite{KMS}}) by the transformation $q \mapsto -
q^{- 1}$. Our notation is consistent with {\cite{LamRibbon}}.

There is a homomorphism $\psi$ from the ring $\Lambda$ of symmetric functions
to the ring of $U_q (\widehat{\mathfrak{sl}}_n)$-module endomorphisms
of $\mathfrak{F}_0$. This is the map that sends a symmetric polynomial $f$
to the endomorphism $f(y_1^{-1},y_2^{-1},\cdots)$ where the $y_i$ are
as in Section~\ref{sec:aha}. If $s_\lambda$ is a Schur polynomial,
the endomorphisms $\psi(s_\lambda)$ were used in \cite{LeclercThibon}
in an analog of the Steinberg tensor product theorem for $\mathfrak{F}$.
See also~\cite{LaniniRam}.

By Theorems~3 and 5 of~{\cite{LamRibbon}} (following Leclerc and Thibon~\cite{LeclercThibon})
\[ \psi (h_k) =\mathcal{U}_k, \qquad \psi (e_k) =
\widetilde{\mathcal{U}}_k,\qquad \psi (p_k) = J_k . \]
Thus $u_k$ is an element of the abstract polynomial ring
generated by $J_1,J_2,\cdots$, while $\mathcal{U}_k$ is
an endomorphism of $\mathfrak{F}_0$ that corresponds to
$u_k$ under the action of the $J_k$ on $\mathfrak{F}_0$.

\begin{theorem}
  \label{thmlltasvo}
  The polynomial $\mathcal{G}_{\lambda / \mu}^n$ is symmetric and
  \begin{equation}
    \label{lltasvo}
    \mathcal{G}_{\lambda / \mu}^n (\mathbf{z} ; q) = \langle \mu |e^{\L_+
    (\mathbf{z})} | \lambda \rangle .
  \end{equation}
\end{theorem}

\begin{proof}
  By Proposition~\ref{operatorequality} the right-hand
  side equals
  \[ \sum_{\nu_1 = 0}^{\infty} \cdots \sum_{\nu_r = 0}^{\infty} z_1^{\nu_1}
     \cdots z_r^{\nu_r}  \langle \mu |\mathcal{U}_{\nu_r} \cdots
     \mathcal{U}_{\nu_1} | \lambda \rangle . \]
  Now the right-hand side enumerates $n$-ribbon tableaux in the definition
  (\ref{nribbontabdef}) and so we obtain (\ref{lltasvo}).
  The symmetry of $\mathcal{G}^n_{\lambda/\mu}$ is due to Lascoux,
  Leclerc and Thibon. It follows from the fact that the
  operators $\mathcal{U}_k$ commute.
\end{proof}

A similar result for Hall-Littlewood polynomials was found by
Jing~\cite{JingHallLittlewood}. Another vertex operator realization
of Hall-Littlewood polynomials may be found in Tsilevich~\cite{Tsilevich}.
Hall-Littlewood polynomials are
limits of LLT polynomials by~\cite{LLTRibbon}, Theorem~VI.6.

As an application of Theorem~\ref{thmlltasvo} we will deduce the Cauchy
identity for LLT polynomials, a result that is due to
Lam~{\cite{LamThesis,LamRibbon,LamBoson}}, proved also by
van Leeuwen~\cite{vLeeuwenSpin}. We will work with two sets of
variables, $z_1, \cdots, z_r$ and $w_1, \cdots, w_r$. Let
\[ \qquad \L_+ (\mathbf{z})^{\ast} = \sum_{k = 1}^{\infty} \frac{p_k
   (\mathbf{z})}{k} J_{- k}. \]
If the $z_i$ are real, then $\L_+(\mathbf{z})$ and $\L_+(\mathbf{z})^*$
are adjoints by Proposition~\ref{jkadjoint}.

We will denote $\mathcal{G}_{\lambda} =\mathcal{G}^n_{\lambda / \varnothing}$
where $\varnothing$ is the empty partition. We have
\begin{equation}
  \label{adjiphv} \langle \lambda |\L_+ (\mathbf{z})^{\ast} |0 \rangle
  =\mathcal{G}_{\lambda} (\mathbf{z}) .
\end{equation}
Indeed, since this is a purely algebraic identity, it is sufficient
to prove this if $z_i$ are real. Then since $\L_+(\mathbf{z})$ and
$\L_+(\mathbf{z})^\ast$ are adjoints, this follows by taking the
the conjugate of (\ref{lltasvo}).

Lam~\cite{LamRibbon} proved a version of the Cauchy identity
for LLT polynomials. We will show how this can be deduced from
Theorem~\ref{thmlltasvo}.

\begin{proposition}
  \label{explltcom}We have
  \[ \, \exp (\L_+ (\mathbf{z})) \, \exp (\L_+ (\mathbf{w})^{\ast}) = 
     \Omega (\mathbf{z}, \mathbf{w}) \, \exp (\L_+ (\mathbf{w})^{\ast}) \,
     \exp (\L_+ (\mathbf{z})) \]
  where
  \[ \Omega (\mathbf{z}, \mathbf{w}) = \prod_{t = 0}^{n - 1} \prod_{i, j}
     (1 - v^t z_i w_j)^{- 1} . \]
\end{proposition}

\begin{proof}
  Using (\ref{jkjmkcomm}) we have
  \[ [\L_+ (\mathbf{z}), \L_+ (\mathbf{w})^{\ast}] = \sum_{k = 1}^{\infty}
     \frac{1}{k} \left( \frac{v^{n k} - 1}{v^k - 1} \right) z_i^k w_j^k = \log
     \; \Omega (\mathbf{z}, \mathbf{w}) . \]
  The statement then follows from the Baker-Campbell-Hausdorff formula.
\end{proof}

We recall that if $\lambda$ is a partition, there is a unique smallest
partition $\delta$ that can be obtained by removing ribbon $n$-strips
from $\lambda$. The partition $\delta$ is called the \textit{$n$-core}
of $\lambda$. If $\delta=\lambda$ then $\lambda$ is called an
\textit{$n$-core} partition. See~\cite{MacdonaldHall}, Example~I.1.8.

\begin{lemma}
If $\delta$ is an $n$-core then $J_k|\delta\rangle=0$ for all $k>0$.
\end{lemma}

\begin{proof}
    Clearly $\mathcal{U}_k|\delta\rangle=0$ for $k>0$, and so
    by Proposition~\ref{operatorequality}
    $e^{\L_+(\mathbf{z})}|\delta\rangle=|\delta\rangle$.
    This means $J_k|\delta\rangle=0$.
\end{proof}

\begin{theorem}[Lam]
  \label{thm:lamcauchy}
  Let $\delta$ be an $n$-core. Then
  \begin{equation}
       \sum_{\lambda} \mathcal{G}_{\lambda/\delta} (\mathbf{z})
       \mathcal{G}_{\lambda/\delta} (\mathbf{w}) = \Omega (\mathbf{z},
       \mathbf{w}) ,
     \end{equation}
  where the sum is over all partitions with $n$-core~$\delta$.
\end{theorem}

\begin{proof}
  We will prove this under the assumption that $\mathbf{w}$ is real. Since
  this is a purely algebraic identity, that is sufficient. We evaluate
  \begin{equation}
    \label{dobletip} \langle \delta| \exp (\L_+ (\mathbf{z})) \exp (\L_+^{\ast}
    (\mathbf{w})) |\delta \rangle
  \end{equation}
  in two different ways. First, by Proposition~\ref{explltcom}, it equals
  \[ \Omega (\mathbf{z}, \mathbf{w}) \langle \delta| \exp(\L_+^{\ast} (\mathbf{w}))
     \exp (\L_+ (\mathbf{z})) |\delta \rangle = \Omega (\mathbf{z},
     \mathbf{w}), \]
  since if $k > 0$ we have $J_k |\delta \rangle = \langle \delta| J_{- k} = 0$, so $\exp
  (\L_+ (\mathbf{z})) |\delta \rangle = |\delta \rangle$, etc. On the other hand, using
  Theorem~\ref{thmlltasvo} and (\ref{adjiphv}) the coefficient
  (\ref{dobletip}) equals
  \[ \sum_{\lambda} \langle \delta| \exp (\L_+ (\mathbf{z})) | \lambda \rangle\,
     \langle \lambda | \exp (\L_+ (\mathbf{w}))^{\ast} |\delta \rangle =
     \sum_{\lambda} \mathcal{G}_{\lambda/\delta} (\mathbf{z}) \mathcal{G}_{\lambda/\delta}
     (\mathbf{w}) .\qedhere \]
\end{proof}

Now we recall the definition of the {\textit{super ribbon function}}
$\mathcal{G}_{\lambda / \mu}^n (\mathbf{z}|\mathbf{w}; q)$ defined
in~{\cite{LamRibbon}}, Definition~29. For this we require a double alphabet $1
\prec 1' \prec 2 \prec 2' \prec \cdots \prec r \prec r'$. A {\textit{super
ribbon tableau $T$}} is a sequence of partitions
\[ \mu = \lambda_{r + 1} \subset \lambda_{r'} \subset \lambda_r \subset \cdots
   \subset \lambda_{1'} \subset \lambda_1 = \lambda . \]
It is assumed that $\lambda_i / \lambda_{i'}$ is a horizontal $n$-ribbon strip,
and that $\lambda_{i'} / \lambda_{i + 1}$ is a vertical $n$-ribbon strip. We can
label the tableaux by labeling the boxes in $\lambda_i / \lambda_{i'}$ with
$i$, and the boxes in $\lambda_{i'} / \lambda_{i + 1}$ with $i'$. Let
$\wt (T) = (\nu_1, \cdots, \nu_r)$ where $\nu_i$ is the number of $i$ in
the tableau, and $\wt' (T) = (\nu_1', \cdots, \nu_r')$ where $\nu_i$ is
the number of~$i'$. Then we define the super ribbon function
\[ \mathcal{G}^n_{\lambda / \mu} (\mathbf{z}|\mathbf{w}; q) = \sum_T q^{s
   (T)} \mathbf{z}^{\wt (T)} (-\mathbf{w})^{\wt' (T)} \]
where the sum is over super ribbon tableaux.

\begin{theorem}
  \label{thm:superllt}
  For any pair of partitions $\mu \subseteq \lambda$,
  \begin{equation}
    \label{supervoformula}
    \langle \mu |e^{\L_+ (\mathbf{z})} e^{-\L_+ (\mathbf{w})} |
     \lambda \rangle =\mathcal{G}_{\lambda / \mu}^n (\mathbf{z}|\mathbf{w};
     q)\,.
  \end{equation}
  $\mathcal{G}_{\lambda/\mu}^n$ vanishes unless $\lambda$ and $\mu$ have
  the same $n$-core.
\end{theorem}

\begin{proof}
  Since the operators $\mathcal{U}_{\nu}$ and $\widetilde{\mathcal{U}}_{\nu'}$
  commute, we may apply Proposition~\ref{operatorequality} and rearrange to
  obtain
  \[ \langle \mu |e^{\L_+ (\mathbf{z})} e^{-\L_+ (\mathbf{w})} |
     \lambda \rangle = \sum_{\nu, \nu'} \mathbf{z}^{\wt (\nu)}
     (-\mathbf{w})^{\wt (\nu')} \langle \mu |
     \widetilde{\mathcal{U}}_{\nu'_r} \mathcal{U}_{\nu_r} \cdots
     \widetilde{\mathcal{U}}_{\nu'_1} \mathcal{U}_{\nu_1} | \lambda \rangle . \]
  Each operator on the right-hand side subtracts either a vertical or horizontal $n$-ribbon strip, and
  (\ref{supervoformula}) follows. The second statement is clear since removing
  a vertical or horizontal $n$-ribbon strip from a partition does not change
  its $n$-core.  
\end{proof}

\begin{corollary}\label{corllt=smf}
  In the notation of (\ref{metaplecticsf}),
  \[\mathcal{M}^n_{\lambda/\mu}(\mathbf{z}):=\langle\mu|T_\Delta(\mathbf{z})|\lambda\rangle=\mathcal{G}_{\lambda,\mu}^n(\mathbf{z}^n|v\mathbf{z}^n).\]
  $\mathcal{M}^n_{\lambda/\mu}$ vanishes unless $\lambda$ and $\mu$ have
  the same $n$-core.
\end{corollary}

\begin{proof}
  This follows from Theorems~\ref{thm:eH-equals-T} and~\ref{thm:superllt} because
  $H_+(z)=\L_+(z^n)-\L_+(vz^n)$.
\end{proof}

We may prove a similar Cauchy identity for the metaplectic symmetric
functions. The following Theorem holds for any values of $g(a)$ satisfying Assumption~\ref{assumptiong(a)}.

\begin{theorem}
  \label{thm:metaplecticcauchy}
  Let $\delta$ be an $n$-core partition. Then
  \begin{equation}
    \label{metacauchy} \sum_{\lambda} \mathcal{M}^n_{\lambda / \delta}
    (\mathbf{z}) \mathcal{M}^n_{\lambda / \delta} (\mathbf{w}) = \Theta
    (\mathbf{z}, \mathbf{w}), \hspace{2em} \Theta (\mathbf{z},
    \mathbf{w}) := \prod_{i, j} \frac{(1 - v z_i^n w^n_j) (1 - v^n
    z_i^n w_j^n)}{(1 - z_i^n w^n_j) (1 - v^{n + 1} z_i^n w_j^n)} .
  \end{equation}
\end{theorem}

\begin{proof}
  This is similar to Theorem~\ref{thm:lamcauchy}. We must first generalize the calculation in
  Proposition~\ref{explltcom}. Now we work with
  \[ H_+ (\mathbf{z}) = \sum_{k = 0}^{\infty} \frac{p_{n k}
     (\mathbf{z})}{k} (1 - v^k) J_k, \hspace{2em} H_+ (\mathbf{z})^{\ast}
     = \sum_{k = 0}^{\infty} \frac{p_{n k} (\mathbf{z})}{k} (1 - v^k) J_{-
     k} . \]
  We see that
  \[ [H_+ (\mathbf{z}), H_+ (\mathbf{w})^{\ast}] = \sum_{k = 0}^{\infty}
     p_{n k} (\mathbf{z}) p_{n k} (\mathbf{w}) \frac{1}{k} (1 - v^k) (1 -
     v^{n k}) \]
  \[ = \sum_{i, j} \sum_{k = 0}^{\infty} \frac{1}{k} (z_i w_j)^{n k} (1 - v^k
     - v^{n k} + v^{(n + 1) k}) = \log \Theta (\mathbf{z}, \mathbf{w}) .
  \]
  Therefore we have
  \[ \exp (H_+ (\mathbf{z})) \exp (H_+ (\mathbf{w})^{\ast}) = \Theta
     (\mathbf{z}, \mathbf{w}) \exp (H_+ (\mathbf{w})^{\ast}) \exp (H_+
     (\mathbf{z})) . \]
  The remainder of the calculation is similar to the proof of Theorem~\ref{thm:lamcauchy}.
\end{proof}

\section{\label{sec:metwhit}Metaplectic Whittaker Functions}

This work originated in the theory of Whittaker functions for the
metaplectic $n$-fold cover of $\GL_r$. These were represented
by Gamma and Delta ice partition functions for finite systems
in \cite{BrubakerBuciumasBump,BBBG}. In this section we will
show that metaplectic Whittaker functions can also be
expressed as partition functions for our infinite-dimensional
systems. More precisely, in (\ref{metaplecticsf}) we defined what
we are calling \textit{metaplectic symmetric functions}.
Like metaplectic Whittaker functions, they are partition
functions of metaplectic ice, but unlike metaplectic Whittaker
functions, the $\mathcal{M}_{\lambda,\mu}^n$ are symmetric functions.
What we will now show is a way of expressing metaplectic Whittaker
functions in terms of the $\mathcal{M}_{\lambda,\mu}^n$.

Let us review the relationship between the metaplectic
ice partition functions and metaplectic Whittaker functions, relying on
\cite{BrubakerBuciumasBump, BBBF, BBBG} for details.  Let $F$ be a
nonarchimedean local field. Assume that the group $\mu_{2n}$ of $2n$-th roots
of unity in $F$ has cardinality $2n$, and that the residue cardinality
$v^{-1}$ is prime to $n$. Let $\varpi$ be a prime element in the ring $\mathfrak{o}$ of
integers and let $\psi$ be a fixed additive character of $F$ that is trivial
on the ring of integers but no larger fractional ideal. Let
\[g(a)=\frac{1}{v^{-1}}\sum_{t \in(\mathfrak{o}/(\varpi))^\times}
  ({\varpi},t)^a\psi\!\left(\frac{t}{\varpi}\right),\]
where $(\;,\;)$ is the $n$-th order Hilbert symbol.
(We are calling the residue cardinality $v^{-1}$ instead
of $v$ or $q$ since it is the reciprocal of the
residue cardinality that will appear in our formulas.
We will use $q$ to denote a square root of $v$.) This
function $g(a)$ satisifies Assumption~\ref{assumptiong(a)}.

There is a central extension
\[1\longrightarrow\mu_{2n}\longrightarrow\widetilde{\GL}_r(F)\longrightarrow\GL_r(F)\longrightarrow1\]
that is essentially an $n$-fold cover, described in \cite{BrubakerBuciumasBump}.
We will refer to this as the \textit{metaplectic} group.

\begin{figure}
%\scalebox{0.8}{
\begin{tikzpicture}
  \coordinate (ab) at (1,0);
  \coordinate (ad) at (3,0);
  \coordinate (af) at (5,0);
  \coordinate (ah) at (7,0);
  \coordinate (aj) at (9,0);
  \coordinate (al) at (11,0);
  \coordinate (ba) at (0,1);
  \coordinate (bc) at (2,1);
  \coordinate (be) at (4,1);
  \coordinate (bg) at (6,1);
  \coordinate (bi) at (8,1);
  \coordinate (bk) at (10,1);
  \coordinate (bm) at (12,1);
  \coordinate (baUP) at (0,1.5);
  \coordinate (bcUP) at (2,1.5);
  \coordinate (beUP) at (4,1.5);
  \coordinate (bgUP) at (6,1.5);
  \coordinate (biUP) at (8,1.5);
  \coordinate (bkUP) at (10,1.5);
  \coordinate (bmUP) at (12,1.5);

  \coordinate (cb) at (1,2);
  \coordinate (cd) at (3,2);
  \coordinate (cf) at (5,2);
  \coordinate (ch) at (7,2);
  \coordinate (cj) at (9,2);
  \coordinate (cl) at (11,2);
  \coordinate (da) at (0,3);
  \coordinate (dc) at (2,3);
  \coordinate (de) at (4,3);
  \coordinate (dg) at (6,3);
  \coordinate (di) at (8,3);
  \coordinate (dk) at (10,3);
  \coordinate (dm) at (12,3);
  \coordinate (daUP) at (0,3.5);
  \coordinate (dcUP) at (2,3.5);
  \coordinate (deUP) at (4,3.5);
  \coordinate (dgUP) at (6,3.5);
  \coordinate (diUP) at (8,3.5);
  \coordinate (dkUP) at (10,3.5);
  \coordinate (dmUP) at (12,3.5);

  \coordinate (eb) at (1,4);
  \coordinate (ed) at (3,4);
  \coordinate (ef) at (5,4);
  \coordinate (eh) at (7,4);
  \coordinate (ej) at (9,4);
  \coordinate (el) at (11,4);
  \coordinate (fa) at (0,5);
  \coordinate (fc) at (2,5);
  \coordinate (fe) at (4,5);
  \coordinate (fg) at (6,5);
  \coordinate (fi) at (8,5);
  \coordinate (fk) at (10,5);
  \coordinate (fm) at (12,5);
  \coordinate (faUP) at (0,5.5);
  \coordinate (fcUP) at (2,5.5);
  \coordinate (feUP) at (4,5.5);
  \coordinate (fgUP) at (6,5.5);
  \coordinate (fiUP) at (8,5.5);
  \coordinate (fkUP) at (10,5.5);
  \coordinate (fmUP) at (12,5.5);

  \coordinate (gb) at (1,6);
  \coordinate (gd) at (3,6);
  \coordinate (gf) at (5,6);
  \coordinate (gh) at (7,6);
  \coordinate (gj) at (9,6);
  \coordinate (gl) at (11,6);
  \coordinate (bb) at (1,1);
  \coordinate (bd) at (3,1);
  \coordinate (bf) at (5,1);
  \coordinate (bh) at (7,1);
  \coordinate (bj) at (9,1);
  \coordinate (bl) at (11,1);
  \coordinate (db) at (1,3);
  \coordinate (dd) at (3,3);
  \coordinate (df) at (5,3);
  \coordinate (dh) at (7,3);
  \coordinate (dj) at (9,3);
  \coordinate (dl) at (11,3);
  \coordinate (fb) at (1,5);
  \coordinate (fd) at (3,5);
  \coordinate (ff) at (5,5);
  \coordinate (fh) at (7,5);
  \coordinate (fj) at (9,5);
  \coordinate (fl) at (11,5);
  \draw[thick] (ab)--(gb);
  \draw[thick] (ad)--(gd);
  \draw[thick] (af)--(gf);
  \draw[thick] (ah)--(gh);
  \draw[thick] (aj)--(gj);
  \draw[thick] (al)--(gl);
  \draw[thick] (ba)--(bm);
  \draw[thick] (da)--(dm);
  \draw[thick] (fa)--(fm);
  \foreach \P in {(ab), (ad), (af), (ah), (aj), (al), (ba), (bc), (be), (bg), (bi), (bk), (bm), (cb), (cd), (cf), (ch), (cj), (cl), (da), (dc), (de), (dg), (di), (dk), (dm), (eb), (ed), (ef), (eh), (ej), (el), (fa), (fc), (fe), (fg), (fi), (fk), (fm), (gb), (gd), (gf), (gh), (gj), (gl)}
  {%
  \draw[fill=white, thick] \P circle (.25);
  }%
  \foreach \P in {(bb), (bd), (bf), (bh), (bj), (bl), (db), (dd), (df), (dh), (dj), (dl), (fb), (fd), (ff), (fh), (fj), (fl)}
  {%
  \path[fill=white] \P circle (.32);
  }%
  \node at (bb) {${z_1}$};
  \node at (bd) {${z_1}$};
  \node at (bf) {${z_1}$};
  \node at (bh) {${z_1}$};
  \node at (bj) {${z_1}$};
  \node at (bl) {${z_1}$};
  \node at (db) {${z_2}$};
  \node at (dd) {${z_2}$};
  \node at (df) {${z_2}$};
  \node at (dh) {${z_2}$};
  \node at (dj) {${z_2}$};
  \node at (dl) {${z_2}$};
  \node at (fb) {${z_3}$};
  \node at (fd) {${z_3}$};
  \node at (ff) {${z_3}$};
  \node at (fh) {${z_3}$};
  \node at (fj) {${z_3}$};
  \node at (fl) {${z_3}$};
  \node at (gb) {$-$};
  \node at (gd) {$-$};
  \node at (gf) {$+$};
  \node at (gh) {$+$};
  \node at (gj) {$+$};
  \node at (gl) {$-$};
  \node at (fa) {$+$};
  \node at (fc) {$+$};
  \node at (fe) {$-$};
  \node at (fg) {$-$};
  \node at (fi) {$+$};
  \node at (fk) {$+$};
  \node at (fm) {$-$};
  \node at (eb) {$-$};
  \node at (ed) {$+$};
  \node at (ef) {$+$};
  \node at (eh) {$-$};
  \node at (ej) {$+$};
  \node at (el) {$+$};
  \node at (da) {$+$};
  \node at (dc) {$-$};
  \node at (de) {$-$};
  \node at (dg) {$-$};
  \node at (di) {$-$};
  \node at (dk) {$-$};
  \node at (dm) {$-$};
  \node at (cb) {$+$};
  \node at (cd) {$+$};
  \node at (cf) {$+$};
  \node at (ch) {$-$};
  \node at (cj) {$+$};
  \node at (cl) {$+$};
  \node at (ba) {$+$};
  \node at (bc) {$+$};
  \node at (be) {$+$};
  \node at (bg) {$+$};
  \node at (bi) {$-$};
  \node at (bk) {$-$};
  \node at (bm) {$-$};
  \node at (ab) {$+$};
  \node at (ad) {$+$};
  \node at (af) {$+$};
  \node at (ah) {$+$};
  \node at (aj) {$+$};
  \node at (al) {$+$};
  \node at (11,7) {$0$};
  \node at (9,7) {$1$};
  \node at (7,7) {$2$};
  \node at (5,7) {$3$};
  \node at (3,7) {$4$};
  \node at (1,7) {$5$};
  \node at (baUP) {$0$};
  \node at (bcUP) {$0$};
  \node at (beUP) {$0$};
  \node at (bgUP) {$0$};
  \node at (biUP) {$1$};
  \node at (bkUP) {$0$};
  \node at (bmUP) {$1$};
  
  \node at (daUP) {$0$};
  \node at (dcUP) {$1$};
  \node at (deUP) {$0$};
  \node at (dgUP) {$1$};
  \node at (diUP) {$0$};
  \node at (dkUP) {$1$};
  \node at (dmUP) {$0$};

  \node at (faUP) {$0$};
  \node at (fcUP) {$0$};
  \node at (feUP) {$1$};
  \node at (fgUP) {$0$};
  \node at (fiUP) {$0$};
  \node at (fkUP) {$0$};
  \node at (fmUP) {$1$};

  \node at (-1.5,1) {$1$};
  \node at (-1.5,3) {$2$};
  \node at (-1.5,5) {$3$};
\end{tikzpicture}
\caption{A state of Delta ice. In this example $n=2$. The charges (written
  above the horizontal edges) are integers modulo $n$ that change
  at the $-$ spins in accordance with Table~\ref{tab:mweights}.
  The charges at $+$ edges must be $\equiv 0$ modulo~$n$. For
  Gamma ice, the system is similar, but the rows are numbered increasing
  from top to bottom, and the left edges have variable charge,
  while the right edges all have charge $0$, since in Gamma ice
  $-^a$ is only allowed with $a$ equal to $0$ modulo~$n$.}
\label{stateofmice}
\end{figure}
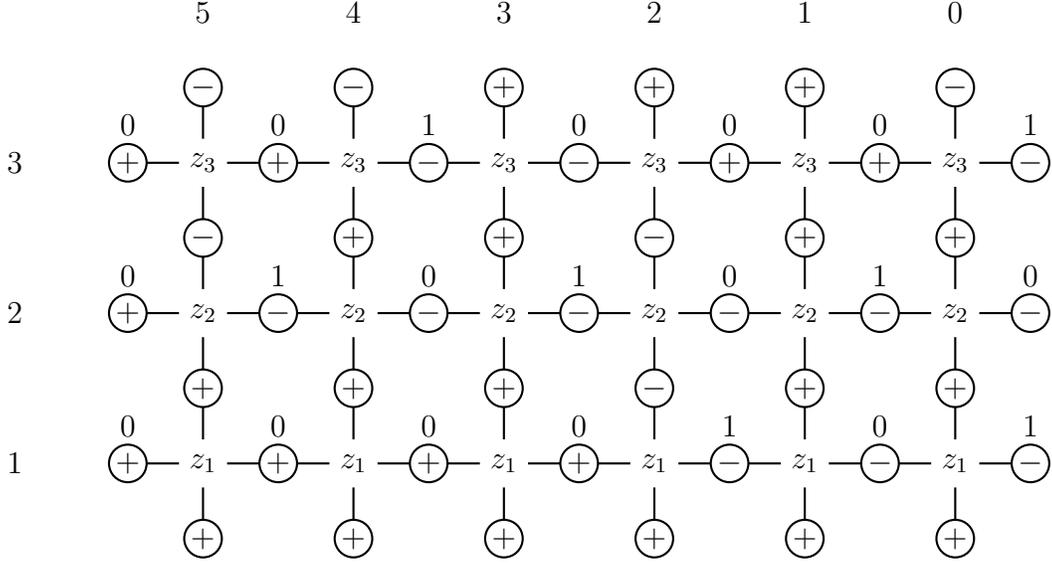

If $\mathbf{z} \in (\mathbb{C}^{\times})^r$ then there is a principal
series representation $\pi_{\mathbf{z}}$ defined in
{\cite{BrubakerBuciumasBump}}. Associated with $\pi_{\mathbf{z}}$ there are
$n^r$ linearly independent {\textit{spherical Whittaker functions}} on
$\widetilde{\GL_r} (F)$. Let $\mathcal{W}_{\mathbf{z}}$ denote the
space of functions spanned by these.
If $W\in\mathcal{W}_{\mathbf{z}}$, we are interested in the values
of $W$ evaluated at
\[\varpi^\lambda:=\mathbf{s}\left(\begin{array}{ccc}\varpi^{\lambda_1}\\&\ddots\\&&\varpi^{\lambda_r}\end{array}\right),\]
where $\mathbf{s}:\GL_r(F)\to\widetilde{\GL}_r(F)$ is a standard
section (see~\cite{BrubakerBuciumasBump}) and $\lambda$ is a
partition of length $\leqslant r$. These are combinatorially
interesting sums of products of Gauss sums and polynomials
in~$v$ whose study goes back to Kazhdan and Patterson~\cite{KazhdanPatterson}.
In~\cite{BrubakerBuciumasBump,BBBG} we showed how to represent
such Whittaker functions in terms of finite systems of Gamma and Delta ice.
In this section we will show that metaplectic Whittaker
functions can also be described as partition functions
of infinite systems, and thereby relate them to the
metaplectic symmetric functions, and to vertex operators.

Let $\lambda = (\lambda_1, \cdots, \lambda_r)$ be a partition of length
$\leqslant r$, let $\mathbf{z}=(z_1, \cdots, z_r)\in(\mathbb{C}^\times)^r$, and let
$\sigma = (\sigma_1, \cdots , \sigma_r)\in (\mathbb{Z}/n\mathbb{Z})^r$.
We will now describe the finite systems
$\mathfrak{S}_{\lambda,\sigma}^\Gamma$ and
$\mathfrak{S}_{\lambda,\sigma}^\Delta$ depending on these data.
These were considered previously in~\cite{BrubakerBuciumasBump} (Gamma ice only)
and in~\cite{BBBG} (both systems).

Let $\rho=(r-1,r-2,\cdots,0)$ so that
\[ \lambda + \rho = (\lambda_1 + r - 1, \lambda_2 + r - 2, \cdots, \lambda_r)
\]
is a strict partition. We consider a grid with $r$ rows and $N$ columns, where $N$ is any positive
integer such that $N \geqslant \lambda_1 + r - 1$. The columns are labeled
$N, N - 1, \ldots, 0$ in decreasing order from left to right and the rows are
labeled $1, \ldots, r$ from the top down for Delta ice and from the bottom up
for Gamma ice.

On the vertical edges along the top boundary, we put $-$ in the $k$-th column
if $k$ is an entry in $\lambda + \rho$; otherwise we put $+$. On the vertical
edges at the bottom, we put $+$ in every column. On the
horizontal edges along the left boundary we put the decorated spin $+^0$ (Delta ice)
or $+^{\sigma_i}$ in the $i$-th row (Gamma ice).  On the horizontal edges on
the right boundary we put the decorated spin $-^{\sigma_i}$ in the $i$-th column
(Delta ice) or $-^0$ (Gamma ice). We use the Gamma Boltzmann weights
in the $i$-th row for Gamma ice and Delta Boltzmann weights for Delta ice.
See Figure~\ref{stateofmice} for an example of the system $\mathfrak{S}_{\lambda,\sigma}^\Delta$.

Let $\delta$ denote the modular quasicharacter of the Borel subgroup on
$\GL_r(F)$, lifted to a function on $\widetilde{\GL}_r(F)$.

\begin{proposition}
  \label{thm:whitaspf}
  Let $\sigma \in (\mathbb{Z}/ n\mathbb{Z})^r$ and
  $\mathbf{z} \in (\mathbb{C}^{\times})^r$. Then there exists a spherical Whittaker function
  $W^\Delta_{\sigma} \in \mathcal{W}_{\mathbf{z}}$ such that for $\lambda$ a
  partition of length $\leqslant r$, we have
  \begin{equation}
  \label{pfaswhit}
   Z (\mathfrak{S}_{\lambda, \sigma}^{\Delta}) = \delta^{- 1 / 2}
     (\varpi^{\lambda}) W^{\Delta}_{\sigma} (\varpi^{\lambda})\;.
  \end{equation}
\end{proposition}

\begin{proof}
  By Theorem~6.3 of {\cite{BrubakerBuciumasBump}}, there exists a spherical
  Whittaker function $W_{\sigma}^{\Gamma}$ such that
  \begin{equation}
    \label{gapartwhit} Z (\mathfrak{S}_{\lambda, \sigma}^{\Gamma}) = \delta^{-
    1 / 2} (\varpi^{\lambda}) W^{\Gamma}_{\sigma} (\varpi^{\lambda}) .
  \end{equation}
  Here we have absorbed the factor $\mathbf{z}^{w_0 \rho + \gamma}$ that appears
  in that theorem into the Whittaker function. It is
  proved in Theorem~2.3 of {\cite{BBBG}} that
  \[ \sum_{\sigma \in (\mathbb{Z}/ n\mathbb{Z})^r} Z (\mathfrak{S}_{\lambda,
     \sigma}^{\Delta}) = (z_1 \cdots z_r)^N \sum_{\tau \in (\mathbb{Z}/
     n\mathbb{Z})^r} Z (\mathfrak{S}_{\lambda, \tau}^{\Gamma}) . \]
  The constant $(z_1 \cdots z_r)^N$ does not appear in {\cite{BBBG}}, but
  this is because in this paper we have changed the Boltzmann weights for
  Gamma ice, in order that the partition function be convergent for infinite
  grids. The arguments there are easily refined to show that, for suitable
  constants $S_{\sigma, \tau} (\mathbf{z})$ independent of $\lambda$, we have
  \[ Z (\mathfrak{S}_{\lambda, \sigma}^{\Delta}) = \sum_{\tau \in
     (\mathbb{Z}/ n\mathbb{Z})^r} S_{\sigma, \tau} (\mathbf{z}) Z
     (\mathfrak{S}_{\lambda, \tau}^{\Gamma}) . \]
  Substituting (\ref{gapartwhit}) into this identity gives us (\ref{pfaswhit})
  with $W_\sigma^\Delta=\sum_\tau S_{\sigma,\tau} W_\tau^\Gamma$.
\end{proof}

Now we wish to relate the partition functions of these finite
systems to the infinite systems defined in Section~\ref{deltaicefockspace}.
Let us choose a partition $\xi$ whose first part $\xi_1\leqslant r-1$.
Now if $\lambda$ is a partition of
length $\leqslant r$ let $\lambda \star \xi$ denote the partition
$\lambda' = (\lambda'_1, \lambda'_2, \cdots)$ where
\[ \lambda'_j = \left\{ \begin{array}{ll}
     \lambda_j + r - 1 & \text{if $j \leqslant r$},\\
     \xi_{j - r} & \text{if $j > r$} .
   \end{array} \right. \]
Note that since we have assumed that $\xi_1\leqslant r-1$ these
entries are weakly decreasing, so $\lambda \star \xi$ is a partition.
In Frobenius notation,
\[\lambda\star\xi=\left(\begin{array}{ccccc}\lambda_1&\lambda_2&\lambda_3&\cdots&\lambda_r\\
\xi'_1+1&\xi'_2+1&\xi'_3+1&\cdots&\xi'_r+1\end{array}\right),\]
where $\xi'$ is the conjugate partition of~$\xi$.

\begin{proposition}
  \label{whitfrommsf}
  Let $\xi$ be a partition such that $\xi_1\leqslant r-1$ 
  and let $\sigma\in(\mathbb{Z}/n\mathbb{Z})^r$.
  Then there exist constants $c(\xi,\sigma; \mathbf{z})$ depending
  on $\xi$ and $\sigma$ such that if
  $\lambda=(\lambda_1,\cdots,\lambda_r)$ is a partition of length $\leqslant r$, then
  \begin{equation}
    \label{eq:whittaker-partition}
    \langle0|T_{\mathbf{z}}|\lambda\star\xi\rangle 
    = \sum_{\sigma\in(\Z/n\Z)^r}
    c(\xi,\sigma;\mathbf{z})\,Z(\mathfrak{S}_{\lambda,\sigma}^\Delta).
  \end{equation}
  Here in the notation (\ref{partitionnot}) both vectors $|0\rangle=|0;0\rangle$
  and $|\lambda\star\xi\rangle=|\lambda\star\xi;0\rangle$ are
  in~$\mathfrak{F}_0$.
\end{proposition}

\begin{proof}
  Let us define an invariant $N:\mathfrak{F}_0\to\mathbb{N}$. Suppose $\xi=u_{\mathbf{i}}$,
  where $\mathbf{i}$ is as in (\ref{boldij}), with $m=0$. If $0>i_0$ then
  we define $N(\xi)=0$; otherwise, $N(\xi)=t$ where
  $t$ is such that $i_{-t}\geqslant 0> i_{-t-1}$. If $\xi$ is interpreted
  as an assignment of spins to a sequence of vertical edges, then
  $N(\xi)$ is the number of $-$ spins to the left of the $0$-th
  column (inclusively) or equivalently (since $\xi\in\mathfrak{F}_0$),
  the number of $+$ spins strictly to the right of the $0$-th column.

  Consider a state of the infinite system $\mathfrak{S}^\Delta_{\mathbf{z}, |\lambda\star\xi\rangle,|0\rangle,r}$ 
  of Section~\ref{deltaicefockspace}. For $0\leqslant r$ let
  $\mathbf{i}^{(k)}$ be the decreasing sequence such that in the
  notation (\ref{semimon}), $u_{\mathbf{i}^{(k)}}$ is the element of $\mathfrak{F}_0$
  corresponding to the configuration of spins below the $k$-th row,
  and $u_{\mathbf{i}^{(k-1)}}$ is the element corresponding to the
  configuration above it. Thus $u_{\mathbf{i}^{(0)}}=|\lambda\star\xi\rangle$
  and $u_{\mathbf{i}^{(r)}}=|0\rangle.$

  We will show that the spins of the horizontal edges connecting vertices of
  the $0$-th column to those of the $-1$-st column are all $-$. Indeed,
  it follows from Lemma~\ref{interleavinglemma}
  that either $N(u_{\mathbf{i}^{(k+1)}})=N(u_{\mathbf{i}^{(k)}})$
  or $N(u_{\mathbf{i}^{(k+1)}})=N(u_{\mathbf{i}^{(k)}}) - 1$. But
  since $N(u_{\mathbf{i}^{(r)}})=0$ and $N(u_{\mathbf{i}^{(0)}})=N(|\lambda\star\xi\rangle)=r$,
  we must have $N(u_{\mathbf{i}^{(k)}})=k$ for all $k$.
  Now the fact that $N(u_{\mathbf{i}^{(k)}})$ and
  $N(u_{\mathbf{i}^{(k-1)}})$ have opposite parity implies
  that the spin in the $k$-th row on the horizontal edge
  to the right of the $0$-th column is $-$, as required.

  Now to complete the proof, we fix spins $\sigma=(\sigma_1,\cdots,\sigma_r)$
  and collect together the states whose decorated spin
  on the edge in the $k$-th row to the right of the $0$-th
  column is $-^{\sigma_k}$. The product of the Boltzmann
  weights to the left of the $0$-th column is
  the Boltzmann weight of a state of $\mathfrak{S}_{\lambda,\sigma}^\Delta$,
  and so clearly the sum of such Boltzmann weights equals
  $Z(\mathfrak{S}_{\lambda,\sigma}^\Delta)$ times a factor
  that is independent of $\lambda$.
\end{proof}

Let us reformulate this result as expressing a metaplectic Whittaker function
in terms of the metaplectic symmetric functions.

\begin{theorem}
  \label{whitfrommw}
  Let $\xi$ be a partition (of any length) such that $\xi_1\leqslant r-1$.
  Then
  \[ \sum_{\sigma\in(\mathbb{Z}/n\mathbb{Z})^r}
  c(\xi,\sigma;\mathbf{z})\,W_{\sigma}^{\Delta} (\varpi^{\lambda}) =
  \delta^{1 / 2} (\varpi^{\lambda}) \mathcal{M}^n_{\lambda \star \xi} (\mathbf{z}) . \]
\end{theorem}

\begin{proof}
This follows from combining Proposition~\ref{whitfrommsf} with
Proposition~\ref{thm:whitaspf}
and the definition (\ref{metaplecticsf}) of the metaplectic symmetric function
$\mathcal{M}^n_{\lambda}$.
\end{proof}

\begin{remark}
By Corollary~\ref{corllt=smf}, this particular Whittaker function
vanishes at $\varpi^\lambda$ unless $\lambda\star\xi$ has empty $n$-core.
Although $\mathcal{M}^n_{\lambda\star\xi}$ is a symmetric function,
this does not imply that the Whittaker function is symmetric
in $\mathbf{z}$ because of the factor $c(\xi,\sigma;\mathbf{z})$.
These coefficients may be of interest for their own sake.
\end{remark}

\begin{remark}
It seems probable that the Whittaker functions on the left-hand side (varying
$\xi$) span the space of Whittaker functions. Such a result would give a
two-way connection between metaplectic Whittaker functions and
metaplectic symmetric functions.
\end{remark}

\section{Vertex Operators\label{sec:vertexops}}
So far we have put a lot of focus on operators of the form
either:
\begin{equation}
    \label{eq:Vpm}
    V_+ (z) = \exp (H_+ [a] (z))  \hspace{2em} \text{or} \hspace{2em} V_- (z) =
   \exp (H_- [a] (z)),
\end{equation}
which we call \textit{half-vertex operators}, 
where $H_+ [a] (z)$ and $H_- [a] (z)$ are formal power series in $z$ and $z^{-
1}$ respectively defined by (\ref{hplusadef}) and (\ref{hminadef}). Recall
that $H_+$ involves the right-moving operators $J_k$ and $H_-$ involves the
left-moving operators $J_{- k}$ with $k > 0$. We have proved that the
operators $T_{\Delta} (z)$ and $T_{\Gamma} (z)$ are of this type.

In this section we will consider their products $V_- (z) V_+ (z)$, such as the operator $T_{\Gamma} (z) T_{\Delta} (z)$, and investigate if they satisfy the properties of a \emph{vertex operator}. There
is one reason to believe that $T_{\Gamma} (z) T_{\Delta} (z)$ is a natural
entity: in symplectic ice ({\cite{Ivanov,BBCG,GrayThesis}}) one represents the
Whittaker function on the $n$-fold cover of $\Sp (2 r)$ with Langlands
parameters $z_1, z_1^{- 1}, \cdots, z_r, z_r^{- 1}$ by the partition function
of a system having alternating layers of Gamma and Delta ice. The two adjacent
layers are joined by a ``cap'' vertex which does not have an obvious analog in
our current setup.

Gamma and Delta ice occur together in another context, namely the equality of
the partition functions for Gamma and Delta ice. In~{\cite{BBBG}} this result
(established earlier with greater difficulty in {\cite{wmd5book}}) is proved
using Yang-Baxter equations. In that context $T_{\Gamma} (z) T_{\Delta} (w)$
only appear there with $z$ and $w$ distinct. For this, our
Theorem~\ref{thm:gdcommute} below is relevant, taking the place of the
Yang-Baxter equations in our current setup.
Still, in this section we are mainly interested in $V (z) = T_{\Gamma} (z)
T_{\Delta} (z)$ with the parameters equal.

Vertex operators exhibit a property called {\textit{locality}}. This is a
generalization of commutativity that was emphasized in
{\cite{DongLepowskyBook,KacBeginners,FrenkelBenZvi}}.
It is explained in Chapter~1 of \cite{KacBeginners} that for
the vertex operators arising in conformal field theory, locality is a
reflection of the locality in the Wightman axioms for a quantum field theory:
two fields with disjoint support having spacelike separation commute as
operators.

If $F$ is a field, let $F[[z]]$ be the ring of formal power series
$\sum_{n\geqslant 0}a_nz^n$ with $a_n\in F$, and let $F((z))$ be the
fraction field of $F[[z]]$, consisting of Laurent series $\sum_{n=-N}^\infty a_nz^n$
with only finitely many negative coefficients. Let $\HS((z))$ denote
$\C((z))\otimes\HS$, the space of Laurent series with coefficients
in~$\HS$.

In vertex algebras a \emph{field} is represented by a formal power series
\begin{equation*}
    A(z) = \sum_{k =-\infty}^{\infty} A_k z^{-k-1}
\end{equation*}
where $A_k$ is an operator on a Hilbert space $\HS$ such that for any vector
$|v\rangle \in \HS$, $A_k |v\rangle = 0$ for $k \gg 0$. A field gives
rise to a map $\HS\to\HS((z))$.

Let $B(w) = \sum_{k =-\infty}^{\infty} B_k w^{-k-1}$ similarly be a field.
Locality is a generalization of commutativity in the sense that two fields
$A(z)$ and $B(w)$ are called \emph{mutually local} if $[A(z), B(w)] = A (z) B
(w) - B (w) A (z)$ is a formal distribution concentrated on the diagonal
$z = w$. We will explain more precisely what this means.

Note that the matrix elements of $A(z)B(w)$ are elements in $\cmplx((z))((w))$,
that is, in $F((w))$ where $F = \cmplx((z))$. Similarly the
matrix elements of $B(w)A(z)$ are elements in $\cmplx((w))((z))$. The
difference between $\cmplx((z))((w))$ and $\cmplx((w))((z))$ is illustrated by
image of the rational function $1/(z-w)$ embedded into the two spaces as
$z^{-1} \sum_{k=0}^\infty (w/z)^k$ and $-w^{-1} \sum_{k=0}^\infty (z/w)^k$
respectively. Requiring that the matrix elements of $[A(z), B(w)]$ should
vanish identically would restrict us to elements in the intersection of
$\cmplx((z))((w))$ and $\cmplx((w))((z))$ in $\cmplx[[z^\pm, w^\pm]]$ which is
the space $\cmplx[[z,w]][z^{-1}, w^{-1}]$ giving a too strong
condition \cite{FrenkelBenZvi}. Instead, we use the more relaxed condition
that
\begin{equation}
  \label{mutlocal}
  (z-w)^N [A(z), B(w)] = 0
\end{equation}
as a formal power series for some positive integer $N$. In this case we say
the fields $A(z)$ and $B(w)$ are \textit{mutually local}.

Let us give another explanation of this notion. We assume that $A_k$ and
$B_l$ commute if $k$ and $l$ are either both positive or both negative, that
$A_0$ and $B_0$ commute with all $A_k$ and $B_l$. Moreover let us assume that
the normal-ordered product
\[ {: \mathrel{A (z) B (w)} :} = \sum_{k = - \infty}^{\infty} \sum_{l = -
   \infty}^{\infty} z^{- k - 1} w^{- l - 1} {:\mathrel{A_k B_l}:} \]
is a bounded operator on $\HS$, where
\[ {:\mathrel{A_k B_l}:} = \left\{ \begin{array}{ll}
     B_l A_k & \text{if $k > 0$},\\
     A_k B_l & \text{otherwise} .
   \end{array} \right. \]
Our assumptions imply that ${: \mathrel{A (z) B (w)} :} = {: \mathrel{B (w) A (z)} :}$. Now
consider:
\[ A (z) B (w) - {: \mathrel{A (z) B (w)} :}. \]
Very often this operator will be given by a power series that is convergent
when $| w | < | z |$. Let us denote this as $\phi (z, w)$. Furthermore it may
be that $B (w) A (z) - {: \mathrel{A (z) B (w)} :}$ is also given by a power series,
convergent when $| z | < | w |$, and that this represents the same rational
function $\phi (z, w)$. In this case the fields $A (z)$ and $B(w)$ are mutually local.

To clarify this with an example, let us work with a Heisenberg Lie algebra
having generators $J_k$ ($k \in \mathbb{Z}$) with $J_0$ central having the
commutator relations
\[ [J_k, J_l] = \delta_{k, - l} k \cdot c \]
where $c$ is another central element. This is the special case $n = 1$ of
(\ref{jkjmkcomm}). The Hilbert space $\HS$ is to be generated by a
vacuum $|0 \rangle$ such that $J_k |0 \rangle = 0$ if $k > 0$ and $c$ acts by
the identity. Now consider the field
\[ J (z) = \sum_{k = - \infty}^{\infty} z^{- k - 1} J_k . \]
We have
\[ J (z) J (w) - {:\mathrel{J (z) J (w)}:} = \sum_{k = 1}^{\infty} [J_k, J_{- k}] z^{-
   1 - k} w^{- 1 + l} \cdot c = \frac{1}{(z - w)^2} \cdot c, \]
the series being convergent when $| w | < | z |$. Since $J (w) J (z) - {:\mathrel{J
(w) J (z)}:}$ gives the same expression in the complementary domain $| z | < |
w |$, the fields $J (z)$ and $J (w)$ are mutually local.

Thus the locality is a generalization of the condition that
$A (z) B (w) = B(w) A (z)$. Dong and Lepowsky~{\cite{DongLepowskyBook}} considered a similar
generalization of the condition that
\begin{equation}
    \label{phasecommute}
  A (z) B (w) = e^{i \pi \tau} B (w) A (z),
\end{equation}
for a phase shift $e^{i \pi \tau}$. Our Proposition~\ref{ulocalcr} below
shows that we need such a generalization of locality. Frenkel and
Reshetikhin~{\cite{FrenkelReshetikhinChiral}} considered even more generally the
case where the phase shift is replaced by an operator $S (w / z)$ that depends
analytically only on $z$ and $w.$ 
For consistency it is necessary that $S (w / z)$ satisfies a parametrized
Yang-Baxter equation. This is automatic if $S (w / z)$ is a scalar, in which
case this identity is similar to (\ref{phasecommute}).

There is another respect in which the
framework of \cite{FrenkelReshetikhinChiral} is more general than the usual locality, and this is that they allow
$S (w / z)$ to have poles not only on the diagonal $z = w$ but on shifted
diagonals $z = \gamma w$ where $\gamma$ lies in a discrete subgroup of
$\mathbb{C}^{\times}$. This concept of locality in {\cite{FrenkelReshetikhinChiral}} is what we see
in our examples with the set of lines $z = v^j w$ and $S(w/z)$ being a scalar.

We require that $S$ is a
meromorphic function, with poles only along the lines $w = v^j z$ for a finite
number of integer values of $j$, such that
\begin{equation}
  \label{generalcommute}
  A (z) B (w) = S (w / z) B (w) A (z) .
\end{equation}
Let us first consider the meaning of this when $A (z) = B (z) = V (z) = V_-
(z) V_+ (z)$ where $V_\pm(z)$ are defined in \eqref{eq:Vpm}. Suppose that we can find a rational function $\phi (x)$ such
that (formally)
\[ V_+ (z) V_- (w) = \phi \left( \frac{z}{w} \right) V_- (w) V_+ (z) . \]
Then since $V_- (z)$ commutes with $V_- (w)$ and $V_+ (z)$ commutes with $V_+
(w)$ we have
\[ V (z) V (w) = \phi \left( \frac{z}{w} \right) : \mathrel{V (z) V (w)} : \]
where the normal-ordered product is
\begin{equation}
   \label{novzvw}
  {: \mathrel{V (z) V (w)} :} = V_- (z) V_- (w) V_+ (z) V_+ (w) .
\end{equation}
Then
\begin{equation}
    \label{generalcommute-phi}
    V (z) V (w) = S (w / z) V (w) V (z), \hspace{2em} S (x) = \phi (x^{-1}) / \phi(x)
\end{equation} 

\begin{remark}
\label{noadvantages}
Strictly speaking $V_- (z) V_+ (z)$ is not an operator on $\mathfrak{F}_m$
since $V_- (z) | \lambda \rangle$ produces an infinite number of terms.
However $\langle \mu|V_- (z) V_+ (z) | \lambda \rangle$ is a finite sum.
Moreover, the normal-ordered product $:\mathrel{V(z)V(w)}:$ defined by
(\ref{novzvw}) is such that $\langle\mu|:\mathrel{V(z)V(w)}:|\lambda\rangle$
is a finite sum. The normal-ordered product has the advantage of being unchanged if $z$ and $w$ are switched.
\end{remark}

We now take the Heisenberg generators $J_k$ to satisfy the commutator relation
(\ref{jkjmkcomm}), and $\mathcal{H}=\mathfrak{F}_m$ for some fixed $m$.
We supplement the $J_k$ ($k \neq 0$) by $J_0$ which acts on
$\mathfrak{F}_m$ by the scalar $m$. Define the shift operator $Q :
\mathfrak{F}_m \longrightarrow \mathfrak{F}_{m + n}$ by
\[ Q (u_{i_m} \wedge u_{i_{m - 1}} \wedge \cdots) = u_{i_m + n} \wedge u_{i_{m
   - 1} + n} \wedge \cdots . \]
We will use the notation $| \lambda \rangle = | \lambda ; m \rangle$ introduced
in (\ref{partitionnot}) for basis vectors.

We may regard $V_- (z) V_+ (z)$ as a map from $\mathfrak{F}$ into a suitable
completion. Depending on the coefficients $a$ it may be useful to supplement
$V_- (z) V_+ (z)$ by a factor such as $Q^r z^{a_0 J_0}$.

\begin{example}
The first case we wish to consider is $V_+ (z) =
T_{\Delta} (z)$, $V_- (z) = T_{\Gamma} (z)$. The operators $T_{\Delta} (z)$ and $T_{\Delta}(w)$
commute as follows from our main theorem (or by
a Yang-Baxter equation argument). Similarly the $T_{\Gamma} (z)$ mutually
commute for varying~$z$. But $T_{\Delta} (z)$ does not commute with
$T_{\Gamma} (w)$. Moreover, we must be cautious about composing these.
Consider
\begin{equation}
  \label{tdtgwrite} \langle \eta |T_{\Delta} (z) T_{\Gamma} (w) | \xi \rangle
  = \sum_{\zeta} \langle \eta |T_{\Delta} (z) | \zeta \rangle  \langle \zeta
  |T_{\Gamma} (w) | \xi \rangle .
\end{equation}
There are an infinite number of terms on the right-hand, side. The sum
converges provided $|z| < c_1 |w|$ where $c_1 = \min (1, |v|^{- 1 - 1 / n})$. 

\begin{remark}
  \label{noremark}There are no such convergence issues if we compose in the
  other (normal-ordered) way: because $T_{\Gamma} (w) T_{\Delta} (z)$ does the
  right-moving modes first, the sum corresponding to (\ref{tdtgwrite}) is a
  finite sum.
\end{remark}

\begin{theorem}
  \label{thm:gdcommute}Suppose that $|z| < c_1 |w|$. Then
  \begin{equation}
    \label{gamdelcommutator} T_{\Delta} (z) T_{\Gamma} (w) = \frac{(1 - vz^n
    w^{- n})  (1 - v^n z^n w^{- n})}{(1 - z^n w^{- n})  (1 - v^{n + 1} z^n
    w^{- n})} T_{\Gamma} (w) T_{\Delta} (z) .
  \end{equation}
\end{theorem}

This, together with Remark~\ref{noremark} allows us to analytically continue
the conditionally convergent composition $T_{\Delta} (z) T_{\Gamma} (w)$,
except to the poles of the denominator in (\ref{gamdelcommutator}).

\begin{proof}[Proof of Theorem~\ref{thm:gdcommute}]
  By a computation very similar to the proof of Theorem~\ref{thm:metaplecticcauchy}, we have
  \begin{equation}
      \label{eq:ex1-phi}
      e^{H_+ (z)} e^{H_- (w)} e^{- H_+ (z)} e^{- H_- (w)} = \frac{(1 - vz^n
     w^{- n})  (1 - v^n z^n w^{- n})}{(1 - z^n w^{- n})  (1 - v^{n + 1} z^n
     w^{- n})},
  \end{equation}
  and (\ref{gamdelcommutator}) follows from our Main Theorem (Theorem~\ref{thm:eH-equals-T}).
\end{proof}

In view of our previous discussion, this means that if we define 
$V (z) = T_{\Gamma} (z) T_{\Delta} (z)$, then the operators $V (z)$, $V (w)$ are
mutually local in the generalized sense of (\ref{generalcommute-phi}) with
$\phi(z/w)$ being the right hand side of \eqref{eq:ex1-phi}.
\end{example}

\begin{example}
For our next example, we work with the operators
\[ L_+ (z) = \sum_{k = 1}^{\infty} \frac{1}{k} z^k J_k, \hspace{2em} L_- (z) =
   \sum_{k = 1}^{\infty} \frac{1}{k} z^{- k} J_{- k} . \]
The operator $L_+ (z)$ appeared in Section~\ref{LLTsection}, and the operator
$L_- (z)$ resembles the operator $L_+ (z)^{\ast}$ that we used there, except
that $z$ is replaced by $z^{- 1}$. Now we will make use of the shift operator,
and $J_0$. Define
\begin{equation}
  U_{\pm} (z) = \exp (L_{\pm} (z)), \hspace{2em} U^{\diamondsuit} (z) = U_-
  (z) U_+ (z)\,, \hspace{2em} U (z) = Q z^{J_0}
  U^{\diamondsuit} (z) .
\end{equation}
Now let us define the normal-ordered product
\[ {: \mathrel{U^{\diamondsuit} (z) U^{\diamondsuit} (w)} :} = U_- (z) U_- (w) U_+ (z)
   U_+ (w) . \]
This is meaningful for all $z$ and $w$ in the sense that if 
${\mu},
\lambda$ are given, then $\langle {\mu}| {: \mathrel{U^{\diamondsuit} (z)
U^{\diamondsuit} (w)} :} | \lambda \rangle$ is always a finite sum.

\begin{proposition}
  \label{diamondnoprop}If $| z | / | w |$ is sufficiently small, then
  \begin{equation}
    \label{udiamondno} U^{\diamondsuit} (z) U^{\diamondsuit} (w) = \prod_{j =
    0}^{n - 1} \frac{1}{1 - v^j z / w} :\mathrel{U^{\diamondsuit} (z) U^{\diamondsuit}
    (w)} : \; .
  \end{equation}
\end{proposition}

\begin{proof}
  We have
  \[ [L_+ (z), L_- (w)] = \sum_{k = 1}^{\infty} \frac{1}{k^2} k \frac{v^{nk} -
     1}{v^k - 1} \left( \frac{z}{w} \right)^k = - \sum_{j = 0}^{n - 1} \log (1
     - v^j z / w), \]
  so by the Baker-Campbell-Hausdorff formula we have
  \[ U^{\diamondsuit} (z) U^{\diamondsuit} (w) = \prod_{j = 0}^{n - 1}
     \frac{1}{1 - v^j z / w} {:\mathrel{ U^{\diamondsuit} (z) U^{\diamondsuit} (w)}:} \qedhere \]
\end{proof}

We may take (\ref{udiamondno}) as giving meaning to $U^{\diamondsuit} (z)
U^{\diamondsuit} (w)$ for all $z, w$ except at the poles of the denominator.
Then naturally $U (z) U (w)$ may be defined to be $U_0 (z) U_0 (w)
U^{\diamondsuit} (z) U^{\diamondsuit} (w)$ where $U_0 (z) = Q z^{J_0}$. (Note
that $U_0 (w)$ commutes with $U^{\diamondsuit} (z)$.)

\begin{proposition}
  \label{ulocalcr}We have
  \[ U (z) U (w) = \left( \prod_{j = 0}^{n - 1} \frac{z - v^j w}{w - v^j z}
     \right) U (w) U (z) . \]
\end{proposition}

\begin{proof}
  By Proposition~\ref{diamondnoprop},
  \begin{equation}
    \label{diamondcommute} U^{\diamondsuit} (z) U^{\diamondsuit} (w) = \left(
    \frac{w}{z} \right)^n  \prod_{j = 0}^{n - 1} \frac{z - v^j w}{w - v^j z}
    U^{\diamondsuit} (w) U^{\diamondsuit} (z) .
  \end{equation}
  On the other hand $J_0$ and $Q$ commute with $J_k$ if $k \neq 0$ while
  $[J_0, Q] = nQ$. We have
  \[ z^{J_0} Qz^{- J_0} = e^{\log (z) J_0} Qe^{- \log (z) J_0} = z^n Q, \]
  so
  \[ U_0 (z) U_0 (w) = z^n Q^2 z^{J_0} w^{J_0} = \left( \frac{z}{w} \right)^n
     U_0 (w) U_0 (z) . \]
  The statement follows.
\end{proof}

We may now discuss the effect of the factor $U_0 (z)$ in this definition. If
we had omitted it we would have had locality relation, but the factor $S (w /
z)$ would have had to include the $(w / z)^n$ that appears in
(\ref{diamondcommute}). By including the factor $U_0 (z)$ in the definition of
$U (z)$, we are able to eliminate the pole at $z = 0$. The resulting $\phi (z
/ w)$ in the locality propery (\ref{generalcommute-phi}) for $U (z)$ is then
$\prod_{j = 0}^{n - 1} (z - v^j w) / (w - v^j z)$.
\end{example}

\bibliographystyle{habbrv}
\bibliography{hamiltonians.bib}

\end{document}